\documentclass[oneside]{amsart}

\usepackage{amssymb}
\usepackage{amsmath}
\usepackage{amsfonts}
\usepackage{mathrsfs}
\usepackage{pifont}
\usepackage{amsthm}
\usepackage{hyperref}
\usepackage{amscd}
\usepackage{color}
\usepackage{enumerate}
\usepackage{tikz-cd}
\usetikzlibrary{calc}
\usepackage{subcaption}
\usepackage{caption}
\usepackage{float}
\usepackage{enumitem}
\usepackage{comment}

\usepackage{tabularx}
\usepackage{booktabs}
\usepackage{longtable}

\usepackage{makecell}
\usepackage{diagbox}
\usepackage{varioref}

\usepackage{nicematrix}
\usepackage{romannum}

\usepackage[linesnumbered,ruled,vlined]{algorithm2e}

\newcommand{\zero}{\textcolor{red}{0}}

\newcommand{\nerve}{\mathcal{N}}
\newcommand{\bch}{\star}

\newcommand{\mfg}{\mathfrak{g}}
\newcommand{\mfS}{\mathfrak{S}}
\newcommand{\mfoc}{\mathfrak{o}_\Sigma}
\newcommand{\mfocp}{\mathfrak{o}_{\Sigma'}}
\newcommand{\mfo}{\mathfrak{o}^1}
\newcommand{\mfp}{\mathfrak{o}^0}
\newcommand{\mfm}{\mathfrak{m}}

\newcommand{\Aff}{\mathbb{A}}
\newcommand{\ZZ}{\mathbb{Z}}
\newcommand{\KK}{\mathbb{K}}
\newcommand{\RR}{\mathbb{R}}
\newcommand{\PP}{\mathbb{P}}

\newcommand{\QQ}{\mathbb{Q}}
\newcommand{\NN}{\mathbb{N}}
\newcommand{\FF}{\mathbb{F}}

\newcommand{\CO}{\mathcal{O}}

\newcommand{\bw}{{\bf w}}

\newcommand{\bu}{{\bf u}}
\newcommand{\bv}{{\bf v}}

\newcommand{\T}{\mathcal{T}}

\newcommand{\F}{\mathcal{F}}

\newcommand{\mcK}{\mathcal{K}}
\newcommand{\mcL}{\mathcal{L}}

\newcommand{\mcP}{\mathcal{P}}

\newcommand{\mcU}{\mathcal{U}}

\newcommand{\A}{\mathcal{A}}
\newcommand{\B}{\mathcal{B}}
\newcommand{\D}{\mathcal{D}}

\newcommand{\V}{\mathcal{V}}

\newcommand{\rT}{\mathrm{T}}
\newcommand{\sing}{\mathrm{sing}}

\newcommand{\wt}{\widetilde}
\newcommand{\p}{\partial}

\newcommand{\rF}{\mathrm{F}}
\newcommand{\rG}{\mathrm{G}}

\newcommand{\sU}{\mathscr{U}}
\newcommand{\sV}{\mathscr{V}}

\DeclareMathOperator{\Pic}{Pic}

\DeclareMathOperator{\conv}{conv}
\DeclareMathOperator{\cone}{Cone}

\DeclareMathOperator{\spec}{Spec}

\DeclareMathOperator{\Cl}{Cl}

\DeclareMathOperator{\id}{id}
\DeclareMathOperator{\Hom}{Hom}
\DeclareMathOperator{\Def}{Def}
\DeclareMathOperator{\cDef}{Def}

\DeclareMathOperator{\Comp}{\mathbf{Comp}}
\DeclareMathOperator{\Art}{\mathbf{Art}}
\DeclareMathOperator{\Set}{\mathbf{Set}}

\DeclareMathOperator{\codim}{codim}

\DeclareMathOperator{\sgn}{sgn}
\DeclareMathOperator{\Star}{star}
\DeclareMathOperator{\Tot}{Tot}

\usepackage{theoremref}
\newtheorem{thm}[equation]{Theorem}
\newtheorem{prop}[equation]{Proposition}
\newtheorem{lemma}[equation]{Lemma}
\newtheorem{cor}[equation]{Corollary}

\theoremstyle{definition}

\newtheorem{defn}[equation]{Definition}

\newtheorem{example}[equation]{Example}
\newtheorem{rem}[equation]{Remark}

\newtheorem{convention}[equation]{Convention}

\numberwithin{equation}{subsection}

\usepackage{newfloat}
\DeclareFloatingEnvironment[
fileext=los,
listname={List of Diagrams},
name=Diagram,
placement=h,
within=none,
]{mdiagram}

\title[Locally Trivial Deformations of Toric Varieties]{Locally Trivial Deformations of Toric Varieties}

\author{Nathan Ilten}
\address{Department of Mathematics, Simon Fraser University,
	8888 University Drive, Burnaby BC V5A1S6, Canada}
\email{nilten@sfu.ca}

\author{Sharon Robins}
\address{Department of Mathematical Sciences, Carnegie Mellon University, 5000 Forbes Avenue, Pittsburgh, PA 15213
}
\email{srobins@andrew.cmu.edu}
\begin{document}
\pagenumbering{arabic}
\begin{abstract}
	We study locally trivial deformations of toric varieties from a combinatorial point of view. For any fan $\Sigma$, we construct a deformation functor $\cDef_\Sigma$ by considering \v{C}ech zero-cochains on certain simplicial complexes. We show that under appropriate hypotheses, $\cDef_\Sigma$ is isomorphic to $\Def'_{X_\Sigma}$, the functor of locally trivial deformations for the toric variety $X_\Sigma$ associated to $\Sigma$. In particular, for any complete toric variety $X$ that is smooth in codimension $2$ and $\QQ$-factorial in codimension $3$, there exists a fan $\Sigma$ such that $\cDef_\Sigma$ is isomorphic to $\Def_X$, the functor of deformations of $X$. We apply these results to give a new criterion for a smooth complete toric variety to have unobstructed deformations, and to compute formulas for higher order obstructions, generalizing a formula of Ilten and Turo for the cup product. We use the functor $\cDef_\Sigma$ to explicitly compute the deformation spaces for a number of toric varieties, and provide examples exhibiting previously unobserved phenomena. In particular, we classify exactly which toric threefolds arising as iterated $\PP^1$-bundles have unobstructed deformation space.
\end{abstract}
\maketitle

\section{Introduction}
\subsection{Background and motivation}\label{sec:background}
Let $\KK$ be an algebraically closed field of characteristic zero and $X$ a variety over $\KK$. The functor $\Def_X$ of isomorphism classes of (infinitesimal) deformations of $X$ provides useful information on how $X$ might fit into a moduli space. 
In the setting where $X$ is smooth, $\Def_X$ coincides with the functor $\Def_X'$ of isomorphism classes of locally trivial deformations of $X$.
In general, locally trivial first order deformations are described by $H^1(X,\T_X)$ and obstructions to lifting locally trivial deformations live in $H^2(X,\T_X)$. In fact, all locally trivial deformations of $X$ are controlled by the \v{C}ech complex of the tangent sheaf $\T_X$, see \S\ref{sec:approach} below.

Despite this seemingly concrete description of $\Def_X'$, it is still quite challenging to explicitly understand $\Def_X'$ for specific examples.
In this paper, we will give a purely \emph{combinatorial} description of the deformation functor $\Def_X'$ when $X$ is a $\QQ$-factorial complete toric variety. 
First order deformations of smooth complete toric varieties were described combinatorially by N.~Ilten in \cite{ilten1}. In \cite{ilten2} Ilten and R.~Vollmert showed that homogeneous first order deformations can be extended to one-parameter families over $\Aff^1$, providing a sort of skeleton of the versal deformation (see also work by A.~Mavlyutov \cite{mavlyutov1,mavlyutov2} and A.~Petracci \cite{petracci}). However, it turns out that in general there are obstructions to combining these one-parameter families. This was observed by Ilten and C.~Turo in \cite{ilten3}, which contains a combinatorial description of the cup product.

These results place the deformation theory of smooth complete toric varieties in a strange liminal space: although they can be obstructed, unlike say smooth Calabi-Yau varieties \cite{tian,todorov} or smooth Fano varieties, they do \emph{not} satisfy ``Murphy's law'' \cite{vakil}, that predicts deformation spaces will have arbitrarily bad singularities. It thus remains an interesting challenge to determine exactly what spaces can occur as the deformation space of a smooth complete toric variety.

\subsection{Main results}
We now summarize the main results of the paper. Let $X=X_\Sigma$ be a $\QQ$-factorial toric variety\footnote{Throughout the paper, all toric varieties are assumed to be normal.}  with corresponding fan $\Sigma$.  See \S\ref{sec:toricprelim} for notation and details on toric varieties.
Our starting point is the combinatorial description of $H^1(X,\T_X)$ and $H^2(X,\T_X)$ from \cite{ilten1,ilten3}.
Indeed, when $X$ is complete, for every $k\geq 1$ there are isomorphisms
  \[ H^k(X,\T_X)\cong  \bigoplus_{\substack{\rho, \bu}} \widetilde{H}^{k-1}(V_{\rho,\bu},\KK),\]
  see \thref{prop:cohom2}.
  Here, $\rho$ ranges over rays of $\Sigma$, $\bu$ ranges over characters of the torus,  and $V_{\rho,\bu}$ is a particular simplicial complex contained in $\Sigma$, see \S\ref{sec:divisor}.

By the above, locally trivial first order deformations of $X$ and obstructions to lifting locally trivial deformations can be described via cohomology of the simplicial complexes $V_{\rho,\bu}$.
Hence, it is natural to try to completely understand the functor $\Def_X'$ in terms of \v{C}ech complexes for the $V_{\rho,\bu}$.
The fan $\Sigma$ induces a closed cover $\sV_{\rho,\bu}$ of each $V_{\rho,\bu}$ and we will consider the \v{C}ech complex $\check{C}^\bullet(\sV_{\rho,\bu},\KK)$ with respect to this cover.
Alternatively, one may interpret $\check{C}^\bullet(\sV_{\rho,\bu},\KK)$ as the complex of simplicial cochains with coefficients in $\KK$ on the nerve $\nerve(\sV_{\rho,\bu})$ of the cover $\sV_{\rho,\bu}$, see Remark \ref{rem:nerve}.

For any local Artinian $\KK$-algebra $A$ with residue field $\KK$ and maximal ideal $\mfm_A$, we will define a natural map
\[
\mfoc:	\bigoplus_{\rho,\bu} \check{C}^0(\sV_{\rho,\bu},\mfm_A)\to \bigoplus_{\rho,\bu} \check{C}^1(\sV_{\rho,\bu},\mfm_A).
\]
We then set
 	\begin{align*}
		\cDef_{\Sigma}(A)&= \{\alpha \in \bigoplus_{\rho, \bu}\check{C}^0(\sV_{\rho,\bu},\mfm_A): \mfoc(\alpha)=0 \} /\sim 
	\end{align*}
	where $\sim$ is a certain equivalence relation. See Definition \ref{defn:cdef} for details. This can be made functorial in the obvious way; we call the resulting functor the \emph{combinatorial deformation functor}.

Our primary result is the following:

\begin{thm}[See \thref{thm:combiso}]\thlabel{thm:main}
Let $X=X_\Sigma$ be a $\QQ$-factorial toric variety without any torus factors and assume that $H^1(X,\CO_X)=H^2(X,\CO_X)=0$, for example, $X$ is a smooth complete toric variety. Then the functor $\Def'_X$ of locally trivial deformations of $X$ is isomorphic to the combinatorial deformation functor $\cDef_\Sigma$.
\end{thm}

\noindent By utilizing a comparison theorem, we obtain a similar combinatorial description of $\Def_X$ for complete toric varieties that are smooth in codimension two and $\QQ$-factorial in codimension three, see \thref{mildsing} for the precise statement.

Our main motivation in introducing the combinatorial deformation functor $\cDef_\Sigma$ was to be able to effectively compute the hull of $\Def_X$  when $X=X_\Sigma$ is a toric variety with sufficiently mild singularities. Using the combinatorial deformation functor, we introduce the \emph{combinatorial deformation equation} and show that one can compute a hull of $\Def_X$ by solving this equation to higher and higher order through a combinatorial process, see \S\ref{sec:cde}. We also show that in some cases, this process can be simplified further by removing certain maximal cones from the fan $\Sigma$ (\thref{thm:samehull}) and limiting those pairs of maximal cones on which we must consider obstruction terms (\thref{prop:closure} and \thref{prop:codimone}).

Given two first order deformations of $X=X_\Sigma$ over $\KK[t_1]/t_1^2$ and $\KK[t_2]/t_2^2$, the cup product computes the obstruction to combining them to a deformation over $\KK[t_1,t_2]/\langle t_1^2,t_2^2\rangle$. This cup product has been described in combinatorial terms by Ilten and Turo in \cite{ilten3}. Using the functor $\cDef_\Sigma$, we are able provide explicit combinatorial formulas for the obstructions to lifting deformations to arbitrary order, see \thref{thm:obsformula}. In particular, we easily recover the cup product formula of \cite{ilten3}. In contrast to the situation of the cup product, our higher order obstruction formulas include not only first order combinatorial deformation data, but also higher order data. 

The combinatorial deformation functor $\cDef_\Sigma$ exhibits a large amount of structure. Utilizing this structure, we provide non-trivial conditions guaranteeing that $X_\Sigma$ has unobstructed deformations:
\begin{thm}(See \thref{thm:unobstructed})
	Let $X_{\Sigma}$ be a complete toric variety that is smooth in codimension 2 and $\QQ$-factorial in codimension 3. Let $\A$ consist of all pairs $(\rho,\bu)$ of rays and characters for which $\widetilde{H}^0(V_{\rho,\bu},\KK)$ does not vanish.
	Set 
	\begin{align*}
		\B:= \left\{ (\rho,\bu+\bv)\in \Sigma(1)\times M\ \Big|
			(\rho,\bu)\in \mathcal{A};\  
			\bv\in \displaystyle\sum_{(\rho',\bu')\in \mathcal{A}} \ZZ_{\geq0}\cdot \bu' 
		\right\}.
	\end{align*}
	If ${H}^{1}(V_{\rho,\bu},\KK)=0$ for all pairs $(\rho,\bu)\in \B$, then $X_{\Sigma}$ is unobstructed.
\end{thm}	
\noindent
In Example \ref{ex:unobstructed2}, we give an example of a smooth toric threefold whose unobstructedness follows from our conditions but cannot be deduced by degree reasons alone.

Finally, we put our machinery to use to calculate the hull of $\Def_X$ for numerous examples.
It makes sense to start with examples of low Picard rank.
By analyzing $H^1(X,\T_X)$ and $H^2(X,\T_X)$ for $X$ a smooth complete toric variety of Picard rank two we show:
\begin{thm}[See \thref{thm:ranktwo}]
Let $X$ be a smooth complete toric variety of Picard rank one or two. Then $X$ has unobstructed deformations. Furthermore, $X$ is rigid unless it is the projectivization of a direct sum of line bundles on $\PP^1$ such that the largest and smallest degrees differ by at least two.
\end{thm}

The first interesting case of examples to consider is thus toric threefolds $X_\Sigma$ of Picard rank three (smooth toric surfaces are always unobstructed by \cite[Corollary 1.5]{ilten1}). We focus on the case where $\Sigma$ is a splitting fan, that is, $X_\Sigma$ is an iterated $\PP^1$-bundle. 
These $\PP^1$-bundles have the form 
\[
\PP(\CO_{\FF_e}\oplus \CO_{\FF_e}(aF+bH))
\]
for $e,a,b\in\ZZ$, $e,b\geq 0$
where $\FF_e=\PP(\CO_{\PP^1}\oplus \CO_{\PP^1}(e))$ is the $e$th Hirzebruch surface, and $F$ and $H$ respectively represent the classes in $\Pic(\FF_e)$ of the fiber and $\CO_{\mathbb{F}_e}(1)$ in the $\PP^1$-bundle fibration of $\FF_e$ over $\PP^1$. 
We discover that such threefolds may have obstruction equations whose lowest terms are quadratic or cubic, or they may be unobstructed:
\begin{thm}\thlabel{thm:obstructed}
	Let \[X= \PP(\CO_{\FF_e}\oplus \CO_{\FF_e}(aF+bH))\]
	with $e,b\geq 0$.
	Then $X$ is obstructed in exactly the following cases: 
	\begin{enumerate}[label={(\roman*)}]
		%[label=Case \arabic*:,  left=0pt, align=left]
		\item The case $e=1$, $a\leq -2$, and $b\geq 3-a$. In this case, the minimal degree of obstructions is three. \label{case:obstrcuted1}
		\item The case $e\geq 2$, $a\leq -e$, and $b \geq  1+\dfrac{2-a}{e}$. If $a\equiv 1 \mod e $, then the minimal degree of obstructions is three. In all other cases, the minimal degree of obstructions is two.
		
		\label{case:obstructed2} 
	\end{enumerate}
	
\end{thm}
By computing the hull of $\Def_X$ for a number of such examples, we find examples whose hulls exhibit the following behaviour:
\begin{enumerate}[label={(\roman*)}]
	\item The hull has a generically non-reduced component (Example \ref{example:1,-2,5}).
	\item The hull is irreducible but has a singularity  at the origin (Example \ref{example:3,-4,3}).
	\item The hull has a pair of irreducible components whose difference in dimension is arbitrarily large (Example \ref{example:e,-e,b}).
\end{enumerate}
None of these phenomena had previously been observed for deformation spaces of smooth toric varieties.

In \cite[Question 1.4]{ilten3}, it was asked if the deformation space of a smooth toric variety $X$ is cut out by quadrics. By \thref{thm:obstructed}, we see that the answer to this question is negative. In particular, any differential graded Lie algebra controlling $\Def_X$ cannot be formal.

\subsection{Our approach}\label{sec:approach}
As mentioned above, for any variety $X$ the functor $\Def_X'$ is controlled by the \v{C}ech complex
 $\check{C}^\bullet (\mcU,\T_X)$ for the tangent sheaf $\T_X$ with respect to an affine open cover $\mcU$ of $X$. Indeed, isomorphism classes of deformations of $X$ over a local Artinian $\KK$-algebra $A$ with residue field $\KK$ are given by
\[
\Def_X'(A)\cong\{x\in \check{C}^1(\mcU,\T_X)\otimes \mfm_A: x_{jk}\bch -x_{ik} \bch x_{ij} =0\}/\sim
\]
where $\mfm_A$ is the maximal ideal of $A$, $\bch$ is the Baker-Campbell-Hausdorff (BCH) product, and $\sim$ is an equivalence relation induced by an action of \v{C}ech zero-cochains, see \S\ref{sec:deffunctor} for details.

More generally, any sheaf of Lie algebras $\mcL$ on a topological space $X$ naturally gives rise to a deformation functor $\rF_\mcL$, see \thref{liefun1}.
The first step in proving \thref{thm:main} is to replace the tangent sheaf $\T_X$ by a simpler sheaf $\mcL$ of Lie algebras. For $\QQ$-factorial toric variety $X$, the generalized Euler sequence gives a surjection 
\[
	\mcL:=\bigoplus_{\rho} \CO(D_\rho)\to \T_X
\]
where the $D_\rho$ are the toric boundary divisors, see  \S\ref{seq:euler}. Moreover, the sheaf $\mcL$ has a natural bracket and the above map is a map of sheaves of Lie algebras (\thref{thm:liebracket}). This  induces a map of functors $\rF_\mcL\to \rF_{\T_X}\cong \Def_X'$ that is an isomorphism under appropriate cohomological vanishing conditions.

The Lie algebra structure on $\mcL$ may be seen as coming from the Cox torsor of $X$, see Remark \ref{rem:cox1}. In fact, this is a manifestation of the fact that, under the hypotheses of \thref{thm:main}, invariant deformations of the Cox torsor are equivalent to locally trivial deformations of $X$, see \thref{rem:cox2} for an even more general statement. This is very much inspired by the approach of J.~Christophersen and J.~Kleppe \cite{comp}.

The second step in proving \thref{thm:main} involves replacing the functor $\rF_{\mcL}$ by an equivalent functor $\widehat\rF_{\mcL}$ (\thref{functornew}) obtained by considering a quotient of the deformation functor controlled by the Thom-Whitney homotopy fiber of an inclusion of differential graded Lie algebras. See
the references in Remark \ref{rem:homotopy} for details on homotopy fibers in deformation theory.
In our specific setting of locally trivial deformations of a toric variety $X_\Sigma$, it turns out that the functor $\widehat{\rF}_{\mcL}$ is exactly the combinatorial deformation functor $\cDef_\Sigma$ from above, and our primary result follows.

\subsection{Other related literature and motivation}
K.~Altmann began a systematic study of the deformation theory of \emph{affine} toric singularities in the 1990's, giving combinatorial descriptions of first order deformations \cite{altmann1}, obstructions \cite{altmann2}, and homogeneous deformations \cite{altmann3}. For the special case of a Gorenstein toric threefold $X$ with an isolated singularity, he gave a combinatorial description of the hull of $\Def_X$ and its irreducible components \cite{altmann4}. This has recently been revisited by Altmann, A.~Constantinescu, and M.~Filip in \cite{altmann5} with similar results in a slightly more general setting using different methods. Dropping the isolated singularity assumption, Filip \cite{Fil1, Fil2} treats the deformation theory of affine Gorenstein toric pairs and relates Laurent polynomials to certain multi-parameter deformation families. Filip has also given a combinatorial description of the cup product in the affine case, see \cite{matej}. 

There are numerous motivations for studying deformations of toric varieties, both in the affine and in the global situations; we now mention several. Deformations of toric varieties are useful in mirror symmetry for studying deformations of embedded Calabi-Yau hypersurfaces \cite{mavlyutov1} and for classifying deformation families of Fano varieties \cite{coats,hilbert12}. Deformation theory of toric varieties has  been used to study singularities on the boundary of the K-moduli spaces of Fano varieties \cite{Pet2}, extremal metrics \cite{tipler}, the boundary of Gieseker moduli spaces \cite{rana}, and to show that deformations of log Calabi-Yau pairs can be obstructed \cite{FPR}.
We hope that our results here will find similar applications.

\subsection{Organization}
We now describe the organization of the remainder of this paper. Section \ref{sec:deffunctor} concerns itself with deformation functors governed by a sheaf of Lie algebras. We recall preliminaries on deformation functors (\S\ref{sec:defsetup}), Lie algebras and  Baker-Campbell-Hausdorff products (\S\ref{sec:bch}), and define the deformation functor controlled by a sheaf of Lie algebras (\S\ref{sec:deffunctornew}). Subsection \ref{sec:functor2} contains the definition of $\widehat\rF_{\mcL}$, a quotient of the functor controlled by a homotopy fiber. 

In Section \ref{sec:defeq} we describe a procedure to algorithmically construct the hull of our functor $\widehat\rF_{\mcL}$ by iteratively solving a deformation equation. The idea of constructing a hull via iterated lifting goes back at least to \cite{schlessinger} and is made more algorithmic in certain settings in \cite{stevens}, but our setting is distinct enough that we provide a thorough treatment. We set up the situation in \S\ref{sec:defeqsetup}, introduce the deformation equation in \S\ref{sec:deq}, and show in \S\ref{sec:defeqversal} that iteratively solving it produces a hull of our homotopy fiber analogue.

In Section \ref{sec:toric} we turn our attention to toric varieties. We recall preliminaries and set notation in \S\ref{sec:toricprelim}. We then discuss cohomology of the structure sheaf in \S\ref{sec:structuresheaf}, introduce the simplicial complexes $V_{\rho,\bu}$ and discuss their relation to boundary divisors in \S\ref{sec:divisor}, and introduce the Euler sequence in \S\ref{seq:euler}.

In Section \ref{sec:deftoric} we study  deformations of toric varieties. We define the combinatorial deformation functor $\cDef_\Sigma$ in \S\ref{sec:cdf} and prove our main result \thref{thm:main}. We also discuss connections to Cox torsors. In \S\ref{sec:cde} we specialize the discussion of \S\ref{sec:defeq} to the toric setting and discuss how to compute the hull of $\cDef_\Sigma$ by solving the combinatorial deformation equation. We discuss formulas for higher order obstructions in \S\ref{sec:obs}.  In \S\ref{sec:rc} we discuss how to further simplify computations by removing certain cones from the fan $\Sigma$.
We prove our criterion for unobstructedness in \S\ref{sec:unobstructed}.

In \S\ref{sec:Examples} we turn our attention to examples. In \S\ref{sec:prim} we introduce primitive collections and prove a sufficient criterion for rigidity (\thref{lemma:rigid}). We show that smooth complete toric varieties of Picard rank less than three are unobstructed in \S\ref{sec:ranktwo}. In \S\ref{sec:rank3}  we study toric threefolds $X$ that are iterated $\PP^1$-bundles, obtaining very explicit descriptions of $H^1(X,\T_X)$ and $H^2(X,\T_X)$. We continue this study in \S\ref{sec:rank3ctd}, providing several examples whose deformation spaces exhibit interesting behaviour and proving \thref{thm:obstructed}.

We conclude with two appendices. In Appendix \ref{sec:comp} we state a folklore theorem (\thref{CMiso}) comparing deformations of a scheme $X$ with deformations of an open subscheme $U$; for lack of a suitable reference we provide a proof. The theorem implies that in particular, for $X$ a Cohen-Macaulay variety that is smooth in codimension two, deformations of $X$ may be identified with deformations of the non-singular locus.
Finally, in Appendix \ref{ap:solve} we show that the deformation equation of \S\ref{sec:defeq} can in fact be iteratively solved.

For the reader who is interested in understanding the precise definition of the combinatorial deformation functor $\cDef_\Sigma$ and the statement of Theorem \ref{thm:main} as quickly as possible, we recommend reading \S\ref{sec:defsetup}, \S\ref{sec:bch}, \S\ref{sec:deffunctornew},  \S\ref{sec:toricprelim}, \S\ref{sec:divisor}, \S\ref{seq:euler}, and \S\ref{sec:cdf}.

\subsection*{Acknowledgements}
We thank A.~Petracci and F.~Meazzini for productive discussions. We further thank D.~Iacono for helping us understand the connection between $\widehat\rF_\mcL$ and Thom-Whitney homotopy fibers. Both authors were partially supported by NSERC. We thank the anonymous referee for a careful reading and helpful suggestions.

\section{Deformation functors}\label{sec:deffunctor}
\subsection{Setup}\label{sec:defsetup}
We assume that the reader is familiar with basic notions from deformation theory and functors of Artinian rings, see~e.g.~\cite{sernesi}.
We will always work over an algebraically closed field $\KK$ of characteristic zero. All tensor products are taken over $\KK$ unless otherwise specified. Let $\Comp$ be the category of complete local Noetherian $\KK$-algebras with the residue field $\KK$. For every $R\in \Comp$ we denote by $\mfm_R$ the maximal ideal of $R$. We also consider the subcategory $\Art$, which consists of local Artinian $\KK$-algebras with the residue field $\KK$. We denote by $\Set$ the category of sets.
For any deformation functor $\rF:\Art\to\Set$ \cite[Definition 3.2.5]{Lie}, we denote its \emph{tangent space} by
\[
	\rT^1 \rF:=\rF(\KK[t]/t^2).
\]

A \textit{small extension} in $\Art$ is an exact sequence  
	\begin{equation}\label{eqn:smallextension}
		0\to I \to A'\xrightarrow{\pi} A \to 0,\end{equation}
	where $\pi$ is a morphism in $\Art$ and $I$ is an ideal of $A'$ such that $\mfm_{A'}\cdot I=0$. The $A'$-module structure on $I$ induces a $\KK$-vector space structure on $I$. The above definition of a small extension is standard, but there is some inconsistency in the literature about the terminology. We opt to maintain consistency with the terminology used in \cite{Lie},\cite{Obstruction}.

Recall that a \emph{complete obstruction theory} for a deformation functor $\rF$ consists of a $\KK$-vector space $W$, called the \textit{obstruction space}, and a function $\phi$  that assigns to every small extension as in \eqref{eqn:smallextension}
	and any $\zeta\in \rF(A)$ an element $\phi(\zeta,A')\in W\otimes I$ such that $\phi(\zeta,A')=0$ if and only if $\zeta$
	lifts to $\rF(A')$. Moreover, $\phi$ satisfies a certain functoriality property, see \cite[Definition 3.6.1]{Lie}.

We will frequently make use of the \emph{standard smoothness criterion} to show that a morphism of functors is smooth.
\begin{thm}[{\cite[Theorem 3.6.5]{Lie}}] \thlabel{standardsmooth}
	Let $(\rF,W_{\rF}, \phi_{\rF})$ and $(\rG,W_{\rG},\phi_{\rG})$ be deformation functors together with associated complete obstruction theories, and let $f: \rF \to \rG$ be a morphism of functors. Assume that:
	\begin{enumerate}[label={(\roman*)}]
		\item  $\rT^1\rF \to \rT^1\rG$ is surjective; and
		\item There exists an injective map $ob_{f}: W_{\rF} \to W_{\rG}$ such that $ob_f \circ \phi_\rF=\phi_\rG\circ f$ for any small extension with $I\cong \KK$.
	\end{enumerate}
	 Then $f$ is smooth. In particular, $f$ is surjective.
\end{thm}  

\subsection{Lie algebras and the Baker-Campbell-Hausdorff product}\label{sec:bch}
Let $\mfg$ be a Lie algebra over $\KK$, not necessarily of finite dimension. By possibly embedding $\mfg$ into a universal enveloping algebras (see \cite[Theorem 9.7]{Hall}), we can always assume that the Lie bracket on $\mfg$ is given by the commutator in some associative algebra $\Lambda$. Through this embedding into an associative algebra $\Lambda$, for any $A\in \Comp$ and $x\in \mfg \otimes \mfm_A$ we define $\exp(x)$ in $\Lambda\otimes A$ using the formal power series of the exponential map.  These maps are clearly convergent in the $\mfm_A$-adic topology. 

The Baker-Campbell-Hausdorff (BCH) product $x\bch y$ of $x,y \in \mfg \otimes \mfm_{A}$ is defined by the relation
\[\exp(x)\cdot\exp(y)=\exp(x\bch y).\]
This product gives $\mfg\otimes \mfm_{\A}$ the structure of a group.
If $x$ and $y$ commute, then
\[
x\bch y=x+y.
\] 
In general, the non-linear terms of $x\bch y$ can be expressed as nested commutators of $x$ and $y$ with rational coefficients (see e.g. \cite{Hofst}). The expression in nested commutators is not unique, and we will call any expression of this form a BCH formula. The first few terms are easily calculated and well known:
\[x\bch y= x+y+\dfrac{1}{2}[x,y]+\dfrac{1}{12}[x,[x,y]]-\dfrac{1}{12}[y,[x,y]]+ (\text{Higher order terms}).\]
An explicit and complete description of a BCH formula is provided by \cite{Dynkin}

\begin{align}\label{dynkin}
	{\displaystyle x\bch y=\sum _{n=1}^{\infty }{\frac {(-1)^{n-1}}{n}}\sum _{\begin{smallmatrix}r_{1}+s_{1}>0\\\vdots \\r_{n}+s_{n}>0\end{smallmatrix}}{\frac {[x^{r_{1}}y^{s_{1}}x^{r_{2}}y^{s_{2}}\dotsm x^{r_{n}}y^{s_{n}}]}{\left(\sum _{j=1}^{n}(r_{j}+s_{j})\right)\cdot \prod _{i=1}^{n}r_{i}!s_{i}!}}.}
\end{align}
In this expression, the summation extends to all nonnegative integer values of $s_i$ and $r_i$, and the following notation is utilized:
\begin{align*}
	{\displaystyle [x^{r_{1}}y^{s_{1}}\dotsm x^{r_{n}}y^{s_{n}}]=[\overbrace {x,[x,\dotsm [x}^{r_{1}},[\overbrace {y,[y,\dotsm [y} ^{s_{1}},\,\dotsm \,[\overbrace {x,[x,\dotsm [x} ^{r_{n}},[\overbrace {y,[y,\dotsm y} ^{s_{n}}]]\dotsm ]]}
\end{align*}
with the definition $[x] := x$.

\subsection{Deformation functors controlled by a sheaf of Lie algebras}\label{sec:deffunctornew}
In this subsection, we will provide a brief review of deformation functors  controlled by a sheaf of Lie algebras.
This is a special case of the deformation functors controlled by semicosimplicial Lie algebras considered in  \cite{dgla1}, see also \cite[\S3.7]{Lie}.

Let $X$ be a topological space, and $\sU=\{U_i\}$ be an open or closed cover of $X$. For any sheaf of abelian groups $\F$ on $X$, we denote by $\check{C}^{\bullet}_{\sing}(\sU,\F)$ the \v{C}ech complex of singular cochains with respect to the cover $\sU$.\footnote{Recall that the abelian groups $\F(V)$ used in constructing the \v{C}ech complex in the case of a closed cover are defined as $\iota^{-1}(\F)(V)$, where $\iota:V\to X$ is the inclusion.
In our setting, we will only consider a closed cover when $\F$ is a constant sheaf with values in an abelian group $B$, and all intersections $V$ of sets in the cover are connected. In this case $\F(V)$ is either $0$ (if $V=\emptyset$) or $B$ (if $V\neq \emptyset$).}

It is often advantageous to consider the subcomplex $\check{C}^{\bullet}(\sU,\F) \subseteq \check{C}^{\bullet}_{\sing}(\sU,\F) $ of alternating \v{C}ech cochains. We will present our results in alternating cochains whenever possible. Let \[d:\check{C}^k(\sU,\F)\to \check{C}^{k+1}(\sU,\F)\] denote the \v{C}ech differential, let $\check{Z}^k(\sU,\F)$ denote the group of $k$-cocycles, and let $\check{H}^k(\sU,\F)$ denote $k$-th \v{C}ech cohomology group. In this section we will only consider \v{C}ech complexes for an open cover $\sU$, but in \S\ref{sec:divisor} we will also consider \v{C}ech complexes for a constant sheaf on a simplicial complex with respect to a closed cover.

For the remainder of this section, we take $\mcL$ to be a sheaf of Lie algebras on a topological space $X$, and $\sU=\{U_i\}$ to be an open cover of $X$. 
	Let $A\in \Comp$. We define a left action of the group $(\check{C}^0(\sU,\mcL)\otimes \mfm_A,\bch)$ on the set $\check{C}^1(\sU,\mcL)\otimes \mfm_A$ given by
	\begin{align*}
		\check{C}^0(\sU,\mcL)\otimes \mfm_A \times \check{C}^1(\sU,\mcL)\otimes \mfm_A &\to \check{C}^1(\sU,\mcL)\otimes \mfm_A\\
		(a,x)&\mapsto a\odot x
	\end{align*}
	via
	\[(a\odot x)_{ij}= a_i \bch x_{ij}\bch -a_j,\]
where we omit the restriction maps in the notation.	It is straightforward to verify that this action is well-defined.

\begin{defn}\thlabel{defn:omaps}
	Let $A\in \Comp$. We use the product $\bch$ to define the maps     
	\begin{align*}
		\mfp:\check{C}^0(\sU,\mcL)\otimes \mfm_{A}&\to \check{C}^1(\sU,\mcL)\otimes \mfm_{A},\\
		\mfo:\check{C}^1(\sU,\mcL)\otimes \mfm_{A}&\to \check{C}^2_{\sing}(\sU,\mcL)\otimes \mfm_{A}
	\end{align*}   
	via
	\begin{align*}
		\mfp(a)_{ij}&=-a_{i}\bch a_{j}, \\
		\mfo(x)_{ijk}&=x_{jk}\bch (-x_{ik}) \bch x_{ij}.
	\end{align*}  
\end{defn}

We remark that for every $a \in \check{C}^0(\sU,\mcL)\otimes \mfm_{A}$, we have 
	\[\mfo(\mfp(a))=0.\]  
This follows from a direct computation.
Furthermore, given a small extension as in \eqref{eqn:smallextension}, 
let $x \in \check{C}^i(\sU,\mcL)\otimes \mfm_{A'}$ for $i=0,1$. 

\begin{lemma}
For any $y\in \check{C}^i(\sU,\mcL)\otimes I$, we have \begin{equation}\label{eqn:odiff}
	\mathfrak{o}^i(x+y)=\mathfrak{o}^i(x)+d(y).
\end{equation}
\end{lemma}	
\begin{proof}
	This follows from a straightforward computation.
\end{proof}

\begin{defn}\thlabel{liefun1}
	Let $\mcL$ be a sheaf of Lie algebras on a topological space $X$, and let $\sU$ be an open cover of $X$.
	We define the functor $\rF_{\mcL,\sU}: \Art \to \Set$ on objects as follows:
	\[\rF_{\mcL,\sU}(A)= \{x \in \check{C}^1(\sU,\mcL)\otimes \mfm_A: \mfo(x)=0 \}/\sim  \]
	where $\sim$ is the equivalence relation induced by the action of $\check{C}^0(\sU,\mcL)\otimes \mfm_A$ on $\check{C}^1(\sU,\mcL)\otimes \mfm_A$.
	
	The functor sends a morphism
	$\pi': A' \to A$ to the map \[\rF_{\mcL,\sU}(\pi'): \rF_{\mcL,\sU}(A') \to \rF_{\mcL,\sU}(A)\]
	induced by the map $\check{C}^1(\sU,\mcL)\otimes \mfm_{A'}\to \check{C}^1(\sU,\mcL)\otimes \mfm_{A}$.
\end{defn}

	While we have endeavoured to describe our results in elementary terms, the proof of \thref{compa} below is most naturally done by working in the more general situation of semicosimplicial (differential graded) Lie algebras. Instead of recalling relevant notions here, we refer the reader to the excellent overview found in \cite[\S1]{Iacono3}. For a more detailed exposition, the reader may consult \cite{Lie}.

\begin{rem}\label{rem:semicosimplicial}
	After choosing an ordering on $\sU$, we may identify $\check{C}^\bullet(\sU,\mcL)$ with the ordered \v{C}ech complex. As such, it has the structure of a semicosimplicial Lie algebra, see e.g.~\cite[Example 2.6.3 and Remark 2.6.4]{Lie} for details. The functor $\rF_{\mcL,\sU}$ is exactly the functor controlled by the semicosimplicial Lie algebra $\check{C}^\bullet(\sU,\mcL)$ in the sense of \cite{dgla1}. 
	The tangent space of $\rF_{\mcL,\sU}$ may be identified with $\check{H}^1(\sU,\mcL)$ and a complete obstruction theory is given by $\check{H}^2(\sU,\mcL)$ with the map induced by $\mfo$, see \cite[Theorem 3.7.3]{Lie}.
\end{rem}
\begin{rem}\label{rem:dgla}
	To any semicosimplicial Lie algebra $\mfg$, one may associate a differential graded Lie algebra  $\Tot(\mfg)$ called its Thom-Whitney totalization  \cite[\S3]{Iacono2}. The underlying complexes of $\mfg$ and $\Tot(\mfg)$ are homotopy equivalent, and the deformation functor controlled by $\mfg$ is naturally isomorphic to the deformation functor $\Def_{\Tot(\mfg)}$  controlled by $\Tot(\mfg)$ \cite[Theorem 7.6]{Iacono2}. 
\end{rem}

\begin{thm}\thlabel{compa}
	Let $f:\mcL\to \mcK$ be a  morphism of  sheaves of Lie algebras on a topological space $X$. Let $\sU$ be an open cover of $X$, and let $\sV$ be a refinement of $\sU$.  Consider the induced morphism of \v{C}ech cohomology $f:
	\check{H}^i(\sU,\mcL) \to \check{H}^i(\sV,\mcK)$.
	\begin{enumerate}
		\item If this map is surjective for $i=1$ and injective for $i=2$,
	then the induced morphism of functors $f:\rF_{\mcL,\sU} \to \rF_{\mcK,\sV}$ is smooth.
		\item If this map is surjective for $i=0$, bijective for $i=1$, and injective for $i=2$,
	then the induced morphism of functors $f:\rF_{\mcL,\sU} \to \rF_{\mcK,\sV}$ is an isomorphism.
	\end{enumerate}
\end{thm}
\begin{proof}
	The first claim follows from \thref{standardsmooth} coupled with Remark \ref{rem:semicosimplicial}.
	For the second claim, $\rF_{\mcL,\sU}$ and $\rF_{\mcK,\sV}$ are isomorphic to the deformation functors controlled by 
	$\Tot(\check{C}^\bullet(\sU,\mcL))$ and $\Tot(\check{C}^\bullet(\sV,\mcK))$ (see Remark \ref{rem:dgla}, \cite[\S5]{dgla1}, or \cite[Corollary 7.6.6]{Lie}). Since these totalizations are homotopy equivalent to the respective \v{C}ech complexes, the claim follows from \cite[Theorem 6.6.2]{Lie}. 
	Alternatively, the second claim can be proved using an argument similar to the proof of \cite[Lemma 5.1]{comp}.
\end{proof}

As a consequence of the above theorem, we see that if $\sU$ is any open cover for which $\mcL$ is acyclic on all intersections of elements of $\sU$, $\rF_{\mcL,\sU}$ is independent of $\sU$. Indeed, in this situation, \v{C}ech cohomology coincides with sheaf cohomology. 
This is in particular the case if $X$ is a separated scheme, $\sU$ is an affine open cover, and $\mcL$ is a quasicoherent sheaf of Lie algebras.
We will thus frequently suppress $\sU$ and just write $\rF_{\mcL}$ for the functor $\rF_{\mcL,\sU}$.

The functors $\rF_{\mcL}$ are useful to study geometric deformation problems. 
As an example, let $X$ be a variety over $\KK$, and let $\Def_X$ and  $\Def'_X$ denote the functor of deformations of $X$ and locally trivial deformations of $X$, respectively. When $X$ is smooth these functors coincide. The locally trivial deformations of $X$ are governed by the tangent sheaf $\T_{X}$ and 
\[\rF_{\T_{X}} \xrightarrow{\exp} \Def'_X \]
is an isomorphism of functors, see \cite[Proposition 2.5]{BGL22}. Here $\exp$ denotes the exponential map defined in \S\ref{sec:bch}, applied to sections of $\T_X$. For each affine open $U$, the map $\exp$ yields automorphisms of $U \times \spec A$ over $\spec A$ reducing to the identity modulo $\mfm_A$. A cochain $x \in \check{C}^1(\sU,\T_X)\otimes \mfm_A$ satisfies $\mfo(x)=0$ if and only if the automorphisms $\exp(x_{ij})$ glue on overlaps to define a locally trivial deformation. The equivalence relation induced by the action of $\check{C}^0(\sU,\T_X)\otimes \mfm_A$ identifies isomorphic deformations.  We will study this situation in detail in \S\ref{sec:deftoric} when $X$ is a $\QQ$-factorial toric variety.

\subsection[A homotopy fiber quotient] 
{A homotopy fiber quotient}\label{sec:functor2}
Let $X$ be a separated scheme and $\mcL$ be a quasi-coherent sheaf on $X$ which is also a sheaf of Lie algebras over $\KK$ such that all restriction maps to non-empty open sets are injective. Fix a finite affine open cover $\sU$ of $X$ and choose a non-empty open set $V\subseteq \bigcap U_i$.
Let $\mcK$ be the pushforward of $\mcL_{|V}$ along the inclusion $V\subseteq X$; the Lie bracket on $\mcL$ induces a Lie bracket on $\mcK$.

We thus have an exact sequence of sheaves on $X$ 
\begin{equation}\label{eq:exactLie}
	\begin{tikzcd}
		0 \ar[r] & \mcL \ar[r,"\iota"] & \mcK \ar[r,"\lambda"] & \mcK/ \mcL \ar[r] &0.
	\end{tikzcd}
\end{equation}
It is worth noting that $\mcK/\mcL$ is not a sheaf of Lie algebras in general.
The exact sequence \eqref{eq:exactLie} induces a map of \v{C}ech complexes 
\begin{equation*}\label{eq:Cechcomplex}
	\begin{tikzcd}
		0 \ar[r] & \check{C}^{\bullet}(\sU,\mcL) \ar[r,"\iota"] & \check{C}^{\bullet}(\sU,\mcK) \ar[r,"\lambda"] & \check{C}^{\bullet}(\sU,\mcK/\mcL) \ar[r]   &0.
	\end{tikzcd}
\end{equation*}
We will use $d$ to refer to the differential of any one of these complexes. Note that  $\iota$ and $\lambda$ are compatible with $d$. 
Since $\iota$ is injective, we will use
the notation $\iota^{-1}$ to denote the map
\[\iota(\check{C}^{\bullet}(\sU,\mcL)) \hookrightarrow \check{C}^{\bullet}(\sU,\mcL) \]
inverse to $\iota$. 

For each affine open set $U\subseteq X$ obtained by intersecting elements of $\sU$, we fix a $\KK$-linear section 
\begin{equation}\label{eqn:s}
s: (\mcK/\mcL)(U) \to \mcK(U)
\end{equation}
to $\lambda$, that is, $\lambda \circ s: (\mcK/\mcL)(U)\to (\mcK/\mcL)(U)$ is the identity map. 
For every $A\in \Art$, we thus have the following commutative diagram:
\begin{mdiagram}[H]
	\[
	\begin{tikzcd}
		\check{C}^{0}(\sU,\mcL)\otimes \mfm_A \ar[r,hook,"\iota"] \ar[d,"\mfp" left] &   \check{C}^{0}(\sU,\mcK)\otimes \mfm_A \ar[r,twoheadrightarrow,"\lambda"{below}] \ar[d,"\mfp" left] & \check{C}^{0}(\sU,\mcK/\mcL)\otimes \mfm_A  \ar[l, bend right=20,"s" above]   \\
		\check{C}^{1}(\sU,\mcL)\otimes \mfm_A \ar[r,hook,"\iota"] \ar[d,"\mfo" left] &  \check{C}^{1}(\sU,\mcK)\otimes \mfm_A\ar[r,twoheadrightarrow,"\lambda"] \ar[d,"\mfo" left] & \check{C}^{1}(\sU,\mcK/\mcL)\otimes \mfm_A \ar[l, bend right=20,"s" above] 
		\\
		\check{C}^{2}(\sU,\mcL)\otimes \mfm_A\ar[r,hook,"\iota"] &   
		\check{C}^{2}(\sU,\mcK)\otimes \mfm_A \ar[r,twoheadrightarrow,"\lambda"]  & \check{C}^{2}(\sU,\mcK/\mcL)\otimes \mfm_A .
	\end{tikzcd}
	\]
	\addtocounter{equation}{1}
	\caption{Map between \v{C}ech complexes involving the maps $\mfp$ and $\mfo$
	}
	\label{diagram1}
\end{mdiagram}

\begin{defn}\thlabel{functornew}
We define 
the functor ${\rG}_{\mcL}:\Art \to \Set$ on objects as follows:
	\begin{align*}
		\rG_{\mcL}(A)&= \{\alpha \in \check{C}^0(\sU, \mcK/\mcL)\otimes \mfm_{A} : \lambda(\mfp(s(\alpha)))=0 \}.
	\end{align*}
Likewise, we define the functor ${\widehat\rF}_{\mcL}:\Art \to \Set$ via
	\begin{align*}
		\widehat\rF_{\mcL}(A)&= \rG_{\mcL}(A)/\sim
	\end{align*}
	where $\alpha= \beta$ in $	\widehat\rF_{\mcL}(A)$  if and only if there exists $a\in \check{C}^0(\sU,\mcL)\otimes \mfm_{A}$ such that 
	\[\iota(a)\odot \mfp(s(\alpha)) = \mfp(s(\beta)). \]
	Both functors act on morphisms in the obvious way.
\end{defn}
\noindent{Although we have suppressed it from the notation, both functors depend on the cover $\sU$, the open set $V\subseteq X$, and the collection of sections $s$.

	\begin{rem}\label{rem:homotopy}
		The functor $\widehat\rF_{\mcL}$ may be viewed as a quotient of the deformation functor $\Def_{\iota}$ controlled by the Thom-Whitney homotopy fiber of the inclusion of differential graded Lie algebras $\iota:\Tot(\check{C}^\bullet(\sU,\mcL))\hookrightarrow \Tot(\check{C}^\bullet(\sU,\mcK))$. See \cite[\S2]{liedes}, \cite[Theorem 2]{dgla2}, and \cite[cf. Example 3.9]{Iacono0}. A concrete description of the deformation functor $\Def_\iota$, which makes evident that $\widehat\rF_{\mcL}$ is a quotient, may be obtained by (twice!) applying Hinich descent as explained in \cite{Iacono3}. See also the ideas in the proofs of \cite[Theorem 4.2]{Iacono0} and \cite[Theorem 8.1.2]{Lie}.
	\end{rem}

Now we will examine the relationship between the functors $\rF_{\mcL}$ and $\widehat\rF_{\mcL}$.
Given a representative $\alpha \in \check{C}^0(\sU,\mcK/\mcL)\otimes \mfm_{A}$ of an element in $\widehat\rF_{\mcL }(A)$, there exists a unique $x\in  \check{C}^1(\sU, \mcL)\otimes \mfm_{A}$ such that $\iota(x)=\mfp(s(\alpha))$. Since, $\mfo(\mfp(s(\alpha)))=0$, by considering the commutativity of Diagram \ref{diagram1}, we can deduce that $\mfo(x)=0$. This mapping $\alpha \mapsto x=\iota^{-1}(\mfp(s(\alpha)))$ induces a map
	\[\iota^{-1}\circ \mfp\circ s: \widehat\rF_{\mcL} \to \rF_{\mcL},\]
	 since it is well-behaved under the equivalence relations used to define $\rF_{\mcL}$ and $\widehat\rF_{\mcL}$.

	\begin{thm}\thlabel{compa2}
The morphism of functors
\[\iota^{-1}\circ\mfp\circ s :\widehat\rF_{\mcL}\to \rF_{\mcL}\] 
		is an isomorphism.
	\end{thm}
	\begin{proof}
	First we will show the injectivity of $\iota^{-1}\circ\mfp\circ s:\widehat\rF_{\mcL}\to \rF_{\mcL}$. Let $A\in \Art$, and let $\alpha,\beta \in \check{C}^0(\sU,\mcK/\mcL)\otimes \mfm_{A}$ be representatives of two elements in $\widehat\rF_{\mcL}(A)$ and
		\[\iota^{-1}\circ \mfp \circ s(\alpha)= x; \quad \iota^{-1}\circ \mfp\circ s(\beta)=y.\]
	From \thref{liefun1} and \thref{functornew}, it is immediately seen that if $x\sim y$, then $\alpha\sim \beta$.

	For the surjectivity of $\iota^{-1}\circ\mfp\circ s$ we will show that the map $\rG_{\mcL}\to \rF_{\mcL}$ is smooth, hence surjective. The surjectivity of $\iota^{-1}\circ\mfp\circ s$ will follow. Since $\bch$ commutes with fiber products in $\Art$ and $s$ and $\lambda$ are $\mfm_A$-linear for $A\in\Art$, the functor $\rG_{\mcL}$ is a deformation functor (cf. \cite[Definition 3.2.5]{Lie}). By Lemma \ref{lemma:tangentspace}, its tangent space may be identified with $\check{H}^0(\sU,\mcK/\mcL)$. By Lemma \ref{lemma:obstruction}, a complete obstruction theory is given by the vector space $\check{H}^1(\sU,\mcK/\mcL)$ with the map induced by $\lambda\circ \mfp\circ s$.

	We remark that $\check{H}^1(\sU,\mcK)=0$. Indeed, as $\mcK$ is constant on the finite cover $\sU$, the claim  follows from the contractibility of a simplex. We thus have a surjection $\check{H}^0(\sU,\mcK/\mcL)\to \check{H}^1(\sU,\mcL)$ and an injection $\check{H}^1(\sU,\mcK/\mcL) \to \check{H}^2(\sU,\mcL)$. The smoothness of the map of functors now follows from \thref{standardsmooth} and Remark \ref{rem:semicosimplicial}.
	\end{proof}

In the remainder of this section, we will establish two lemmas used in the proof of \thref{compa2}.

	\begin{lemma}\label{lemma:tangentspace}
The tangent space of $\rG_\mcL$ is 
\[
	\rT^1 \rG_\mcL=\check{H}^0(\sU,\mcK/\mcL).
\]
The natural map $\check{H}^0(\sU,\mcK/\mcL)\to \rT^1 \widehat\rF_\mcL$ induces an isomorphism
\[
	\check{H}^0(\sU,\mcK/\mcL)/\lambda(\check{H}^0(\sU,\mcK))\to \rT^1 \widehat\rF_\mcL.
\]
	\end{lemma}
	\begin{proof}
Let $t$ be the coordinate on the ring of dual numbers $\KK[t]/t^2$. The condition that $\alpha \in \check{C}^0(\sU, \mcK/\mcL) \otimes t$ satisfies  $\lambda(\mfp(s(\alpha))) = 0 \mod t^2$ is equivalent to the condition that $d(\alpha)=0$, showing the first claim. 
For the second claim, $\alpha \in \check{H}^0(\sU,\mcK/\mcL)\otimes t$ is trivial in $\widehat\rF_\mcL$ if and only if there exists $a\in \check{C}^0(\sU,\mcL)\otimes t$ such that $\iota(a)\odot \mfp(s(\alpha))=0$, or equivalently, $\mfp(s(\alpha))=\mfp(\iota(a))$. Since $t^2=0$, this is equivalent to $s(\alpha)-\iota(a)$ being a cocycle. Since $\lambda(s(\alpha)-\iota(a))=\alpha$, the kernel of the map $\check{H}^0(\sU,\mcK/\mcL)\otimes t\to \rT^1 \widehat\rF_\mcL$ is exactly the image of $\check{H}^0(\sU,\mcK))\otimes t$ under $\lambda$. 
	\end{proof}

	\begin{lemma}\label{lemma:obstruction}
		Consider a small extension as in \eqref{eqn:smallextension} and any $\alpha\in \rG_\mcL(A)$. Let $\alpha'\in \check{C}^0(\sU,\mcK/\mcL)\otimes \mfm_{A'}$ be any lift of $\alpha$. Then:
	\begin{enumerate}[label={(\roman*)}]
		\item \label{prop:cocycle:fun2.1}  $\lambda(\mfp(s(\alpha')))$ is an alternating 1-cocycle in $\check{C}^1(\sU,\mcK/\mcL)\otimes I$;
		\item \label{prop:cocycle:fun2.2}  for any other lift $\xi'\in \check{C}^0(\sU,\mcK/\mcL)\otimes \mfm_{A'}$ of $\alpha$, we have
		\[\lambda(\mfp(s(\xi')))=\lambda(\mfp(s(\alpha')))+ d(\xi'-\alpha');\]
		\item \label{prop:cocycle:fun2.3}  for any $a'\in \check{C}^0(\sU,\mcL)\otimes \mfm_{A'}$, we have \[\lambda(\iota(a')\odot \mfp(s(\alpha')))= \lambda(\mfp(s(\alpha'))).\]
	\end{enumerate}
	\end{lemma}
	\begin{proof}
It is straightforward to check that $\lambda(\mfp(s(\alpha')))$ is alternating.
	 We will show that it is a cocycle as well.
	Set $\eta'=s(\lambda(\mfp(s(\alpha'))))$; this belongs to  $\check{C}^1(\sU,\mcK/\mcL)\otimes I$. Since  \[\lambda\big(\eta'-\mfp(s(\alpha'))\big)=0,\] there exists $x'\in \check{C}^1(\sU,\mcL)\otimes \mfm_{A'}$ such that $\iota(x')= \eta'-\mfp(s(\alpha'))$. 
Then
\begin{align*}
d(\eta')&=d(\eta')-\mfo(\mfp(s(\alpha')))\\&=\mfo(\eta'-\mfp(s(\alpha')))\\
&=\mfo(\iota(x'))\\&=\iota(\mfo(x'))
\end{align*}
with the second equality following from \eqref{eqn:odiff}.
Then $d(\lambda(\mfp(s(\alpha'))))=\lambda(d(\eta'))=\lambda(\iota(\mfo(x')))=0$ and claim \ref{prop:cocycle:fun2.1} follows.

Using that any two liftings of $\alpha$ differ by a cochain in $\check{C}^0(\sU,\mcK/\mcL)\otimes I$, claim \ref{prop:cocycle:fun2.2} follows from \eqref{eqn:odiff} and the fact that $d$ and $\lambda$ commute.
	
For	claim \ref{prop:cocycle:fun2.3}, let $x'$ and $\eta'$ be as above.
Then	\begin{align*}
		\iota(a'\odot x')&= 	\iota(a')\odot \iota(x')\\
		&= \iota(a')\odot (\eta'-\mfp(s(\alpha')))\\
		&= \eta'-(\iota(a')\odot \mfp(s(\alpha'))).
	\end{align*}
Applying $\lambda$, we obtain
		$\lambda(\eta')=\lambda(\iota(a')\odot \mfp(s(\alpha')))$ and 
the claim follows.
\end{proof}

\section{The deformation equation}\label{sec:defeq}
\subsection{Setup}\label{sec:defeqsetup}
Let $X$ be a separated scheme and $\mcL$ be a quasi-coherent sheaf on $X$ which is also a sheaf of Lie algebras over $\KK$ such that all restriction maps are injective. Fix a finite affine open cover $\sU$ of $X$ and choose a non-empty open set $V\subseteq \bigcap U_i$. The goal of this section is to explicitly construct the hull of the deformation functors $\rF_{\mcL}$ and its isomorphic avatar $\widehat\rF_{\mcL}$ (see Definitions \ref{liefun1} and \thref{functornew}) under suitable hypotheses.

 According to Schlessinger's theorem (see, for example, \cite[Theorem 3.5.10]{Lie}),   $\widehat\rF_{\mcL}$ has a hull $R\in \Comp$ if 
 \[\rT^1 \widehat\rF_{\mcL}\cong \check{H}^0(\sU,\mcK/\mcL)/\lambda(\check{H}^0(\sU,\mcK))\] 
is finite-dimensional. In other words, there exists a smooth morphism of functors $f:\Hom(R,-)\to \widehat\rF_{\mcL}$ such that the induced map on tangent spaces is bijective. By adapting the ideas from \cite{stevens}, we define the deformation equation and use it to explicitly construct the hull of $\widehat\rF_{\mcL}$, see \thref{hull}.

To construct the deformation equation for $\widehat\rF_{\mcL }$, we need a finite-dimensional obstruction space. Hence, we assume that $\check{H}^1(\sU,\mcK/\mcL)$ is finite-dimensional, but not necessarily that $\rT^1 \widehat\rF_{\mcL }$ is finite-dimensional. We first fix the following data:
\begin{enumerate}
	\item Elements $\theta_{\ell}\in \check{Z}^0(\sU,\mcK/\mcL)$, $\ell=1,\ldots,p$;
	\item Elements $\omega_\ell\in \check{Z}^1(\sU,\mcK/\mcL)$, $\ell=1,\ldots,q$ whose images in $\check{H}^1(\sU,\mcK/\mcL)$ form a basis.
\end{enumerate}
We are working here with the \v{C}ech complex of alternating cochains $\check{C}^k(\sU,\mcK/\mcL)$ with respect to the open cover $\sU$.

Let $S=\KK[[t_1,\ldots,t_p]]$ with maximal ideal $\mfm=\langle t_1,\ldots,t_p\rangle$. We let $\mfm_k$ denote the $k$th graded piece of $\mfm$, which is a $\KK$-vector space with basis the monomials in $t_1,\ldots,t_p$ of degree $k$.  We let $\mfm_{\leq k}$ (respectively $(\mfm^2)_{\leq k}$) denote the direct sum of the graded pieces of degree at most $k$ of $\mfm$ (respectively $\mfm^2$).
It will be useful to fix a \textit{graded local monomial order} on $S$ (see  e.g.~\cite[Definition 1.2.4]{GP}). 
For any ideal $I\subseteq S$, the \emph{standard monomials} of $I$ are those the monomials of $S$ that are not leading monomials for any element of $I$. Assuming that $I$ contains a power of $\mfm$, the \emph{normal form} with respect to $I$ of any $f\in S$ is the unique element $\overline{f}$ such that $f\equiv \overline{f}\mod I$ and $\overline{f}$ only contains standard monomials.

\subsection{The deformation equation}\label{sec:deq}
For each $r\geq 1$, we inductively construct $\alpha^{(r)}\in \check{C}^0(\sU,\mcK/\mcL)\otimes \mfm_{\leq r}$ and \textit{obstruction polynomials} $g_1^{(r)},\ldots, g_q^{(r)} \in (\mfm^2)_{\leq r}$  such that   
\begin{equation} \label{eq:defeq1}
	\lambda(\mfp(s(\alpha^{(r)})))\equiv 0\quad \mod J_{r}, 
\end{equation}
where $J_r=\langle g_1^{(r)},\ldots,g_q^{(r)}\rangle+\mfm^{r+1}$.
We begin by setting
\[
\alpha^{(1)}= \sum_{\ell=1}^{p}t_{\ell}\cdot \theta_{\ell} \quad \mathrm{and} \quad g_1^{(1)}=\ldots= g_q^{(1)}=0. 
\]	
Since $\alpha^{(1)}\in \check{Z}^0(\sU,\mcK/\mcL)\otimes \mfm_{\leq 1}$, it is immediate to see that
\[
0=d(\alpha^{(1)}) \equiv \lambda(\mfp(s(\alpha^{(1)}))) \qquad \mod J_{1}.
\]
In practice, it is enough to solve the \textit{deformation equation} 
\begin{equation}\label{eq:defeq2}
	\lambda(\mfp(s(\alpha^{(r)})))-\sum_{\ell=1}^q g_{\ell}^{(r)}\cdot\omega_\ell
	\equiv d(\beta^{(r+1)})+\sum_{\ell=1}^q \gamma_\ell^{(r+1)} \cdot \omega_{\ell}\qquad \mod  \mfm\cdot  J_{r}
\end{equation}
for
	\[\beta^{(r+1)} \in \check{C}^0(\sU,\mcK/\mcL)\otimes \mfm_{r+1};\qquad
	\gamma_\ell^{(r+1)}\in \mfm_{r+1}.
\]
	We then set 
\[\alpha^{(r+1)}=\alpha^{(r)}-\beta^{(r+1)};\qquad g_\ell^{(r+1)}=g_\ell^{(r)}+\gamma_\ell^{(r+1)}.\]

In Proposition \ref{prop:defeqsolving} we show that
\begin{equation}\label{eq:defeq}
	\lambda(\mfp(s(\alpha^{(r+1)})))\equiv \sum_{\ell=1}^q g_\ell^{(r+1)}\cdot \omega_{\ell}\qquad \mod \mfm\cdot J_{r}.
\end{equation}
In particular, the desired equation  \eqref{eq:defeq1} modulo $J_{r+1}$ follows from \eqref{eq:defeq} together with the observation that by construction, $\mfm\cdot J_r\subseteq J_{r+1}$.

As a convention, we set $J_0=\mfm$.
\begin{prop}\thlabel{prop:defeqsolving}
	Suppose that $\check{H}^1(\sU,\mcK/\mcL)$ is finite dimensional. Then:
	\begin{enumerate}[label={(\roman*)}]
		\item Given a solution $\alpha^{(r)},g_\ell^{(r)}$ of \eqref{eq:defeq}  modulo $\mfm\cdot J_{r-1}$ with $J_r+\mfm^r = J_{r-1}+\mfm^r$, there is a solution of \eqref{eq:defeq2} modulo $\mfm\cdot J_r$. 
			\label{prop:defeqsolve1}
		\item Given any solution to \eqref{eq:defeq2} modulo $\mfm\cdot J_r$ the resulting $\alpha^{(r+1)}$ and $g_\ell^{(r+1)}$ satisfy \eqref{eq:defeq} modulo $\mfm\cdot J_{r}$ and $J_{r+1} + \mfm^{r+1} = J_{r}+\mfm^{r+1}$. \label{prop:defeqsolve2}
	\end{enumerate}
\end{prop}
\noindent We defer the proof of \thref{prop:defeqsolving} to Appendix \ref{ap:solve}.

\begin{rem}\thlabel{rem:sm}
	By passing to normal forms with respect to $\mfm\cdot J_r$ in \eqref{eq:defeq2}, we may assume that $\beta^{(r+1)}$ and $\gamma_\ell^{(r+1)}$ only involve standard monomials of $\mfm\cdot J_r$. By the inductive construction of $J_r$ and the fact that our monomial order is graded, this means that we may assume that the $g_\ell^{(r+1)}$ only involve standard monomials of $\mfm\cdot J_r$. 
\end{rem}
\begin{rem}\thlabel{rem:relevant}
	In situations where the cohomology groups $\check{H}^k(\sU,\mcK/\mcL)$ are multigraded, we may often obtain greater control over which monomials can appear in the obstruction polynomials $g_\ell^{(r)}$. This allows us to simplify computations. Given standard monomials $t^w$ for $\mfm\cdot J_{r-1}$ and $t^{w'}$ for $\mfm\cdot J_r$, we say that $t^w$ is \emph{relevant} for $t^{w'}$ if there exists a monomial $t^{w''}$ with $t^w$ as a factor such that the monomial $t^{w'}$ has non-zero coefficient in the normal form of $t^{w''}$ with respect to $\mfm\cdot J_r$.

	This condition may be used as follows, assuming we are only using standard monomials as in \thref{rem:sm}: if $t^w$ is \emph{not} relevant for $t^{w'}$, then the coefficient of $t^w$ in $\alpha^{(r)}$ has no effect on the coefficients of $t^{w'}$ in the possible solutions $\beta^{(r+1)}$ and $\gamma_\ell^{(r+1)}$ of \eqref{eq:defeq2}.

	Indeed, to solve \eqref{eq:defeq2}, we must consider the normal form of 
	\[\lambda(\mfp(s(\alpha^{(r)})))-\sum_{\ell=1}^q g_{\ell}^{(r)}\omega_\ell\] with respect to $\mfm\cdot J_r$. By \thref{rem:sm} the term
$\sum_{\ell=1}^q g_{\ell}^{(r)}\omega_\ell$ is already in normal form. Furthermore, by the BCH formula (see \S\ref{sec:bch}) it follows that the coefficient of $t^w$ in $\alpha^{(r)}$ will only affect the coefficients in $\lambda(\mfp(s(\alpha^{(r)})))$ of those monomials $t^{w''}$ with $t^w$ as a factor; passing to the normal form only affects coefficients of monomials for which $t^{w}$ is relevant.
\end{rem}

\subsection{Versality}\label{sec:defeqversal}
Using our solutions to the deformation equation \eqref{eq:defeq}, we will construct the hull of $\widehat\rF_{\mcL }$. Let $g_{\ell}$
be the projective limit of $g_{\ell}^{(r)}$ in $S$, and let $\alpha$ be the projective limit of $\alpha^{(r)}$ in $\check{C}^0(\sU,\mcK/\mcL)\otimes S$. We define
\[J= \langle g_{1},\ldots,g_{q} \rangle, \quad R:=S/J \quad \mathrm{and} \quad R_r:=S/J_r.\]
Moreover, we have
\[
J + \mfm^{r+1} = J_r + \mfm^{r+1}.
\] 

The pair $(\alpha, R)$
defines a map of functors of Artinian rings \[f: \Hom(R,-)\to \widehat\rF_{\mcL }.\] Indeed, every $\zeta \in \Hom(R,A)$ factors through a morphism $\zeta_r : R_r \to A$ for $r\gg 0$ and then $f(\zeta_r)= \zeta_r(\alpha^{(r)})$ 
defines $f$.

\begin{thm}\thlabel{hull}
	Assume that the images of $\theta_1,\ldots,\theta_p$ in $$\check{H}^0(\sU,\mcK/\mcL)/\lambda(\check{H}^0(\sU,\mcK))$$ are a basis. Then, $f:\Hom(R,-)\to \widehat\rF_{\mcL }$ is a hull, that is, $f$ is smooth and induces an isomorphism on tangent spaces.
\end{thm}
\begin{proof}
Since $J\subseteq \mfm^2$, we have $\rT^1\Hom(R,-)\cong (\mfm_R/\mfm_R^2)^*\cong (\mfm/\mfm^2)^{*}$ (see, for example, \cite[Example 3.5.3]{Lie}). The induced map on tangent spaces is given by
	\begin{align*}
		f:(\mfm/\mfm^2)^{*} &\to \rT^1 \widehat\rF_{\mcL }=\widehat\rF_{\mcL }(\KK[t]/t^2)\\
		t_{\ell}^{*}&\mapsto  \theta_{\ell}\otimes t.
	\end{align*}
	This map is an isomorphism, since the ${\theta_{\ell}\otimes t}$ form a basis for $\rT^1 \widehat\rF_{\mcL }$ by Lemma \ref{lemma:tangentspace}. 
	
	By \thref{lemma:obfinj} (see Appendix \ref{ap:solve}), we have  that $ob_f$ is an injective obstruction map for $f$. Thus by \thref{standardsmooth} it follows that $f$ is smooth.
\end{proof}

\section{Toric varieties}\label{sec:toric}
\subsection{Preliminaries}\label{sec:toricprelim}
In this section, we will fix certain notation and recall relevant facts about toric varieties. We refer the reader to \cite{cls} for more details on toric geometry.

We will consider a lattice $M$ with the dual lattice $N=\text{Hom}(M,\mathbb{Z})$ and associated vector space $N_{\mathbb{R}}=N\otimes_{\mathbb{Z}}\mathbb{R}$.
Given a fan $\Sigma$ in $N_{\RR}$, there is an associated toric variety $X_{\Sigma}$ with an action by the torus $T_{N}=\spec\KK[M]$ (see \cite[\S 3.1]{cls}). The variety $X_{\Sigma}$ has a $T_{N}$-invariant open affine cover $\sU=\{U_{\sigma}\}_{\sigma\in \Sigma_{\max}}$, where $\Sigma_{\text{max}}$ is the set of maximal cones in $\Sigma$. Here, each   $U_{\sigma}$ is defined as \[U_{\sigma}= \text{Spec~} \KK[\sigma^{\vee}\cap M]; \quad \sigma^{\vee}=\{\bu\in M\otimes\RR: v(\bu)\geq0~~ \text{for~all~} v\in \sigma \}.\]

We will denote the regular function on $T_{N}$ associated with $\bu\in M$ by $\chi^\bu$.
We denote the set rays of $\Sigma$ by $\Sigma(1)$. Given a ray $\rho \in \Sigma(1)$, the  primitive lattice generator of $\rho$ is denoted by $n_{\rho}$ and the evaluation of $n_{\rho}$ at $\bu\in M$ is denoted by $\rho(\bu)$. 
Recall that the \emph{support} of the fan $\Sigma$ is $|\Sigma|=\displaystyle\bigcup_{\sigma\in \Sigma} \sigma$. 

Important aspects of the geometry of $X_\Sigma$ can be seen from the combinatorics of the fan $\Sigma$. In particular:
\begin{enumerate}
	\item The variety $X_{\Sigma}$ is smooth if and only if $\Sigma$ is smooth, that is every cone $\sigma \in \Sigma$ is generated by part of a basis of $N$ \cite[Proposition 4.2.7]{cls};
	\item $X_\Sigma$ is $\QQ$-factorial if and only if $\Sigma$ is simplicial \cite[Proposition 4.2.7]{cls};
	\item $X_\Sigma$ is complete if and only if $|\Sigma|=N_\RR$ \cite[Theorem 3.19]{cls}.
\end{enumerate}

The variety $X_{\Sigma}$ has a \emph{torus factor} if it is equivariantly isomorphic to the product of a nontrivial torus and a toric variety of smaller dimension. We will frequently assume that the toric variety does not have a torus factor; this is equivalent to the condition  that $N_{\RR}$ is spanned by the primitive ray generators $n_\rho$ for  $\rho\in \Sigma(1)$ \cite[Corollary 3.3.10]{cls}.

\subsection{Cohomology of the structure sheaf}\label{sec:structuresheaf}
Let $\Sigma$ be a fan in $N_{\RR}$ and $X_{\Sigma}$ be the associated toric variety. The cohomology groups of the structure sheaf of $X_{\Sigma}$ naturally have an $M$-grading, and each graded piece can be understood combinatorially using certain subsets of $N_{\RR}$. For every $\bu \in M$, we associate the subset of $N_{\RR}$ given by
\begin{equation*}
		V_{\bu}:=\bigcup_{\sigma\in \Sigma} \conv\{n_{\rho}: \rho(\bu)<0\}_{\rho \in  \Sigma(1)\cap \sigma}\subseteq N_{\RR}.
\end{equation*}

The following result establishes a relation between the sheaf cohomology of the structure sheaf $\CO_{X_{\Sigma}}$ and the reduced singular cohomology of $V_{\bu}$.
\begin{prop}[cf. {\cite[Theorem 9.1.3]{cls}}] \thlabel{van1}
	For $\bu\in M$ and $k\geq 0$, we have
	\[
	H^k(X,\CO_{X_\Sigma})_{\bu}\cong \widetilde{H}^{k-1}(V_{\bu},\KK).
	\]
\end{prop}
When the set $V_{\bu}$ is contractible to a point, its reduced singular cohomology vanishes. In fact, we will make use of the following vanishing result:
\begin{prop}[cf. {\cite[Theorem 9.2.3]{cls}}]\thlabel{prop:O}
	Suppose that $|\Sigma|$ is convex. Then for all $k>0$,
	\[
	H^k(X_\Sigma,\CO_{X_\Sigma})=0.
	\]
\end{prop}

\subsection{Cohomology of torus invariant divisors}\label{sec:divisor}
Torus invariant divisors on $X_\Sigma$ and the cohomology groups of the associated reflexive sheaves can be understood combinatorially as well; we focus here on the prime torus invariant divisors which are in bijection with the rays of $\Sigma$ \cite[\S 4.1]{cls}. We denote the divisor corresponding to $\rho\in\Sigma(1)$ by $D_{\rho}$.

Similar to above, the torus action on $X_{\Sigma}$ induces a natural $M$-grading on the sections of the sheaf $\CO(D_{\rho})$. The homogeneous part of $H^0(U_{\sigma},\CO(D_{\rho}))$ of degree $\bu$ is denoted by $H^0(U_{\sigma},\CO(D_{\rho}))_{\bu}$. By \cite[Equation 9.1.2]{cls}, we have
$H^0(U_{\sigma},\CO(D_{\rho}))_{\bu}$ is nonzero if and only if for all $\rho'\in \Sigma(1)\cap \sigma$
\begin{align}\label{eq:torussection}
	\rho'(\bu)\geq 
	\begin{cases}
	0\quad &\rho'\not=\rho,\\
	-1 \quad &\rho'=\rho.
	\end{cases}
\end{align}
In this case, $H^0(U_{\sigma},\CO(D_{\rho}))_{\bu}$ is a one-dimensional $\KK$-vector space with basis $\chi^{\bu}$.

For $\rho\in \Sigma(1)$, and $\bu\in M$,  we define the subset
\[V_{\rho,\bu}:= \bigcup _{\sigma \in \Sigma}
\conv\left\{ n_{\rho'}\ \Big|\begin{array}{c@{\ \textrm{if}\ }c}
	\rho'(\bu)<0 & \rho'\neq \rho\\
	\rho'(\bu)<-1 & \rho'=\rho
\end{array}
\right\}_{\rho'\in\Sigma(1)\cap\sigma}\subseteq N_{\RR}.\]
When $\Sigma$ is simplicial, this is a (topological realization of a) simplicial complex.
If $\rho(\bu)\not=-1$, it is immediate to see that $V_{\rho,\bu}=V_{\bu}$. For every $\sigma\in \Sigma$, there is a canonical exact sequence
\begin{equation}\label{canonicalseq}
	0 \rightarrow H^0(U_{\sigma},\CO(D_{\rho}))_\bu \xrightarrow{\iota} \KK  \xrightarrow{\lambda} H^0(V_{\rho,\bu}\cap\sigma,\KK) \rightarrow 0,
\end{equation}
see \cite[Equation 9.1.10]{cls}.
Since $\lambda$ is either an isomorphism or the zero map, there is a unique $\KK$-linear section $s: H^0(V_{\rho,\bu}\cap\sigma,\KK)\to \KK$ of $\lambda$.

There is a natural closed cover 
\[
\sV_{\rho,\bu}=\{V_{\rho,\bu}\cap\sigma\}_{\sigma\in\Sigma_{\max}}
\]
of the set $V_{\rho,\bu}$ indexed by elements of $\Sigma_{\max}$, with all of its intersections being contractible. 
The above exact sequence of vector spaces leads to a short exact
sequence of alternating \v{C}ech complexes with respect to the covers $\sU,\Sigma_{\max},\V_{\rho,\bu}$ (cf. \cite[p. 403]{cls},\cite[p. 8]{ilten3}):
\begin{equation}\label{exact}
	\begin{tikzcd}
		\displaystyle\bigoplus_{\substack{\rho\in\Sigma(1)\\\bu\in M}}\check{C}^{\bullet}(\sU,\CO(D_\rho))_\bu  \ar[r,hook,"\iota"]  & 
		\displaystyle\bigoplus_{\substack{\rho\in\Sigma(1)\\\bu\in M}}\check{C}^{\bullet}(\Sigma_{\max},\KK) \ar[r,twoheadrightarrow,"\lambda"{below}] & 
		\displaystyle\bigoplus_{\substack{\rho\in\Sigma(1)\\\bu\in M}}\check{C}^{\bullet}(\sV_{\rho,\bu},\KK)  \ar[l, bend right=20,"s"]. 
	\end{tikzcd}
\end{equation}
We will use $d$ to refer to the differential of any one of these complexes. Note that although $\iota$ and $\lambda$ are compatible with $d$, the section $s$ is not.
\begin{convention}\label{convention}
	We will frequently use the notation $\chi^\bu$ and $f_\rho$ to distinguish between the individual direct summands appearing in the terms of \eqref{exact}, for example, a basis of $\bigoplus_{\substack{\rho\in\Sigma(1)\\\bu\in M}}\check{C}^k(\Sigma_{\max},\KK)$ is given by $\{\chi^\bu\cdot f_\rho\}_{\rho,\bu}$.
\end{convention}

\begin{rem}\label{rem:nerve}
	Recall that the \emph{nerve} of a cover $\sV=\{V_i\}_{i\in I}$ is the abstract simplicial complex whose faces (or simplices) are those subsets $J\subseteq I$ such that $\bigcap_{j\in J} V_j\neq \emptyset$. We will denote the nerve of a cover $\sV$ by $\nerve(\sV)$, and its topological realization by $|\nerve(\sV)|$.

	Suppose that the cover $\sV$ has the property that the intersection of any of its sets is either empty or is connected; this is true for the closed covers $\Sigma_{\max}$ of $|\Sigma|$ and $\sV_{\rho,\bu}$ of $V_{\rho,\bu}$. Then for a constant sheaf $\F$ with values in an abelian group $B$, the alternating \v{C}ech complex	$\check{C}^\bullet(\sV,\F)$ may be identified with the complex $C^\bullet(\nerve(\sV),B)$ of simplicial cochains on $\nerve(\sV)$ with coefficients in $B$.

	In our setting, $B$ will either be $\KK$, or the maximal ideal $\mfm_A$ of an Artinian ring $A$.
In particular, the reader may interpret the complexes
\[
	\check{C}^\bullet(\sV_{\rho,\bu},\KK)
\]
appearing in \eqref{exact} either as \v{C}ech complexes for the cover $\sV_{\rho,\bu}$, or as simplicial cochain complexes on $\nerve(\sV_{\rho,\bu})$.
\end{rem}

For the middle and right \v{C}ech complexes in \eqref{exact}, there exists a canonical isomorphism between \v{C}ech cohomology and singular cohomology.\footnote{Indeed, in this setting there is an isomorphism between \v{C}ech cohomology and sheaf cohomology \cite[II.5.2]{godement}, and then an isomorphism between sheaf cohomology and singular cohomology \cite[III.1]{bredon}.
Alternatively, in light of Remark \ref{rem:nerve}, there is a homotopy equivalence between $|\nerve(\sV_{\rho,\bu})|$ and $V_{\rho,\bu}$ \cite[Theorem 3.9]{nerve}, and singular and simplicial cohomology agree for triangulable spaces.} In other words, we have 
\[H^k(\check{C}^{\bullet}(\Sigma_{\max},\KK))\cong H^k(|\Sigma|,\KK); \quad H^k(\check{C}^{\bullet}(\sV_{\rho,\bu},\KK))\cong H^k(V_{\rho,\bu},\KK) .\]
Since $H^0(|\Sigma|,\KK)=\KK$ and $H^k(|\Sigma|,\KK)=0$ for $k\geq1$, the long exact sequence of cohomology implies that the boundary map induces an isomorphism
\begin{equation}\label{iso1}
	\wt{H}^{k-1}(V_{\rho,\bu},\KK)\cong H^{k}(X_{\Sigma},\CO(D_{\rho}))_{\bu}
\end{equation}
for $k\geq1$,
see \cite[Theorem 9.1.3]{cls}. 

\begin{cor}\thlabel{cor:vanishdegree}
	Suppose that $H^k(X_{\Sigma},\CO_{X_{\Sigma}})_{\bu}=0$. If $\rho(\bu)\not=-1$, then
	\[\wt{H}^{k-1}(V_{\rho,\bu},\KK)=0.\]
\end{cor}
\begin{proof}
	If $\rho(\bu)\not=-1$, then $V_{\rho,\bu}=V_{\bu}$. Then the claim follows from \thref{van1}
\end{proof}

\begin{rem}\thlabel{exactseqlocal}
	Since $T_{N}\subseteq U_{\sigma}$ for every $\sigma$, we have the injection 
	\[\iota: T_{N} \hookrightarrow \bigcap U_{\sigma}.\]
	At this stage, we have not yet endowed $\CO(D_{\rho})$ with the structure of a sheaf of Lie algebras. Nonetheless, as in \eqref{eq:exactLie} at the start of \S\ref{sec:functor2}
	we obtain an exact sequence of quasi-coherent sheaves
	\begin{equation*}
		\begin{tikzcd}
			0 \ar[r] & \CO(D_{\rho})  \ar[r,"\iota"] & \iota_{*}(\CO(D_{\rho})_{|T_{N}}) \ar[r,"\lambda"] &  \iota_{*}(\CO(D_{\rho})_{|T_{N}})\big/ \CO(D_{\rho}) \ar[r] &0.
		\end{tikzcd}
	\end{equation*}
	Additionally, we observe that
	\begin{align*}
		\CO(D_{\rho})(U_{\sigma}) &= \bigoplus_{\bu \in M} H^0(U_{\sigma},\CO(D_{\rho}))_{\bu},\\
		\iota_{*}(\CO(D_{\rho})_{|T_{N}})(U_{\sigma})& = \CO(D_{\rho})(T_{N})
		= \bigoplus_{\bu\in M} \KK.
	\end{align*}
	Comparing with \eqref{canonicalseq},
	we can view the quotient sheaf as follows:
	\[ \iota_{*}(\CO(D_{\rho})_{|T_{N}})/ 	\CO(D_{\rho})(U_{\sigma}) = \bigoplus_{\bu \in M} H^0(V_{\rho,\bu},\KK).\]
	Subsequently, the middle and right \v{C}ech complexes in \eqref{exact} can also be expressed as 
	\begin{align*}
		\displaystyle\bigoplus_{\substack{\rho\in\Sigma(1)\\\bu\in M}}\check{C}^k(\Sigma_{\max},\KK)&=
		\bigoplus_{\rho\in \Sigma(1)} \check{C}^k(\sU,\iota_{*}(\CO(D_{\rho})_{|T_{N}})), \\
		\displaystyle\bigoplus_{\substack{\rho\in\Sigma(1)\\\bu\in M}}\check{C}^k(\sV_{\rho,\bu},\KK) &= \bigoplus_{\rho\in \Sigma(1)} \check{C}^k(\sU,\iota_{*}(\CO(D_{\rho})_{|T_{N}})\big/ 	\CO(D_{\rho}) ).
		\end{align*}
	
\end{rem}

\begin{figure}[htbp]
	\begin{tikzpicture}[scale=0.9]
		%% Define coordinates
		\coordinate (1) at (0,1);
		\coordinate (2) at (-1,1);
		\coordinate (3) at (-1,0);
		\coordinate (4) at (-1,-1);
		\coordinate (5) at (0,-1);
		\coordinate (6) at (1,0);
		\coordinate (0) at (0,0);
		\coordinate (8) at (0,2);
		\coordinate (9) at (0,-2);
		\coordinate (10) at (2,0);
		\coordinate (11) at (4,0);
		\coordinate (12) at (3,0);
		
		\coordinate (A) at (5,0);
		\coordinate (B) at (0,3);
		\coordinate (C) at (-3,3);
		\coordinate (D) at (-3,0);
		\coordinate (E) at (-3,-3);
		\coordinate (F) at (0,-3);

		\draw (1)--(2)--(3)--(4)--(5)--(6)--cycle;
		
		\draw (0)--(A);
		\draw (0)--(B);
		\draw (0)--(C);
		\draw (0)--(D);
		\draw (0)--(E);
		\draw (0)--(F);
		
		\draw (8)--(2);
		\draw (8)--(6);
		\draw (8)--(10);
		\draw (8)--(11);
		\draw (8)--(12);

		\draw (9)--(4);
		\draw (9)--(6);
		\draw (9)--(10);
		\draw (9)--(11);
		\draw (9)--(12);

		\draw[fill=black] (0) circle (1pt);
		\node[anchor=south west,font=\tiny] at (0) {0};
		\draw[fill=black] (1) circle (1pt);
		\node[anchor=south west,font=\tiny] at (1) {1};
		\draw[fill=black] (2) circle (1pt);
		\node[anchor=north east,font=\tiny] at (2) {2};
		\draw[fill=black] (3) circle (1pt);
		\node[anchor=south east,font=\tiny] at (3) {3};
		\draw[fill=black] (4) circle (1pt);
		\node[anchor=south east,font=\tiny] at (4) {4};
		\draw[fill=black] (5) circle (1pt);
		\node[anchor=north west,font=\tiny] at (5) {5};
		\draw[fill=black] (6) circle (1pt);
		\node[anchor=south west,font=\tiny] at (6) {6};
		
		\draw[fill=black] (8) circle (1pt);
		\node[anchor=south west,font=\tiny] at (8) {8};
		\draw[fill=black] (9) circle (1pt);
		\node[anchor=north west,font=\tiny] at (9) {9};

		\draw[fill=black] (10) circle (1pt);
		\node[anchor=south west,font=\tiny] at (10) {10};
		\draw[fill=black] (11) circle (1pt);
		\node[anchor=south west,font=\tiny] at (11) {11};
		\draw[fill=black] (12) circle (1pt);
		\node[anchor=south west,font=\tiny] at (12) {12};

	\end{tikzpicture}
	\caption{Representation of fan in Example \ref{ex:unobstructed1} as an abstract simplicial complex with $\rho_7$ as a vertex at $\infty$ (not to scale).}
	\label{fig:fan2}
	
\end{figure}

\begin{figure}[htbp]
	\centering
	
	\begin{subfigure}{0.45\textwidth}
		\centering
		\begin{tikzpicture}[scale=0.9]
			\coordinate (3) at (0,0);
			\coordinate (2) at (1,0);
			\coordinate (4) at (-1,0);

			\draw[dashed] (2)--($(2)+(1,0)$);
			\draw[dashed] (2)--($(2)+(1,1)$);
			\draw[dashed] (2)--($(2)+(0,1)$);
			\draw[dashed] (2)--($(2)+(0,-1)$);
			\draw[dashed] (2)--(3);
			
			\draw[dashed] (3)--($(3)+(0,1)$);
			\draw[dashed] (3)--($(3)+(0,-1)$);
			
			\draw[dashed] (4)--($(4)+(-1,0)$);
			\draw[dashed] (4)--($(4)+(-1,1)$);
			\draw[dashed] (4)--($(4)+(0,1)$);
			\draw[dashed] (4)--($(4)+(0,-1)$);
			\draw[dashed] (4)--(3);

			\draw[fill=red] (2) circle (2pt);
			\node[anchor=south east,font=\tiny] at (2) {2};
			\draw[fill=black] (3) circle (1pt);
			\node[anchor=south west,font=\tiny] at (3) {3};
			\draw[fill=red] (4) circle (2pt);
			\node[anchor=south west,font=\tiny] at (4) {4};
		\end{tikzpicture}
		\caption{$\rho=\rho_3,\bu=(1,0,0)$}
	\end{subfigure}\hfill
	\begin{subfigure}{0.45\textwidth}
		\centering
		\begin{tikzpicture}[scale=0.9]
			\coordinate (7) at (0,0);
			\coordinate (11) at (1,0);
			\coordinate (12) at (2,0);

			\draw[dashed] (7)--(11);
			\draw[dashed] (7)--($(7)+(0,1)$);
			\draw[dashed] (7)--($(7)+(0,-1)$);
			\draw[dashed] (7)--($(7)+(-1,1)$);
			\draw[dashed] (7)--($(7)+(-1,0)$);
			\draw[dashed] (7)--($(7)+(-1,-1)$);

			\draw[dashed] (11)--(12);
			\draw[dashed] (11)--($(11)+(0,1)$);
			\draw[dashed] (11)--($(11)+(0,-1)$);		
			
			\draw[dashed] (12)--($(12)+(0,1)$);
			\draw[dashed] (12)--($(12)+(0,-1)$);
			\draw[dashed] (12)--($(12)+(1,0)$);

			\draw[fill=red] (7) circle (2pt);
			\node[anchor=south west,font=\tiny] at (7) {7};
			\draw[fill=black] (11) circle (1pt);
			\node[anchor=south west,font=\tiny] at (11) {11};
			\draw[fill=red] (12) circle (2pt);
			\node[anchor=south west,font=\tiny] at (12) {12};
		\end{tikzpicture}
		\caption{$\rho=\rho_{11}, \bu=(0,0,1)$}
	\end{subfigure}
	
	\par\medskip
	
	\begin{subfigure}{0.45\textwidth}
		\centering
		\begin{tikzpicture}[scale=0.9]
			\coordinate (6) at (-1/2,0);
			\coordinate (10) at (0,0);
			\coordinate (12) at (1,0);
			\coordinate (0) at (-1,0);
			\coordinate (1) at (-1,1);
			\coordinate (5) at (-1,-1);

			\draw[dashed] (12)--($(12)+(1,0)$);
			\draw[dashed] (12)--($(12)+(0,2)$);
			\draw[dashed] (12)--($(12)+(0,-2)$);
			
			\draw[dashed] (10)--($(10)+(0,2)$);
			\draw[dashed] (10)--($(10)+(0,-2)$);
			\draw[dashed] (10)--(12);
			\draw[dashed] (10)--(6);
			
			\draw[dashed] (6)--($(6)+(0,2)$);
			\draw[dashed] (6)--($(6)+(0,-2)$);
			
			\filldraw[fill=red!20, draw=red] (1)--(6)--(5)--cycle;
			\draw[red] (6)--(1);
			\draw[red] (6)--(0);
			\draw[red] (6)--(5);
			
			\draw[red] (0)--(1);
			\draw[red] (0)--(5);
			\draw[dashed] (0)--($(0)+(-1,1)$);
			\draw[dashed] (0)--($(0)+(-1,-1)$);
			\draw[dashed] (0)--($(0)+(-1,0)$);

			\draw[dashed] (1)--($(1)+(0,1)$);
			\draw[dashed] (1)--($(1)+(-1,1)$);
			\draw[dashed] (5)--($(5)+(-1,-1)$);
			\draw[dashed] (5)--($(5)+(0,-1)$);

			\draw[fill=red] (6) circle (2pt);
			\node[anchor=south west,font=\tiny] at (6) {6};
			\draw[fill=black] (10) circle (1pt);
			\node[anchor=south west,font=\tiny] at (10) {10};
			\draw[fill=red] (12) circle (2pt);
			\node[anchor=south west,font=\tiny] at (12) {12};
			\draw[fill=red] (0) circle (2pt);
			\node[anchor=south west,font=\tiny] at (0) {0};
			\draw[fill=red] (1) circle (2pt);
			\node[anchor=south west,font=\tiny] at (1) {1};
			\draw[fill=red] (5) circle (2pt);
			\node[anchor=north west,font=\tiny] at (5) {5};

		\end{tikzpicture}
		\caption{$\rho=\rho_{10},\bu=(-1,0,-1)$}
	\end{subfigure}\hfill
	\begin{subfigure}{0.45\textwidth}
		\centering
		\begin{tikzpicture}[scale=0.9]
			\coordinate (10) at (0,0);
			\coordinate (12) at (0,-1/2);
			\coordinate (11) at (0,-1);
			\coordinate (6) at (0,1);

			\draw[dashed] (11)--($(11)+(1,0)$);
			\draw[dashed] (11)--($(11)+(-1,0)$);
			\draw[dashed] (11)--($(11)+(0,-1)$);
			\draw[red] (11)--(12);
			\draw[dashed] (12)--($(12)+(1,0)$);
			\draw[dashed] (12)--($(12)+(-1,0)$);
			\draw[dashed] (12)--(10);
			\draw[dashed] (10)--($(10)+(1,0)$);
			\draw[dashed] (10)--($(10)+(-1,0)$);
			\draw[dashed] (10)--(6);
			
			\draw[dashed] (6)--($(6)+(1,0)$);
			\draw[dashed] (6)--($(6)+(-1,0)$);
			
			\draw[dashed] (6)--($(6)+(1,1)$);
			\draw[dashed] (6)--($(6)+(0,1)$);
			\draw[dashed] (6)--($(6)+(-1,1)$);
			
			\draw[fill=red] (6) circle (2pt);
			\node[anchor=north west,font=\tiny] at (6) {6};
			\draw[fill=black] (10) circle (1pt);
			\node[anchor=south west,font=\tiny] at (10) {10};
			\draw[fill=red] (12) circle (2pt);
			\node[anchor=south west,font=\tiny] at (12) {12};
			\draw[fill=red] (11) circle (2pt);
			\node[anchor=south west,font=\tiny] at (11) {11};

		\end{tikzpicture}
		\caption{$\rho=\rho_{10},\bu=(-1,0,0)$}
	\end{subfigure}
	
	\par\medskip
	
	\begin{subfigure}{0.45\textwidth}
		\centering
		\begin{tikzpicture}[scale=0.9]
			
			\coordinate (0) at (0,0);
			\coordinate (1) at (0,1);
			\coordinate (2) at (-1,1);
			\coordinate (3) at (-1,0);
			\coordinate (4) at (-1,-1);
			\coordinate (5) at (0,-1);
			\coordinate (6) at (1,0);
			
			\draw[red] (1)--(2)--(3)--(4)--(5)--(6)--cycle;
			
			\draw[dashed] (0)--(6);
			\draw[dashed] (0)--(1);
			\draw[dashed] (0)--(2);
			\draw[dashed] (0)--(3);
			\draw[dashed] (0)--(4);
			\draw[dashed] (0)--(5);
			\draw[dashed] (6)--($(6)+(0,1)$);
			\draw[dashed] (6)--($(6)+(0,-1)$);
			\draw[dashed] (6)--($(6)+(1,0)$);
			\draw[dashed] (1)--($(1)+(0,1)$);
			\draw[dashed] (2)--($(2)+(-1,1)$);
			\draw[dashed] (2)--($(2)+(0,1)$);

			\draw[dashed] (3)--($(3)+(-1,0)$);
			\draw[dashed] (4)--($(4)+(-1,-1)$);
			\draw[dashed] (4)--($(4)+(0,-1)$);
			\draw[dashed] (5)--($(5)+(0,-1)$);

			\draw[fill=black] (0) circle (1pt);
			\node[anchor=south west,font=\tiny] at (0) {$0$};
			\draw[fill=red] (1) circle (2pt);
			\node[anchor=south west,font=\tiny] at (1) {$1$};
			\draw[fill=red] (2) circle (2pt);
			\node[anchor=north east,font=\tiny] at (2) {$2$};
			\draw[fill=red] (3) circle (2pt);
			\node[anchor=south east,font=\tiny] at (3) {$3$};
			\draw[fill=red] (4) circle (2pt);
			\node[anchor=south east,font=\tiny] at (4) {$4$};
			\draw[fill=red] (5) circle (2pt);
			\node[anchor=north west,font=\tiny] at (5) {$5$};
			\draw[fill=red] (6) circle (2pt);
			\node[anchor=south west,font=\tiny] at (6) {$6$};
			
		\end{tikzpicture}
		\caption{$\rho=\rho_0, \bu=(0,0,-1)$}

	\end{subfigure}
	
	\caption{Intersections of $\Sigma$ with, and projections of $V_{\rho,\bu}$ onto, $\langle -,\bu  \rangle=-1$ in Example \ref{ex:unobstructed1}}
	\label{fig:hyp}

\end{figure}

\begin{example}\label{ex:unobstructed1}
	We consider the toric threefold $X_\Sigma$ whose fan $\Sigma$ in $\RR^3$ may be described as follows.
The generators of its rays are given by the columns of the following matrix:
	\NiceMatrixOptions{code-for-first-row = \color{darkgray}}
	\setcounter{MaxMatrixCols}{20}
	\begin{equation*}
		\begin{pNiceMatrix}[first-row,first-col,margin]
			&\rho_0&\rho_1&\rho_{2}&\rho_3&\rho_4&\rho_{5}&\rho_{6}&\rho_7& \rho_8&\rho_9&\rho_{10}&\rho_{11}&\rho_{12}\\	
			&0&0&-1&-1&-1& 0&1& 0&0& 0&1& 1& 2 \\
			&0&1& 1& 0&-1&-1&0& 0&1&-1&0& 0& 0 \\
			&1&1& 1& 1& 1& 1&1&-1&0& 0&0&-1&-1 
		\end{pNiceMatrix}.
	\end{equation*}
A set of rays forms a cone in $\Sigma$ if the corresponding set of vertices belong to a common simplex in Figure \vref{fig:fan2}, where the ray $\rho_7$ corresponds to the point at infinity. Taking the lattice $N=\ZZ^3$, it is straightforward to verify that $\Sigma$ is smooth and complete.
	
It can be shown that $H^1(X_{\Sigma},\CO(D_\rho))_\bu\cong \widetilde H^0(V_{\rho,\bu},\KK)$ is non-zero only for the $(\rho,\bu)$ pairs
\[(\rho_3,(1,0,0)),(\rho_{11},(0,0,1)),(\rho_{10},(-1,0,-1)),(\rho_{10},(-1,0,0)).\]
Similarly, $H^2(X_{\Sigma},\CO(D_\rho))_\bu$ is zero except for the $(\rho,\bu)$ pair $(\rho_{0},(0,0,-1))$. Although in principle there are infinitely many ray-degree pairs that needs to be checked, there are actually only finitely many different simplicial complexes that can occur, and each case can be verified individually. 

For the cases with non-vanishing cohomology, Figure \vref{fig:hyp} shows the intersections of $\Sigma$ with, and projections of the simplicial complexes $V_{\rho,\bu}$ onto, the hyperplane $\langle -,\bu  \rangle=-1$.
The first four simplicial complexes have two connected components, hence have $\dim \widetilde H^0(V_{\rho,\bu},\KK)=1$, while the final simplical complex has a single cycle, hence $\dim \widetilde H^1(V_{\rho,\bu},\KK)=1$. We will see in Example \ref{ex:unobstructed2} that the toric threefold $X_\Sigma$ is unobstructed. 
\end{example}

\subsection{Euler sequence}\label{seq:euler}
Let $X_{\Sigma}$ be a $\QQ$-factorial toric variety with no torus factors. To understand locally trivial deformations of $X_\Sigma$ we will need control of its tangent sheaf. There is an exact sequence of sheaves 
\begin{equation*}
	\begin{tikzcd}
		0 \ar[r]&\Hom_{\ZZ}(\Cl(X_{\Sigma}),\ZZ)\otimes_{\ZZ} \CO_X \ar[r]&	\bigoplus \limits_{\rho \in \Sigma(1)} \CO(D_{\rho})  \ar[r,"\eta"] & \T_{X_{\Sigma}} \ar[r] &0
	\end{tikzcd}
\end{equation*}
called the Euler sequence, see  \cite[Theorem 8.1.6]{cls} (and dualize). By \cite[Equation (5)]{ilten3}, the image of the local section $\chi^\bu$ of $\CO(D_{\rho})$ is given by the derivation $\p(\rho,\bu)$ defined by \[\p(\rho,\bu)(\chi^\bv)=\rho(\bv)\chi^{\bv+\bu}.\]
When $\Cl(X_{\Sigma})$ is trivial, the map $\eta$  is an isomorphism.

\begin{prop}[cf. {\cite[Corollary 3.9]{jac}}]\thlabel{prop:isofromEulereq}
Let $X_\Sigma$ be a $\QQ$-factorial toric variety with no torus factors. Then for $k\geq0$ \[\eta: \bigoplus_{\rho \in \Sigma(1)}H^k(X_{\Sigma},\CO(D_{\rho}))\rightarrow H^k(X_{\Sigma},\T_{X_{\Sigma}})\]
is:
	\begin{enumerate}[label={(\roman*)}]
		\item  injective if $H^{k}(X_{\Sigma},\CO_{X_{\Sigma}})=0$ and $k\geq1$;
		\item  surjective if $H^{k+1}(X_{\Sigma},\CO_{X_{\Sigma}})=0$ and $k\geq0$.
		
	\end{enumerate}
\end{prop}
\begin{proof}
	The claims follow by applying the assumptions of cohomology vanishing to the long exact sequence induced by the Euler sequence, see also \cite{jac},\cite[Lemma 2.5]{ilten3}.
\end{proof}
Combining $\eta$ with \eqref{iso1} we obtain the following:
\begin{prop}[cf. {\cite[Proposition 1.4]{ilten1}, \cite[Proposition 3.1]{ilten3}}]\thlabel{prop:cohom2}
	Let $X_\Sigma$ be a $\QQ$-factorial toric variety with no torus factors. Then for $k\geq1$ \[ \bigoplus_{\substack{\rho\in\Sigma(1), \bu\in M \\ \rho(\bu)=-1}} \widetilde{H}^{k-1}(V_{\rho,\bu},\KK) \rightarrow H^k(X_{\Sigma},\T_{X_{\Sigma}})\]
	is:
	\begin{enumerate}[label={(\roman*)}]
		\item  injective if $H^{k}(X_{\Sigma},\CO_{X_{\Sigma}})=0$;
		\item surjective if $H^{k+1}(X_{\Sigma},\CO_{X_{\Sigma}})=0$;
		
	\end{enumerate}
\end{prop}
\begin{proof}
	The proof follows from combining the results of \thref{prop:isofromEulereq}, Equation \eqref{iso1}, and \thref{cor:vanishdegree}.
\end{proof}

\section{Deformations of toric varieties}
\label{sec:deftoric}

\subsection{The combinatorial deformation functor}\label{sec:cdf}
Let $X_{\Sigma}$ be a $\QQ$-factorial toric variety with no torus factors. In this section, we define the combinatorial deformation functor and show that under appropriate hypotheses it is isomorphic to $\Def'_{X_{\Sigma}}$, the functor of locally trivial deformations of $X_{\Sigma}$. 

\begin{defn}\label{bilinear}
	We define a bilinear map 
	\[[-,-]:\bigoplus_{\rho\in \Sigma(1)}\CO(D_{\rho})|_{T_N}\times \bigoplus_{\rho\in \Sigma(1)}\CO(D_{\rho})|_{T_N} \to \bigoplus_{\rho\in \Sigma(1)}\CO(D_{\rho})|_{T_N} \]
	by setting
	\[[\chi^\bu\cdot f_\rho,\chi^{\bu'}\cdot f_{\rho'}]:=\rho(\bu')\chi^{\bu+\bu'}\cdot f_{\rho'}-\rho'(\bu)\chi^{\bu+\bu'}\cdot f_\rho\]
	and extending linearly.
	Here $f_\rho$ and $f_{\rho'}$ denote the canonical sections of $\CO(D_\rho)|_{T_N}$ and $\CO(D_{\rho'})|_{T_N}$ respectively, corresponding to the rational function $1\in \KK(X_{\Sigma})$. 
\end{defn}
\noindent	It is straightforward to check that $[-,-]$ is alternating and satisfies the Jacobi identity. Hence, it is a Lie bracket.

\begin{thm}\thlabel{thm:liebracket}
	Let $X_{\Sigma}$ be a $\QQ$-factorial toric variety without any torus factors. Then the bracket of Definition \ref{bilinear} extends to a Lie bracket on $\bigoplus_{\rho\in \Sigma(1)} \CO(D_\rho)$ such that the map
	\[\eta: \bigoplus_{\rho\in \Sigma(1)} \CO(D_\rho) \to \T_{X_{\Sigma}} \] of the Euler sequence (see \S \ref{seq:euler}) is a map of sheaves of $M$-graded Lie algebras.
\end{thm}
\begin{proof}
	Suppose that $\chi^{\bu} \in H^0(U_{\sigma},\CO(D_{\rho}))$ and $\chi^{\bu'}\in H^0(U_{\tau},\CO(D_{\rho'}))$. By \eqref{eq:torussection}, for all $\rho''\in \Sigma(1) \cap \sigma \cap \tau$ we obtain the following system of inequalities:
	\begin{equation*}
		\begin{aligned}
			&\rho'(\bu)\geq0, \quad \rho'(\bu')\geq-1\\
			&\rho''(\bu)\geq0, \quad \rho''(\bu')
			\geq0 \quad \text{if $\rho''\not=\rho,\rho'$} \\
			& \rho(\bu)\geq -1, \quad \rho(\bu')\geq0.
		\end{aligned}
	\end{equation*}
	This means that for all $\rho''\in \Sigma(1) \cap \sigma \cap \tau$,
	\[
	\rho''(\bu+\bu')\geq 0
	\begin{cases}
		\text{if $\rho''\not=\rho, \rho'$,}\\
		\text{if $\rho''=\rho$ and $\rho(\bu')\not=0$ }\\
		\text{if $\rho''=\rho'$ and $\rho'(\bu)\not=0$ }.
	\end{cases}
	\]
	Thus, using \eqref{eq:torussection} again we observe that
	\[\rho'(\bu)\chi^{\bu+\bu'}\in H^0(U_{\sigma}\cap U_{\tau}, \CO(D_{\rho})), \quad \rho(\bu')\chi^{\bu+\bu'}\in H^0(U_{\sigma}\cap U_{\tau}, \CO(D_{\rho'})). \]
	Therefore, the bracket of Definition \ref{bilinear} extends to a Lie bracket on $\bigoplus_{\rho} \CO(D_\rho)$. 
	
	It is straightforward to verify that
	\[[\p(\rho,\bu), \p(\rho',\bu')]=\rho(\bu')\p(\rho',\bu+\bu')-\rho'(\bu)\p(\rho,\bu+\bu').\]
	It follows that the Lie bracket on $\bigoplus_{\rho}\CO(D_{\rho})$ is compatible with $\eta$ and the Lie bracket on $\T_{X_\Sigma}$, hence $\eta$ is a map of sheaves of Lie algebras.
\end{proof}

Combining with \thref{exactseqlocal} and setting $\mcL=\bigoplus_\rho \CO(D_\rho)$, the above result allows us to consider the functor $\widehat\rF_\mcL$ of Definition \ref{functornew} in the setting of toric varieties. We make this explicit here.

\begin{defn}\label{defn:cdef}
	 Let $X_{\Sigma}$ be a $\QQ$-factorial toric variety.
	 We define $\mfoc$ to be the composition $\lambda \circ \mfp \circ s$ with $\lambda,s$ as in \S\ref{sec:divisor} and $\mfp$ defined using Definition \ref{bilinear}.
	 The \emph{combinatorial deformation functor} $\cDef_{\Sigma}:\Art \to \Set$ is defined on objects as follows:
	\begin{align*}
		\cDef_{\Sigma}(A)&= \{\alpha \in \bigoplus_{\rho, \bu}\check{C}^0(\sV_{\rho,\bu},\mfm_A) : \mfoc(\alpha)=0 \} /\sim 
	\end{align*}
	where $\alpha=\beta$ in $\cDef_{\Sigma}(A)$  if and only if there exists $\gamma\in \bigoplus_{\rho,\bu}\check{C}^0(\sU,\CO(D_{\rho}))\otimes \mfm_{A}$ such that 
	\[\iota(\gamma)\odot \mfp(s(\alpha)) = \mfp(s(\beta)). \]
	It is defined on morphisms in the obvious way.
\end{defn}

\noindent Recall from Remark \ref{rem:nerve} that 
we may interpret $\check{C}^i(\sV_{\rho,\bu},\mfm_A)$ for $i=0,1$ as simplicial $0$- or $1$-cochains on the nerve $\nerve(\sV_{\rho,\bu})$ with coefficients in $\mfm_A$. 

The following is our primary result:

\begin{thm}\thlabel{thm:combiso}
	Let $X_{\Sigma}$ be a $\QQ$-factorial toric variety without any torus factors.
	Suppose that $H^{1}(X_{\Sigma},\CO_{X_{\Sigma}})=
H^{2}(X_{\Sigma},\CO_{X_{\Sigma}})
=0$. Then, $\cDef_{\Sigma}$ is isomorphic to $\Def'_{X_{\Sigma}}$. In fact, we have the following isomorphisms of deformation functors:
\[ \cDef_{\Sigma} \overset{\iota^{-1}\circ\mfp \circ s}{\cong} \rF_{\bigoplus_{\rho}\CO(D_{\rho})} \overset{\eta}{\cong} \rF_{\T_{X_{\Sigma}}} \overset{\exp}{\cong} \Def'_{X_{\Sigma}}.\]  
\end{thm}
\begin{proof}
	According to \cite[Proposition 2.5]{BGL22}, we have  $\rF_{\T_{X_{\Sigma}}}\overset{\exp}{\cong} \Def'_{X_{\Sigma}}$. By Theorem \ref{thm:liebracket}, $\eta$ induces  a morphism of \v{C}ech complexes of sheaves of Lie algebras:
	\[\eta:\bigoplus_{\rho\in \Sigma(1)} \check{C}^{\bullet}(\sU,\CO(D_{\rho})) \to \check{C}^{\bullet}(\sU,\T_{X_{\Sigma}}).\]
	Using the vanishing of $H^{k}(X_{\Sigma},\CO_{X_{\Sigma}})$ for $k=1$ and $2$ and \thref{prop:isofromEulereq}, we have that $\bigoplus H^k(X_{\Sigma}, \CO(D_{\rho})) \to H^k(X_{\Sigma}, \T_{X_{\Sigma}})$ is surjective for $k=0$, bijective for $k=1$ and injective for $k=2$.
	Hence, Theorem \ref{compa} implies that  
	\[\rF_{\bigoplus_{\rho}\CO(D_{\rho})}\overset{\eta}{\cong} \rF_{\T_{X_{\Sigma}}}.\]
	
	Combining \thref{thm:liebracket} and \thref{exactseqlocal}, $\cDef_{\Sigma}$ may be identified with the functor 
	$\widehat\rF_{\mcL}$ 
	(Definition \ref{functornew}) where $\mcL=\bigoplus_{\rho} \CO(D_\rho)$, $\sU=\{U_\sigma\}_{\sigma\in \Sigma_{\max}}$, $V=T_N$, and $s$ is as in \S\ref{sec:divisor}. By \thref{compa2}  we obtain 
	\[\cDef_{\Sigma} \overset{\iota^{-1}\circ \mfp \circ s}{\cong}\rF_{\bigoplus_{\rho}\CO(D_{\rho})} ,\]
	completing the proof.
\end{proof}

We are especially interested in understanding all deformations, not just locally trivial ones. The following corollary allows us to do this when the singularities of $X_\Sigma$ are mild enough.

\begin{cor}\thlabel{mildsing}
	Let $X_\Sigma$ be a complete toric variety of dimension at least $3$. Assume that $X_\Sigma$ is smooth in codimension $2$, and $\QQ$-factorial in codimension $3$. Let $\widehat \Sigma$ be any simplicial subfan of $\Sigma$ containing all three-dimensional cones of $\Sigma$. Then $\Def_{X_\Sigma}$ is isomorphic to $\cDef_{\widehat\Sigma}$.
\end{cor}

\begin{proof}
	First, we observe that any toric variety associated with a fan is Cohen-Macaulay, see \cite[Theorem 9.2.9]{cls}. Since $\widehat \Sigma$ is a subfan of $\Sigma$, the resulting toric variety $X_{\widehat \Sigma}$ is an open subset of $X_{\Sigma}$. By the construction of $\widehat \Sigma$, the inequality $\codim(X_{\Sigma}\setminus X_{\widehat \Sigma})\geq 4$ follows from the orbit-cone correspondence (\cite[Theorem 3.2.6]{cls}. This allow us to utilize  \thref{CMiso} (see Appendix \ref{sec:comp}), thereby obtaining the isomorphism
	  \[\Def_{X_\Sigma} \cong \Def_{X_{\widehat\Sigma}}.\]
	  
	  From the preceding discussion, our focus now shifts to $\Def_{X_{\widehat\Sigma}}$. Observe that $X_{\widehat\Sigma}$ is smooth in codimension 2 and $\QQ$-factorial. By \cite[Theorem 11.4.8]{cls}, $X_{\hat \Sigma}$ has abelian finite quotient singularities. Since $X_{\Sigma}$ is smooth in codimension 2, these singularities are isolated. By \cite[\S3b]{rigid} (cf.~also \cite[\S 5.1]{altmann1}), the deformations of $X_{\widehat\Sigma}$ are locally trivial, leading us to conclude that 
	\[
	\Def_{X_{\widehat\Sigma}} \cong \Def'_{X_{\widehat\Sigma}}.
	\]

	We will now prove that \[\cDef_{\widehat\Sigma} \cong \Def'_{X_{\widehat\Sigma}}\]
	by showing that $X_{\widehat\Sigma}$ satisfies all the assumptions in \thref{thm:combiso}; this will complete the proof. For $\bu \in M$, recall the sets $V_{\bu}$ defined in \S\ref{sec:structuresheaf}, and denote the corresponding set for $\widehat \Sigma$ by $\widehat V_{\bu}$.  Let $V_{\bu}^{(2)}$ denote the 2-skeleton of $V_{\bu}$. Concretely,
	\begin{align*}
			\widehat V_{\bu}:&=\bigcup_{\sigma\in \widehat \Sigma} \conv\{n_{\rho}: \rho(\bu)<0\}_{\rho \in  \widehat\Sigma(1)\cap \sigma}\subseteq N_{\RR},\\
			V_{\bu}^{(2)}:&=\bigcup_{\sigma\in \Sigma(3)} \conv\{n_{\rho}: \rho(\bu)<0\}_{\rho \in  \Sigma(1)\cap \sigma}\subseteq N_{\RR},
	\end{align*}
	where $\Sigma(3)$ denotes the $3$-dimensional cones of $\Sigma$.
	It is straightforward to see that the singular cohomology groups $\wt H^{k-1}(V_{\bu},\KK)$ for $k=1,2$ depend solely on $V_{\bu}^{(2)}$. Moreover, since $\widehat \Sigma$ contains all three-dimensional cones of $\Sigma$, we have
	\[V_{\bu}^{(2)}= \widehat V_{\bu}^{(2)}.\]

	 It follows from the preceding discussion combined with \thref{van1} that for $k=1,2$ and every $\bu \in M$, 
	 \begin{equation}\label{eq:iso}
	 	H^{k}(X_{ \Sigma},\CO_{X_{\Sigma}})_{\bu}\cong \wt H^{k-1}(V_{\bu},\KK)= H^{k-1}(V^{(2)}_{\bu},\KK)\cong H^{k}(X_{\widehat \Sigma},\CO_{X_{\widehat\Sigma}})_{\bu}. 
 \end{equation}	
	Moreover, since $\Sigma$ is a complete fan and hence $|\Sigma|$ is convex, by \thref{prop:O}, we have $H^{k}(X_{ \Sigma},\CO_{X_{\Sigma}})=0$ for $k\geq 1$. Therefore, by \eqref{eq:iso} we obtain that
	\footnote{Alternatively, one can obtain this cohomological vanishing via the long exact sequence for cohomology with support in $Z = X_{\Sigma}\setminus X_{\widehat\Sigma}$, since $\codim Z \ge 4$.}
	\[ H^{1}(X_{\widehat \Sigma},\CO_{X_{\widehat\Sigma}})=  H^{2}(X_{\widehat \Sigma},\CO_{X_{\widehat\Sigma}})=0.\]
	Additionally, as $\Sigma$ is complete,  $X_{\Sigma}$ does not have any torus factors. Since $\Sigma(1)=\widehat\Sigma(1)$, the variety $X_{\widehat\Sigma}$ also does not have any torus factors.
	
	By the above discussion, $X_{\widehat\Sigma}$ satisfies all the assumptions in \thref{thm:combiso}. Therefore, we conclude that
	$\cDef_{\widehat\Sigma}$ and  $\Def'_{\widehat X_{\Sigma}}$ are isomorphic.\end{proof}

\begin{rem}[Comparison with the Cox torsor I]\label{rem:cox1}
Let $X=X_\Sigma$ be a $\QQ$-factorial toric variety with no torus factors.
	The Lie bracket on $\bigoplus_{\rho\in \Sigma(1)} \CO(D_\rho)$ may be interpreted as coming from the Lie bracket on the tangent sheaf of the affine space $\Aff^{\#\Sigma(1)}$ associated to the Cox ring of $X_\Sigma$, as we now briefly explain. 
	The variety $X=X_\Sigma$ arises as a geometric quotient $\pi:\widetilde X\to X$ of an open subset $\widetilde X$ of $\Aff^{\#\Sigma(1)}$ under the action of the quasitorus $G=\Hom(\Cl(X), \KK^{*})$, see \cite[Theorem 5.1.11]{cls}. We will call $\widetilde X$ the \emph{Cox torsor} of $X$ (although it is actually only a torsor when $X$ is smooth).\footnote{In e.g. \cite{cox}, $\widetilde X$ is called the \emph{characteristic space} of $X$.}

	The variety $\widetilde X$ is itself toric, given by a fan $\widetilde \Sigma$ whose cones are in dimension-preserving bijection with the cones of $\Sigma$ (cf. \cite[Proposition 5.1.9]{cls}). In particular, for a cone $\sigma\in\Sigma$, denote by $\widetilde \sigma\in \widetilde\Sigma$ the corresponding cone. Since the affine space has trivial class group, the generalized Euler exact sequence for the Cox torsor gives a torus-equivariant isomorphism
	\[\widetilde\eta: \bigoplus \limits_{\widetilde{\rho} \in \widetilde{\Sigma}(1)} \CO(D_{\widetilde{\rho}})  \rightarrow \T_{X_{\widetilde{\Sigma}}}\]
	inducing a $G$-equivariant Lie bracket on $\bigoplus \limits_{\widetilde{\rho} \in \widetilde{\Sigma}(1)} \CO(D_{\widetilde\rho})$.  

Moreover, for any torus-invariant open subset $U_\sigma\subseteq X$, $\pi$ induces an isomorphism
	\[ \bigoplus \limits_{\rho \in \Sigma(1)} \CO(D_{\rho})(U_{\sigma}) \overset{\pi^{*}}{\cong} \Big(\bigoplus \limits_{\widetilde{\rho} \in \widetilde{\Sigma}(1)}\CO(D_{\widetilde{\rho}})^G(U_{\widetilde \sigma})\Big). \]
The Lie bracket on  
$\bigoplus_{\rho\in \Sigma(1)} \CO(D_\rho)$
from Definition \ref{bilinear} is exactly obtained by applying this isomorphism to the above bracket on $\bigoplus \limits_{\widetilde{\rho} \in \widetilde{\Sigma}(1)} 
\CO(D_{\widetilde\rho})$. 
\end{rem}

\begin{rem}[Comparison with the Cox torsor II]\thlabel{rem:cox2}
	The discussion of Remark \ref{rem:cox1} can be used to show that there is an isomorphism between $\rF_{\bigoplus \CO(D_{\rho})}$  and the functor $\Def_{\widetilde X}^G$ of $G$-invariant deformations of the Cox torsor $\widetilde X$.
	 Indeed, the sheaf $\T_{\widetilde X}^G\cong \bigoplus_{\widetilde \rho\in\widetilde\Sigma(1)} \CO(D_{\widetilde \rho})^G$ controls $G$-invariant deformations of $\widetilde X$ (since $\widetilde X$ is smooth). By the above we have an isomorphism of \v{C}ech complexes
	 \[ \bigoplus \limits_{\rho \in \Sigma(1)}\check{C}^{\bullet}(\{U_\sigma\},\CO(D_{\rho})) \overset{\pi^{*}}{\cong} \bigoplus \limits_{\widetilde\rho \in \widetilde\Sigma(1)}\check{C}^{\bullet}(\{U_{\widetilde \sigma}\},\CO(D_{\widetilde\rho})^G)\]
	 and the claim follows.
	 Note that if $H^1(X,\CO_X)= H^2(X,\CO_X)=0$, \thref{thm:combiso} implies that we actually have an isomorphism
	\[\Def_{\widetilde X}^G \to \Def'_{X_{\Sigma}},\]
that is, $G$-invariant deformations of the Cox torsor are equivalent to locally trivial deformations of $X_\Sigma$.

This is very much related to the work of \cite{comp}, in which $G$-invariant deformations of some affine scheme $Y$ are compared with deformations of the quotient $X$ under the action of $G$ of an invariant open subscheme $U\subseteq Y$; here $G$ is a linearly reductive group. In the toric setting, the natural affine scheme $Y$ to consider is the spectrum of the Cox ring $\Aff^{\#\Sigma(1)}$. However, as affine space is rigid, one will only obtain trivial deformations in this manner.

The above discussion can be generalized far beyond toric varieties. Let $X$ be any normal $\QQ$-factorial variety with finitely generated class group and no non-trivial global invertible functions, and let $\widetilde X$ be the relative spectrum of its Cox sheaf \cite[Construction 1.4.2.1]{cox}. We again call $\widetilde X$ the Cox torsor of $X$ (although it is only a torsor if $X$ is factorial). As before, $\pi:\widetilde X\to X$ is a geometric quotient by the group $G=\Hom(\Cl(X),\KK^*)$.
If $X$ is factorial, there is a generalized Euler sequence 
\[
	0\to \Hom_\ZZ (\Cl(X),\ZZ)\otimes \CO_X\to \pi_*(\T_{\widetilde X}^G) \to \T_X\to 0,
\]
see e.g.~\cite[Theorem 5.12]{comp} for even more general conditions guaranteeing such a sequence.
In any case, given such a generalized Euler sequence, if $H^1(X,\CO_X)=H^2(X,\CO_X)=0$, we may apply Theorem \ref{compa} to conclude that $\Def_X'$ is isomorphic to the functor 
$(\Def_{\widetilde X}')^G$
of locally trivial $G$-invariant deformations of the Cox torsor $\widetilde X$.
\end{rem}

\subsection{The combinatorial deformation equation}\label{sec:cde}
In this section, we specialize the setup discussed in \S\ref{sec:defeq} to the combinatorial deformation functor $\cDef_{\Sigma}$.
As discussed in the proof of \thref{thm:combiso}, this functor may be identified with a functor of the form $\widehat\rF_{\mcL}$ with respect to the open cover $\sU=\{U_\sigma\}_{\sigma\in\Sigma_{\max}}$ of $X_\Sigma$, allowing us to apply the setup of \S\ref{sec:defeq}. However, we will think about $\cDef_\Sigma$ instead as in Definition \ref{defn:cdef} with respect to the various closed covers $\sV_{\rho,\bu}$, which are also indexed by elements of $\Sigma_{\max}$.

To start solving the deformation equation $\eqref{eq:defeq2}$, we need to fix bases for the tangent and obstruction spaces of $\cDef_\Sigma$.
We make this explicit here.
By Lemma \ref{lemma:tangentspace}, \thref{exactseqlocal}, and Equation \eqref{iso1} we obtain the tangent space
\[\rT^1 \cDef_{\Sigma}\cong  \bigoplus_{\rho\in\Sigma(1), \bu\in M } \widetilde{H}^{0}(V_{\rho,\bu},\KK). \]
To any connected component $C$ of the simplicial complex $V_{\rho,\bu}$ we associate a zero cocycle  $\theta_C \in \check{Z}^0(\sV_{\rho,\bu},\KK)$ as follows:
\[
\{\theta_C\}_{\sigma}=  
\begin{cases}
	1 & \textrm{if} \; \sigma\in \Sigma_{\max} \; \textrm{and} \; C\cap \sigma\not= \emptyset,\\
	0 & \textrm{otherwise.}
\end{cases}
\]
The images of $\theta_{C}$ span $\widetilde{H}^0(V_{\rho,\bu},\KK)$ and removing any one of these provides a basis.
Doing this for all $\rho,\bu$ with $\widetilde{H}^{0}(V_{\rho,\bu},\KK)\neq 0$, we obtain a basis \[\theta_1 ,\ldots,\theta_p \] of $\rT^1 \cDef_{\Sigma}$
where each $\theta_i$ is of the form $\theta_{C}\cdot \chi^\bu\cdot f_\rho$ for some $\rho\in\Sigma(1)$, $\bu\in M$, and connected component $C$ of $V_{\rho,\bu}$ (see Convention \ref{convention}).
This will also determine
\[\alpha^{(1)}= \sum_{\ell=1}^{p} t_{\ell}\cdot \theta_{\ell}.\]

An obstruction space for $\cDef_{\Sigma}$ is given by
\[ \bigoplus_{\rho\in\Sigma(1), \bu\in M } \widetilde{H}^{1}(V_{\rho,\bu},\KK),\]
by Lemma \ref{lemma:obstruction}, \thref{exactseqlocal}, and Equation \eqref{iso1}. Before choosing cocycles $\omega_1,\ldots,\omega_q$ whose images give a basis of the obstruction space, 
we fix any $\KK$-linear map 
\[\psi: \bigoplus_{\rho,\bu}\check{C}^1(\sV_{\rho,\bu},\KK) \to \bigoplus_{\rho,\bu} \check{C}^0(\sV_{\rho,\bu},\KK)\]
that is compatible with the direct sum decomposition and such that $d\circ \psi(\omega)=\omega$ if $\omega\in \check{C}^1(\sV_{\rho,\bu},\KK)$ is a coboundary.

Here is one explicit way to do this: fix an ordering of the elements of $\Sigma_{\max}$. For each connected component $C$ of $V_{\rho,\bu}$, this determines a unique cone $\sigma_C$ that is minimal among all cones $\sigma\in\Sigma_{\max}$ intersecting $C$ non-trivially.
For any cone $\sigma$ intersecting $C$, there is a unique sequence $\tau_1=\sigma,\tau_2,\ldots,\tau_k=\sigma_C$ such that for each $i\geq 1$, $\tau_i\cap \tau_{i+1}\cap C\neq \emptyset$, $k$ is minimal, and the sequence is minimal in the lexicographic order with respect to the previous properties. 
Given $\omega\in \check{C}^1(\sV_{\rho,\bu},\KK)$
we then define
\[\psi(\omega)_{\sigma}=\sum_{i=1}^{k-1} \omega_{\tau_{i+1},\tau_{i}}
\]
if $\sigma\cap C\neq \emptyset$ for some connected component $C$, and set $\psi(\omega)_{\sigma}=  0$ otherwise.
Note that in particular, $\psi(\omega)_{\sigma_C}=0$. It is straightforward to verify that $\psi$ has the desired property.

We now choose one-cocycles \[\omega_1,\ldots,\omega_q\in\bigoplus_{\rho\in\Sigma(1), \bu\in M } \check{Z}^1(\sV_{\rho,\bu},\KK)\] whose images in 
\[
 \bigoplus_{\rho\in\Sigma(1), \bu\in M } \widetilde{H}^{1}(V_{\rho,\bu},\KK)
\]
form a basis, such that each $\omega_i$ lies in a single direct summand, and such that $d(\psi(\omega))=0$. From an arbitrary set of cocycles $\omega_1',\ldots,\omega_q'$ whose images form a basis, we may set
\[
\omega_i=\omega_i'-d(\psi(\omega_i'))
\]
to obtain that $d(\psi(\omega_i))=0$.

In this situation, we refer to the corresponding deformation equation \eqref{eq:defeq2} for $\cDef_{\Sigma}$ as the \emph{combinatorial deformation equation}. By \thref{prop:defeqsolving}, we know that for each order $r$ it has a solution 
\begin{align*}
	&\beta^{(r+1)} \in \bigoplus_{\rho,\bu}\check{C}^0(\sV_{\rho,\bu},\KK)\otimes \mfm_{r+1};\qquad
	\gamma_\ell^{(r+1)}\in \mfm_{r+1}.
\end{align*}
In fact, we may obtain a solution using the map $\psi$:
\begin{prop}\label{prop:cdefeqsolving}
	Let $\eta\in \bigoplus_{\rho,\bu}\check{C}^1(\sV_{\rho,\bu},\KK)\otimes \mfm_{r+1}$ be the normal form of $\mfoc(\alpha^{(r)})-\sum_{\ell=1}^qg_{\ell}^{(r)}\omega_\ell$ with respect to $\mfm\cdot J_r$. Then $\beta^{(r+1)}=\psi(\eta)$ and $\gamma_\ell^{(r+1)}$ 
give a solution to the combinatorial deformation equation, where $\gamma_\ell^{(r+1)}$ is determined by
\[
\sum_{\ell=1}^q \gamma_\ell^{(r+1)}\cdot  \omega_\ell=\eta-d(\beta^{(r+1)}).
\]
\end{prop}
\begin{proof}
	By \thref{prop:defeqsolving}\ref{prop:defeqsolve1}, there exists
	\begin{align*}
		\beta' \in \bigoplus_{\rho,\bu}\check{C}^0(\sV_{\rho,\bu},\KK)\otimes \mfm_{r+1}; \quad	\gamma_\ell'\in \mfm_{r+1}
	\end{align*}
	such that	
	\begin{equation*}
		\eta= d(\beta')+\sum_{\ell=1}^q \gamma_\ell' \cdot \omega_{\ell}. 
	\end{equation*}
	Applying $d\circ \psi$ to both sides we obtain
	\begin{equation*}
		d(\beta^{(r+1)})= d(\beta'). 
	\end{equation*}
	Since the $\omega_\ell$ are linearly independent, this implies $\gamma_\ell^{(r+1)}=\gamma_\ell'$.
\end{proof}
To summarize the contents of Proposition \ref{prop:cdefeqsolving}, in order to solve the combinatorial deformation equation, we only need to reduce $\mfoc(\alpha^{(r)})-\sum_{\ell=1}^qg_{\ell}^{(r)}\cdot\omega_\ell$ to its normal form with respect to $\mfm\cdot J_r$ (an unavoidable algebraic step), and then apply the map $\psi$ (a purely combinatorial step). 
By \thref{hull}, iteratively solving the combinatorial deformation equation gives us a procedure for computing the hull of $\cDef_\Sigma$.
We will do this explicitly in several examples in \S\ref{sec:rank3ctd}.

\subsection{Higher order obstructions}\label{sec:obs}
To solve the combinatorial deformation equation discussed in \S \ref{sec:cde} we have to compute $\mfp(s(\alpha))$. In this section, we will derive a general formula for $\mfp(s(\alpha))$ which is of theoretical interest and present explicit formulas for lower order terms.

We first establish some notation. For $w=(w_1,\ldots,w_p)\in \ZZ^p_{\geq0}$, we denote $t_1^{w_1}\ldots t_p^{w_p}$ by $t^{w}$. We choose $\theta_1,\ldots,\theta_p$ using the construction in \S\ref{sec:cde}.  Let $\varphi \in \Hom(\ZZ^p,M)$ be the map sending the $\ell$-th basis vector of $\ZZ^p$ to the degree in $M$ of $\theta_\ell$. Consider 
\[
\alpha=\sum_{\rho\in\Sigma(1)} \sum_{w\in\ZZ_{\geq 0}^p \setminus \{0\}} c^w_{\rho}\cdot t^{w}\cdot  \chi^{\varphi(w)}\cdot f_\rho \in \bigoplus_{\substack{\rho\in\Sigma(1) \\ \bu\in M}}\check{C}^0(\V_{\rho,\bu},\KK)\otimes \mfm,
\]
where $c^w_{\rho}$ is a cochain in $\check{C}^0(\V_{\rho,\varphi(w)},\KK)$ and $\chi^{\varphi(w)}\cdot f_{\rho}$ specifies the summand in which $c_\rho^w$ lies (see Convention \ref{convention}). 

\begin{defn}
	For integers $d\geq 1$ and $1\leq k \leq d$ we define the set $\Delta_{d,k}$ as follows:
	\[
	\Delta_{d,k}=\left\{(a_1,\ldots,a_{d})\in \ZZ_{\geq 1}^{d}\ :\ \begin{array}{l l}
		a_i <i\ \textrm{or}\ a_i=k-\displaystyle\sum_{\substack {j>i\\ j<  a_j}} a_j-j&\textrm{if}\ i<k\\
		a_i=k & \textrm{if}\ i=k\\
		a_i< i &\textrm{if}\ i> k 
	\end{array} \right\}.
	\]
	Likewise, for $a\in \Delta_{d,k}$, we set
	\[
	\sgn(a)=(-1)^{\#\{i\ |\ a_i>i\}}.
	\]
\end{defn}

\begin{defn}
	For an integer $d\geq 1$ and $w \in\ZZ_{\geq 0}^p \setminus \{0\}$ we define the set $\nabla_{w,d}$ as follows: 
	\[
	\nabla_{w,d}=\left\{\vec \bw=(\bw_1,\ldots,\bw_d)\in (\ZZ_{\geq 0}^p\setminus\{0\})^d\ \bigg|\ \sum_{i=1}^d \bw_i=w\right\}.
	\]
\end{defn}

\begin{example}
	Here we give some examples of $\Delta_{d,k}$ and $\nabla_{w,d}$. For $d=3$ and $k=1,2,3$ we have the following sets:
	\begin{align*}
		&\Delta_{3,1}=\{(1,1,1),(1,1,2)\}\\
		&\Delta_{3,2}=\{(2,2,1),(2,2,2)\} \\
		&\Delta_{3,3}=\{(3,1,3),(2,3,3)\}.
	\end{align*}
	For $w=(1,1,1)$ and $d=2,3$, we have the following sets:
	\begin{align*}
		\nabla_{(1,1,1),3}&=\big\{ \big(\pi(e_1),\pi(e_2),\pi(e_3)\big) :  \pi\in S_3\}\\
		\nabla_{(1,1,1),2}&=\big\{ \big(\pi(e_1),\pi(e_2+e_3)\big) :  \pi\in S_3\}.
	\end{align*}
	Here, $\pi(e_i)$ denotes the image of the standard basis vector $e_i$ of $\RR^3$ under the action of $\pi\in S_3$, where $S_3$ is the symmetric group on three elements.
	These sets appear in the formula for the coefficient of 
	$t_1t_2t_3$ in Table \vref{Table:obspoly}.
\end{example}

Consider any BCH formula
\[
x\bch y=\sum_{d \geq 1} \sum_{\sigma(x,y)\in\mfS_d} b_{\sigma} [\sigma(x,y)]
\]
where $b_\sigma\in \QQ$, $\mfS_d$ is some set of words $\sigma(x,y)$ of length $d$  in $x$ and $y$, and 
$[\sigma(x,y)]$ denotes the iterated Lie bracket (see notation following \eqref{dynkin}); we may for example take Dynkin's formula \eqref{dynkin}. Set  
\[\sgn(\sigma)=(-1)^{\#\{i\ | \sigma_i=x\}}.\]
We will use the notation $\vec\rho=(\rho_1,\ldots,\rho_d)$ for a $d$-tuple of rays of $\Sigma$. In this section, for compactness of notation and to avoid confusion with elements of $\mfS_d$, we will denote elements of $\Sigma_{\max}$ by $i$ and $j$.

\begin{thm}\thlabel{thm:obsformula}
	 The coefficient of $t^w$ in $\mfp(s(\alpha))_{ij}$ is the product of $\chi^{\varphi(w)}$
	with
	\begin{align*}
		\sum_{\substack{d\geq 1\\ \vec \bw \in \nabla_{w,d}\\
				\vec\rho\in\Sigma(1)^d}}
		\left(\sum_{\sigma\in \mfS_d} \sgn(\sigma) b_\sigma \prod_{k=1}^d \big(c^{\bw_k}_{\rho_k}\big)_{\sigma(i,j)_{d-k+1}}\right)
		 \Biggl(
		\sum_{\substack{k=1\ldots d\\a\in\Delta_{d,k}}}\sgn(a)\prod_{\ell\neq k}\rho_{\ell}(\varphi(\bw_{a_\ell})) \cdot f_{\rho_k}
		\Biggr).
	\end{align*}
\end{thm}

We postpone the proof of \thref{thm:obsformula} until the end of this section. We first proceed to describe the explicit formulas for lower order terms. Using the notation

{\scriptsize
\begin{align*}
	\alpha^{(1)}_i &= t_1\cdot\chi^{\bu_1}\cdot c^{(1,0)}_{1,i}\cdot  f_1 + t_2\cdot\chi^{\bu_2}\cdot c^{(0,1)}_{2,i}\cdot  f_2\\
	\alpha^{(2)}_i&= \alpha^{(1)}_i+
	t_1t_2\cdot\chi^{\bu_1+\bu_2}\cdot (c^{(1,1)}_{1,i}\cdot f_1+ c^{(1,1)}_{2,i}\cdot f_2)\\
	\alpha^{(3)}_i&= \alpha^{(2)}_i+ t_1^2t_2\cdot\chi^{2\bu_1+\bu_2}\cdot (c^{(2,1)}_{1,i}\cdot f_1+ c^{(2,1)}_{2,i}\cdot f_2)+ t_1t_2^2\cdot\chi^{\bu_1+2\bu_2}\cdot (c^{(1,2)}_{1,i}\cdot f_1+ c^{(1,2)}_{2,i}\cdot f_2)
\end{align*}
}
\noindent
we list the formulas for the coefficients of  $t_1t_2,t_1t_2^2,t_1t_2^3$ and $t_1^2t_2^2$ in Table~\ref{Table:obspoly}. By applying $\lambda$ to the coefficient of $t_1t_2$ in $\mfp(s(\alpha^{(1)}))_{ij}$, we recover the combinatorial cup-product from \cite[Theorem 4.3]{ilten3}. Similarly, using the notation
{\scriptsize
\begin{align*}
	\alpha^{(2)}_i &= t_1\cdot\chi^{\bu_1}\cdot c^{e_1}_{1,i}\cdot  f_1 + t_2\cdot\chi^{\bu_2}\cdot c^{e_2}_{2,i}\cdot  f_2+ t_3\cdot\chi^{\bu_3}\cdot c^{e_3}_{3,i}\cdot  f_3+ \\
	&t_1t_2\cdot\chi^{\bu_1+\bu_2}\cdot (c^{e_1+e_2}_{1,i}\cdot f_1+ c^{e_1+e_2}_{2,i}\cdot f_2) + t_1t_3\cdot\chi^{\bu_1+\bu_3}\cdot (c^{e_1+e_3}_{1,i}\cdot f_1+ c^{e_1+e_3}_{3,i}\cdot f_3)\\
	&+ t_2t_3\cdot\chi^{\bu_2+\bu_3}\cdot (c^{e_2+e_3}_{2,i}\cdot f_2+ c^{e_2+e_3}_{3,i}\cdot f_3)
\end{align*}
}
we also list a formula for the coefficient of $t_1t_2t_3$ of $\mfp(\alpha^{(2)})_{ij}$ in Table \ref{Table:obspoly}.
We thus have explicit formulas for all obstructions of third order, and fourth order obstructions involving only two deformation directions. In theory, we could write down similar formulas for higher order obstructions using \thref{thm:obsformula}, but they become increasingly large. We note that unlike for the cup product case, the formulas for obstructions of degree larger than two involve not only first order deformations but also higher order perturbation data.

\begin{table}[htbp]
	\centering
	\caption{List of lower-order obstruction polynomials, with contributions solely from first order terms on first line}
	\scriptsize
	\label{Table:obspoly}
	\resizebox{\textwidth}{!}{%
		\begin{tabular}{@{}p{\textwidth}@{}}
			\toprule	
			coefficient of 
			$t_1t_2\cdot \chi^{\bu_1+\bu_2}$ in 
			$\mfp(s(\alpha^{(1)}))_{ij}$ is
			\vspace{0.5ex}
			\begin{align*}
				\scriptstyle\frac{1}{2}\Big(c^{(1,0)}_{1,i}\cdot c^{(0,1)}_{2,j}- c^{(1,0)}_{1,j}\cdot c^{(0,1)}_{2,i}    \Big) \cdot (\rho_2(\bu_1)\cdot f_1- \rho_1(\bu_2)\cdot f_2)
			\end{align*}\\
			\midrule
			
			coefficient of 
			$t_1t_2^2\cdot \chi^{\bu_1+2\bu_2}$ in 
			$\mfp(s(\alpha^{(2)}))_{ij}$ is
			\vspace{0.5ex}
			\begin{align*}
				\scriptstyle \frac{-1}{12} \big\{c^{(1,0)}_{1,i}\cdot c^{(0,1)}_{2,j}- c^{(1,0)}_{1,j}\cdot c^{(0,1)}_{2,i} \big \}\cdot \big \{c^{(0,1)}_{2,i}+ c^{(0,1)}_{2,j} \big\}\cdot  \rho_2(\bu_1)\cdot \big\{\rho_2(\bu_1+\bu_2)\cdot f_1 - 2 \rho_1(\bu_2)\cdot f_2 \big\} \\\\
				{\scriptstyle + \frac{1}{2} \big\{c^{(1,1)}_{1,i}\cdot c^{(0,1)}_{2,j}- c^{(1,1)}_{1,j}\cdot c^{(0,1)}_{2,i} \big \}\cdot  \big\{\rho_2(\bu_1+\bu_2)\cdot f_1  -\rho_1(\bu_2) \cdot f_2\big\} }
				{\scriptstyle + \frac{1}{2} \big\{c^{(1,1)}_{2,i}\cdot c^{(0,1)}_{2,j}- c^{(1,1)}_{2,j}\cdot c^{(0,1)}_{2,i} \big \}  \cdot\rho_2(\bu_1)\cdot  f_2 }
			\end{align*} \\
			\midrule
			
			coefficient of 
			$t_1t_2^3\cdot \chi^{\bu_1+3\bu_2}$ in 
			$\mfp(s(\alpha^{(3)}))_{ij}$ is
			\vspace{0.5ex}
			\begin{align*}
				\scriptstyle \frac{1}{24} \big\{c^{(1,0)}_{1,i}\cdot c^{(0,1)}_{2,j}- c^{(1,0)}_{1,j}\cdot c^{(0,1)}_{2,i} \big\}\cdot c^{(0,1)}_{2,i}\cdot c^{(0,1)}_{2,j}\cdot 
				 \rho_2(\bu_1)\cdot \rho_2(\bu_1+\bu_2)\cdot  \big\{\rho_2(\bu_1+2\bu_2) \cdot f_1 - 3 \rho_1(\bu_2)\cdot f_2 \big\} 
				\\\\
				{\scriptstyle - \frac{1}{12} \big\{c^{(1,1)}_{1,i}\cdot c^{(0,1)}_{2,j}- c^{(1,1)}_{1,j}\cdot c^{(0,1)}_{2,i} \big\} \cdot \big\{c^{(0,1)}_{2,i}+ c^{(0,1)}_{2,j}\big\} 
				 \cdot \rho_2(\bu_1+\bu_2)\cdot \big \{\rho_2(\bu_1+2\bu_2)\cdot f_1 -2 \rho_1(\bu_2) \cdot f_2\big \}} \\
				{\scriptstyle - \frac{1}{12}\cdot \big\{c^{(1,1)}_{2,i}\cdot c^{(0,1)}_{2,j}- c^{(1,1)}_{2,j}\cdot c^{(0,1)}_{2,i} \big\} \cdot \big\{c^{(0,1)}_{2,i}+ c^{(0,1)}_{2,j}\big\} 
					\scriptstyle \cdot \rho_2(\bu_1+\bu_2) \cdot  \rho_2(\bu_1)\cdot f_2 } \\
				{\scriptstyle + \frac{1}{2} \big\{c^{(1,2)}_{1,i}\cdot c^{(0,1)}_{2,j} - c^{(1,2)}_{1,j}\cdot c^{(0,1)}_{2,i} \big\}  \scriptstyle  \big \{\rho_2(\bu_1+2\bu_2)\cdot f_1 - \rho_1(\bu_2) \cdot f_2\big \}} 	
				{\scriptstyle + \frac{1}{2} \big\{c^{(1,2)}_{2,i}\cdot c^{(0,1)}_{2,j} - c^{(1,2)}_{2,j}\cdot c^{(0,1)}_{2,i} \big\}  \rho_2(\bu_1+\bu_2) \cdot f_2}
			\end{align*} \\
			\midrule
			
			coefficient of 
			$t_1^2t_2^2\cdot \chi^{2\bu_1+2\bu_2}$ in 
			$\mfp(s(\alpha^{(3)}))_{ij}$ is
			\vspace{0.5ex}
			\begin{align*}
				\scriptstyle \frac{1}{12} \big\{c^{(1,0)}_{1,i}\cdot c^{(0,1)}_{2,j}- c^{(1,0)}_{1,j}\cdot c^{(0,1)}_{2,i} \big\}\cdot  
				\big\{c^{(1,0)}_{1,i}\cdot c^{(0,1)}_{2,j}+ c^{(1,0)}_{1,j}\cdot c^{(0,1)}_{2,i} \big\}\cdot 
				\rho_2(\bu_1)\cdot \rho_1(\bu_2)\cdot  \big\{ \rho_2(2\bu_1+\bu_2)\cdot f_1 - \rho_1(\bu_1+2\bu_2) \cdot f_2 \big\} 
				\\\\
				{\scriptstyle- \frac{1}{12} \big\{c^{(1,1)}_{1,i}+c^{(1,1)}_{1,j} \big\} \cdot \big\{c^{(1,0)}_{1,i}\cdot c^{(0,1)}_{2,j}- c^{(1,0)}_{1,j}\cdot c^{(0,1)}_{2,i} \big\}\cdot 
					\rho_1(\bu_2)\cdot  \big\{\rho_2(\bu_1+\bu_2)  \cdot f_1-  \rho_1(\bu_1+\bu_2)\cdot f_2 \big\} }
				\\
				{\scriptstyle- \frac{1}{12} \big\{c^{(1,1)}_{2,i}+c^{(1,1)}_{2,j} \big\} \cdot \big\{c^{(1,0)}_{1,i}\cdot c^{(0,1)}_{2,j}- c^{(1,0)}_{1,j}\cdot c^{(0,1)}_{2,i} \big\}\cdot 
					\rho_2(\bu_1)\cdot  \big\{\rho_2(\bu_1+\bu_2)  \cdot f_1-  \rho_1(\bu_1+\bu_2)\cdot f_2 \big\} }
				\\   
				{\scriptstyle- \frac{1}{12} \big\{c^{(1,0)}_{1,i}+c^{(1,0)}_{1,j} \big\} \cdot \big\{c^{(1,1)}_{1,i}\cdot c^{(0,1)}_{2,j}- c^{(1,1)}_{1,j}\cdot c^{(0,1)}_{2,i} \big\}\cdot 
					\rho_1(\bu_2)\cdot  \big\{\rho_2(3\bu_1+2\bu_2)  \cdot f_1-  \rho_1(\bu_1+2\bu_2)\cdot f_2 \big\} }
				\\ 
				{\scriptstyle- \frac{1}{12} \big\{c^{(1,0)}_{1,i}+c^{(1,0)}_{1,j} \big\} \cdot \big\{ c^{(0,1)}_{2,i} \cdot c^{(1,1)}_{2,j} - c^{(0,1)}_{2,j} \cdot c^{(1,1)}_{2,i} \big\}\cdot 
					\rho_2(\bu_1)\cdot  \big\{\rho_2(\bu_1)  \cdot f_1-  \rho_1(\bu_1+2\bu_2)\cdot f_2 \big\} }
				\\ 
				{\scriptstyle- \frac{1}{12} \big\{c^{(0,1)}_{2,i}+c^{(0,1)}_{2,j} \big\} \cdot \big\{c^{(1,1)}_{1,i}\cdot c^{(1,0)}_{1,j}- c^{(1,1)}_{1,j}\cdot c^{(1,0)}_{1,i} \big\}\cdot 
					\rho_1(\bu_2)\cdot  \big\{\rho_2(2\bu_1+\bu_2)  \cdot f_1-  \rho_1(\bu_2)\cdot f_2 \big\}} 
				\\ 
				{\scriptstyle - \frac{1}{12} \big\{c^{(0,1)}_{2,i}+c^{(0,1)}_{2,j} \big\} \cdot \big\{c^{(1,1)}_{2,i}\cdot c^{(1,0)}_{1,j}- c^{(1,1)}_{2,j}\cdot c^{(1,0)}_{1,i} \big\}\cdot 
					\rho_2(\bu_1)\cdot  \big\{-\rho_2(2\bu_1+\bu_2)  \cdot f_1+ \rho_1(2\bu_1+3\bu_2)\cdot f_2 \big\} }
				\\
				{\scriptstyle+ \frac{1}{2}\cdot
					\big\{ c^{(1,1)}_{1,i}\cdot c^{(1,1)}_{2,j}- c^{(1,1)}_{1,j}\cdot c^{(1,1)}_{2,i} \big\} \cdot
					\big\{  \rho_2(\bu_1+\bu_2) \cdot f_1 -\rho_1(\bu_1+\bu_2) \cdot f_2 \big\}} 
				\\
				{\scriptstyle+\frac{1}{2}\cdot
					\big\{ c^{(1,2)}_{2,i} \cdot  c^{(1,0)}_{1,j}- c^{(1,2)}_{2,j} \cdot c^{(1,0)}_{1,i} \big \} \cdot
					\big\{-\rho_2(\bu_1) \cdot f_1+ \rho_1(\bu_1+2\bu_2)\cdot f_2 \big\} 
					+
					\big\{c^{(1,2)}_{1,i}\cdot c^{(1,0)}_{1,j}- c^{(1,2)}_{1,j}\cdot c^{(1,0)}_{1,i} \big\}
					\cdot  \rho_1(\bu_2)\cdot f_1}
				\\
				{	\scriptstyle+\frac{1}{2}\cdot
					\big\{c^{(2,1)}_{1,i}\cdot c^{(0,1)}_{2,j}- c^{(2,1)}_{1,j}\cdot c^{(0,1)}_{2,i} \big\}
					\cdot
					\big\{\rho_2(2\bu_1+\bu_2)  \cdot f_1-  \rho_1(\bu_2)\cdot f_2 \big\} 
					\scriptstyle+
					\big\{c^{(2,1)}_{2,i}\cdot c^{(0,1)}_{2,j}- c^{(2,1)}_{2,j}\cdot c^{(0,1)}_{2,i} \big\}
					\cdot  \rho_2(\bu_1)\cdot f_2}
				\\
			\end{align*} \\
			\midrule
			
			coefficient of 
			$t_1t_2t_3\cdot \chi^{\bu_1+\bu_2+\bu_3}$ in 
			$\mfp(s(\alpha^{(2)}))_{ij}$ is
			\vspace{0.5ex}
			\begin{align*}
				\scriptstyle \frac{1}{12}\cdot \sum_{\pi \in S_3} c^{\pi(e_1)}_{\pi(1),j}\cdot c^{\pi(e_2)}_{\pi(2),i}\cdot \big \{c^{\pi(e_3)}_{\pi(3),i}+ c^{\pi(e_3)}_{\pi(3),j} \big \}  \cdot
				\big\{ \big(\rho_{\pi(2)}(\bu_{\pi(1)})\cdot \rho_{\pi(3)}(\bu_{\pi(1)})+ \rho_{\pi(2)}(\bu_{\pi(1)})\cdot \rho_{\pi(3)}(\bu_{\pi(2)})\big)\cdot f_{\pi(1)}
				\\ 
				\scriptstyle+ \big(-\rho_{\pi(1)}(\bu_{\pi(2)})\cdot \rho_{\pi(3)}(\bu_{\pi(1)})- \rho_{\pi(1)}(\bu_{\pi(2)})\cdot \rho_{\pi(3)}(\bu_{\pi(2)})\big)\cdot f_{\pi(2)}
				\\
				\scriptstyle +\big(-\rho_{\pi(1)}(\bu_{\pi(3)})\cdot \rho_{\pi(2)}(\bu_{\pi(1)})+ \rho_{\pi(1)}(\bu_{\pi(2)})\cdot \rho_{\pi(2)}(\bu_{\pi(3)})\big)\cdot f_{\pi(3)} \big \}
				\\\\
				{\scriptstyle
					+\frac{1}{2}\cdot \sum_{\pi \in S_3} \big\{c^{\pi(e_1)}_{\pi(1),i}\cdot c^{\pi(e_2)+\pi(e_3)}_{\pi(2),j}- c^{\pi(e_1)}_{\pi(1),j}\cdot c^{\pi(e_2)+\pi(e_3)}_{\pi(2),i}\big\} \cdot \big\{\rho_{\pi(2)}(\bu_{\pi(1)})\cdot f_{\pi(1)}- \rho_{\pi(1)}(\bu_{\pi(2)}+ \bu_{\pi(3)})\cdot f_{\pi(2)} \big\} 
				}   
			\end{align*}
			\\
			\bottomrule
		\end{tabular}
	}
\end{table}

We now start on proving \thref{thm:obsformula}.
Let $\rho_1,\ldots,\rho_{m}$ be not-necessarily distinct rays of $\Sigma$, and $\bv_1,\ldots,\bv_{m}\in M$. We will need an explicit formula for the iterated Lie bracket
\[
[\chi^{\bv_{m}}\cdot f_{\rho_{m}}\ \chi^{\bv_{m-1}}\cdot f_{\rho_{m-1}}\ \cdots \ \chi^{\bv_1}\cdot f_{\rho_1}],
\]
see discussion following \eqref{dynkin} for notation.

\begin{prop}\thlabel{prop:iteratedlie}
	The iterated Lie bracket
	\[
	[\chi^{\bv_m}\cdot f_{\rho_m}\ \chi^{\bv_{m-1}}\cdot f_{\rho_{m-1}}\ \cdots \ \chi^{\bv_1}\cdot f_{\rho_1}]
	\]
	is equal to 
	\[
	\chi^{\bv_1+\cdots+\bv_{m}}\cdot \sum_{k=1}^{m}\left(\sum_{a\in\Delta_{m,k}}\sgn(a)\prod_{i\neq k}\rho_{i}(\bv_{a_i}) \right)\cdot f_{\rho_k}.\]
\end{prop}
\begin{proof}
	The proof is by induction. The base case $m=1$ is immediate. For proving the induction step, it is enough to show that
	\begin{equation}\label{eq:Delta1}
		\sum_{a\in\Delta_{m-1,k}} \sum_{j=1}^{m-1}\sgn(a) \prod_{i\neq k}\rho_{i}(\bv_{a_i})\rho_m(\bv_{j})= \sum_{a\in\Delta_{m,k}} \sgn(a) \prod_{i\neq k}\rho_{i}(\bv_{a_i})
	\end{equation}
	for $k=1,\ldots,m-1$ and
	\begin{equation}\label{eq:Delta2}
		\sum_{k=1}^{m-1}-\rho_k(\bv_m)\cdot \sum_{a\in\Delta_{m-1,k}}\sgn(a)\prod_{i\neq k}\rho_{i}(\bv_{a_i})\cdot = \sum_{a\in\Delta_{m,m}} \sgn(a) \prod_{i\neq m}\rho_{i}(\bv_{a_i}).
	\end{equation}

	To establish this, we use the following straightforward inductive relations on the sets $\Delta_{m,k}$. For $k\leq m-1$ the map
	$\pi_k: \Delta_{m,k} \to \Delta_{m-1,k}$
	defined by \[\pi_k(a) =(a_1,\ldots,a_{m-1})\]
	is $(m-1)$-to-$1$ and satisfies $\sgn(\pi_k(a))=\sgn(a)$. Additionally, for $k\leq m-1$ the map $\hat\pi_k: \Delta_{m-1,k} \to \Delta_{m,m}$ defined by 	\[\hat\pi_k(a)=(a_1,\ldots,a_{k-1},m,a_{k+1},\ldots,a_{m-1},m)\]
is injective, $\sgn(\hat\pi_k(a))=-\sgn(a)$, and we have
	\[\Delta_{m,m}= \bigsqcup_{k=1}^{m-1} \hat \pi_k ( \Delta_{m-1,k}),\]
	where $\bigsqcup$ denotes disjoint union.
The proofs of \eqref{eq:Delta1} and \eqref{eq:Delta2} follow directly from these observations.
\end{proof}

\begin{proof}[Proof of \thref{thm:obsformula}]
We have
	\begin{align*}
		\mfp(\alpha)_{ij}&= \sum_{d \geq 1} \sum_{\sigma(x,y)\in\mfS_d} b_{\sigma} [\sigma(-\alpha_i,\alpha_j)].
	\end{align*}
	Expanding the right-hand side, we obtain that the coefficient of  $t^w$ is \begin{align*}
	\sum_{\substack{d\geq 1\\\vec \bv \in \nabla_{w,d}\\
			\vec\rho\in\Sigma(1)^d}}
	\sum_{\sigma\in \mfS_d} 
	\sgn(\sigma)\cdot b_{\sigma}\cdot \prod_{k=1}^d \Big(c^{\bw_k}_{\rho_k}\Big)_{\sigma(i,j)_{d-k+1}} \cdot \Big
		[\chi^{\varphi(\bw_d)}\cdot f_{\rho_d}\  \cdots \ \chi^{\varphi(\bw_1)}\cdot f_{\rho_1} \Big].
	\end{align*}
	Applying \thref{prop:iteratedlie}, we then obtain that this is equal to $\chi^{\varphi(w)}$
	multiplied with the quantity in the statement of the theorem.
\end{proof}

\subsection{Removing cones}\label{sec:rc}
In \S \ref{sec:cde}, we discussed how to compute the hull of $\cDef_{\Sigma}$ using the combinatorial deformation equation. When $X_{\Sigma}$ has mild singularities, this approach indeed yields the hull of $\Def_{X_{\Sigma}}$, see \thref{mildsing}. Although working with $\cDef_{\Sigma}$ is less complex than $\Def_{X_\Sigma}$, computations still involve dealing with every maximal cone of $\Sigma$ and their pairwise intersections. In this section, we will explore methods to streamline the process of determining the hull of $\cDef_{\Sigma}$ by reducing the number of maximal cones and intersections which we must consider. We will apply these techniques when we compute hull for examples in \S \ref{sec:rank3ctd}.

Throughout this section, $\Sigma$ is any simplicial fan.
\begin{defn}\label{defn:A}
	Let 
	\begin{equation*}
	\A=\{(\rho,\bu)\in\Sigma(1)\times M\ |\  \widetilde{H}^0(V_{\rho,\bu},\KK)\not=0\}.
\end{equation*}
	We let $\Gamma \subseteq \Sigma(1)\times M$ be the smallest set containing $\A$ that satisfies the following:
	\begin{enumerate}
		\item	If $(\rho,\bu),(\rho',\bu')\in \Gamma$, $\rho\neq \rho'$, and $\rho'(\bu)\neq 0$ then $(\rho,\bu+\bu')\in\Gamma$;
		\item If $(\rho,\bu),(\rho,\bu')\in \Gamma$ and $\rho(\bu')\neq \rho(\bu)$ then $(\rho,\bu+\bu')\in\Gamma$.
	\end{enumerate}
	We then define $\mcL_\Gamma= \bigoplus_{(\rho, \bu)\in \Gamma} \CO(D_{\rho})_{\bu}$, where $\CO(D_{\rho})_{\bu}$ is the subsheaf of $\CO(D_\rho)$ defined by
	\[
		\CO(D_\rho)_\bu(U)=\CO(D_\rho)(U)\cap \KK\cdot \chi^\bu.
	\]
\end{defn}
\noindent
It is straightforward to verify that the subsheaf $\mcL_\Gamma$ is stable under the Lie bracket (Definition \ref{bilinear})  on 
$\bigoplus_{\rho \in \Sigma(1)}\CO(D_{\rho}).$

Let $\Sigma'$ be a subfan of $\Sigma$ with $\Sigma_{\max}'\subseteq \Sigma_{\max}$. We say that $\Sigma'$ \emph{covers $\Gamma$} if for every $(\rho,\bu)\in \Gamma$, $V_{\rho,\bu}\subseteq |\Sigma'|$. In this case, we obtain a cover $\sV'_{\rho,\bu}$ of $V_{\rho,\bu}$ by intersecting $V_{\rho,\bu}$ with the cones in $\Sigma_{\max}'$.

\begin{figure}[htbp]
	
	\begin{tikzpicture}[scale=0.6]
		\coordinate (1) at (3,0);
		\coordinate (2) at (0,3);
		\coordinate (3) at (-3,0);
		\coordinate (4) at (0,-3);
		\coordinate (5) at (0,0);

		\draw (1)-- (2)--(3)--(4)--cycle;
		\draw (5)--($(5)+(4,0)$);
		\draw (5)--($(5)+(0,4)$);
		\draw (5)--($(5)+(-4,0)$);
		\draw (5)--($(5)+(0,-4)$);
		
		\draw[fill=black] (1) circle (2pt);
		\node[anchor=north west,font=\tiny] at (1) {1};
		\draw[fill=black] (2) circle (2pt);
		\node[anchor=south west,font=\tiny] at (2) {2};
		\draw[fill=black] (3) circle (2pt);
		\node[anchor=north east,font=\tiny] at (3) {3};
		\draw[fill=black] (4) circle (2pt);
		\node[anchor=north west,font=\tiny] at (4) {4};
		\draw[fill=black] (5) circle (2pt);
		\node[anchor=north west,font=\tiny] at (5) {5};

		\node[anchor=north west,font=\small] at (1,1) {$\sigma_2$};
		\node[anchor=north east,font=\small] at (-1,1) {$\sigma_3$};
		\node[anchor=south east,font=\small] at (-1,-1) {$\sigma_4$};
		\node[anchor=south west,font=\small] at (1,-1) {$\sigma_1$};

	\end{tikzpicture}
	\caption{A representation of the fan in Example \ref{example:Gamma} as an abstract simplicial complex with the ray $\rho_6$ as a vertex at $\infty$ (not to scale).}
	
	\label{fig:cones}
\end{figure}

\begin{example}\label{example:Gamma}
	Fix the lattice $N=\ZZ^3$. We consider a fan $\Sigma$ with six rays, where the generator of the  $i$th ray $\rho_i$ is given by the $i$th column of the following matrix:
	\[
	\begin{pmatrix}
		1 &	0& -1&	0& 0 &  0 \\
		0 &	1&	e& -1& 0 &  0 \\
		0 &	0&	a&	b& 1 & -1
	\end{pmatrix}.
	\]
	We assume that $e,b \geq0$ (the reason for this assumption will be explained  in \S \ref{sec:rank3}).
Rays belong to a common cone of $\Sigma$ if the corresponding set of vertices in Figure \vref{fig:cones} belong to the same simplex, with the ray $\rho_6$ as a vertex at infinity.

	Suppose that we know that the ray-degree pairs $(\rho,\bu)\in \A$  are of the form 
	\[ \Big( \rho_2, (*,*,0) \Big) \quad  \textrm{or} \quad \Big(\rho_5, (*,*,-1) \Big) \quad  \text{or} \quad \Big(\rho_6, (*,*,1) \Big).\] 
	We will establish this in \thref{lemma:SigmaD}.
	It is then straightforward to verify that any element of $\Gamma$ not in $\A$ must be of the form
		\[ \Big( \rho_5, (*,*,-1) \Big) \quad  \textrm{or} \quad \Big(\rho_6, (*,*,1) \Big) \quad  \text{or} \quad \Big(\rho_5, (*,*,0) \Big) \quad  \text{or} \quad \Big(\rho_6, (*,*,0) \Big).\]
	
	Given the above assumption on $\A$, we claim that the fan $\Sigma'$ with maximal cones
	\begin{align*}
		\sigma_1&= \cone(\rho_1,\rho_4,\rho_5), \quad \sigma_2= \cone(\rho_1,\rho_2,\rho_5), \\
		\sigma_3&= \cone(\rho_2,\rho_3,\rho_5), \quad  \sigma_4= \cone(\rho_3,\rho_4,\rho_5).
	\end{align*}
	covers $\Gamma$. Indeed, it is straightforward to see that for $(\rho,\bu)\in\Gamma$, $n_{\rho_6}\notin V_{\rho,\bu}$ and the claim follows.
	\end{example}

\begin{defn}
Suppose that $\Sigma'$ covers $\Gamma$.
	We define the functor 
	\begin{align*}
		\cDef_{\Sigma',\Gamma}(A)&= \{\alpha \in \bigoplus_{(\rho, \bu)\in \Gamma}\check{C}^0(\sV'_{\rho,\bu},\KK)\otimes \mfm_{A} : \mfocp(\alpha)=0 \} /\sim
	\end{align*}
	where $\alpha=\beta$ in $\cDef_{\Sigma',\Gamma}(A)$  if and only if there exists \[\gamma\in \bigoplus_{(\rho,\bu)\in \Gamma}\check{C}^0(\sU',\CO(D_{\rho}))_{\bu}\otimes \mfm_{A}\] such that 
	\[\iota(\gamma)\odot \mfp(s(\alpha)) = \mfp(s(\beta)). \]
	Here $\sU'=\{U_\sigma\}_{\sigma\in\Sigma'_{\max}}$.
	This functor is defined on morphisms in the obvious way.
\end{defn}

\begin{thm}\thlabel{thm:samehull}
	Let $\Sigma$ be a simplicial fan and suppose that $\Sigma'$ covers $\Gamma$.	
	Then there are smooth maps $\rF_{\mcL_\Gamma}\to \cDef_\Sigma$ and $\rF_{\mcL_\Gamma}\to \cDef_{\Sigma', \Gamma}$ that induce isomorphisms on tangent spaces. In particular, if $\rT^1 \cDef_\Sigma$ is finite dimensional, then $\cDef_{\Sigma}$ and $\cDef_{\Sigma',\Gamma}$ have the same hull.
\end{thm}
\begin{proof}
	We have the natural map of functors $f_1: \rF_{\mcL_\Gamma,\sU}\to \rF_{\bigoplus \CO(D_{\rho})}$ induced from the injection $\mcL_\Gamma \to \bigoplus \CO(D_{\rho})$ and the cover $\sU$. Consider the open cover $\sU'=\{U_\sigma\}_{\sigma\in\Sigma_{\max}'}$ of $X_{\Sigma'}\subseteq X_\Sigma$. Similar to the construction of the deformation functor $\rF_{\mcL_\Gamma}=\rF_{\mcL_\Gamma,\sU}$  using the \v{C}ech complex $\check{C}^{\bullet}(\sU,\mcL_\Gamma)$, we can define the deformation functor $\rF_{\mcL_\Gamma,\sU'}$  using the \v{C}ech complex $\check{C}^{\bullet}(\sU',\mcL_\Gamma)$.  The natural map on \v{C}ech complexes induces a morphism of functors $f_2:\rF_{\mcL_\Gamma,\sU}\to \rF_{\mcL_\Gamma,\sU'}$.

	To prove the theorem, we will show that $f_1$ and $f_2$ are smooth maps with isomorphisms on the tangent spaces, and there are isomorphisms $g_1$ and $g_2$ as in the following diagram:
	
	\[
	\begin{tikzcd}
		& \rF_{\bigoplus \CO(D_{\rho})} \arrow[r,"g_1","\cong"swap ]  & \cDef_{\Sigma} \\
		\rF_{\mcL_\Gamma,\sU} \arrow[ru,"f_1"] \arrow[rd,"f_2"] & &\\
		& \rF_{\mcL_\Gamma,\sU'} \arrow[r,"g_2","\cong"swap ] & \cDef_{\Sigma',\Gamma}.
	\end{tikzcd}
	\]

	First, consider the map $f_1$. By construction
	\[ \rT^1 \rF_{\mcL_\Gamma}= \bigoplus_{(\rho,\bu)\in \Gamma} \check{H}^1(\sU,\CO(D_{\rho})_\bu)=\bigoplus_{(\rho,\bu) \in\Sigma(1)\times M} \check{H}^1(\sU,\CO(D_{\rho}))_{\bu}= \rT^1 \rF_{\bigoplus \CO(D_{\rho})} \]
	and 
	\[\bigoplus_{(\rho,\bu)\in \Gamma} \check{H}^2(\sU,\CO(D_{\rho})_\bu) \to \bigoplus_{(\rho,\bu) \in\Sigma(1)\times M} \check{H}^2(\sU,\CO(D_{\rho}))_{\bu} \]
	is an inclusion. Thus, by \thref{standardsmooth} the map $f_1$ is smooth. 

	Now consider $f_2$.  The maps between the tangent and obstruction spaces are isomorphisms since $\Sigma'$ covers $\Gamma$, so $f_2$ is also smooth by \thref{standardsmooth}.

	We know that $\cDef_\Sigma$ is the functor $\widehat{\rF}_{\bigoplus \CO(D_{\rho})}$.	Likewise, a straightforward adaptation of  \thref{exactseqlocal} implies that $\cDef_{\Sigma',\Gamma}$ is the functor $\widehat\rF_{\mcL_\Gamma}$ with respect to the open cover $\sU'$. The isomorphisms $g_1$ and $g_2$ thus follows from \thref{compa2}.
\end{proof}

From the above theorem, we can reduce the number of maximal cones that need to be considered. In the combinatorial deformation equation, we also need to compute $\mfoc(\alpha)_{\sigma\tau}$ for every pair $\sigma,\tau$ of maximal cones. The following cocycle property and \thref{prop:closure} reduce the number of cases that need to be calculated. Let $\bigwedge^2 \Sigma_{\max}$ be the set consisting of size two subsets of  $\Sigma_{\max}$.

\begin{defn}
	We say a set $\D \subseteq \bigwedge^2 \Sigma_{\max}$ has the \emph{cocycle property}  if for $\{\sigma,\kappa\}, \{\tau,\kappa\}\in \D$ with $\sigma\cap \tau \subseteq \kappa$, it follows that $\{\sigma,\tau\}\in \D$. For $\D\subseteq \bigwedge^2 \Sigma_{\max}$, we define $\overline{\D}$ to be the smallest set containing $\D$ that has the cocycle property.
\end{defn}

\begin{prop}\thlabel{prop:closure}
	Let $\Sigma$ be a simplicial fan, and let $\D\subseteq \bigwedge^2 \Sigma_{\max}$ be such that $\overline{\D}= \bigwedge^2 \Sigma_{\max}$. Then a cocycle $\omega \in \check{Z}^1(\sV_{\rho,\bu},\KK)$ is determined by $\{\omega_{\sigma\tau}\ |\ \{\sigma,\tau\} \in \D\}$. 
\end{prop}
\begin{proof}
	Let $\D_0=\D$ and define $\D_i$ recursively by
	\[\D_{i+1}=\D_{i}\cup \left\{ \{\sigma,\tau\}\ |\ \{\tau,\kappa\},\{\sigma,\kappa\} \in \D_{i} \; \textrm{with} \; \sigma\cap \tau \subseteq \kappa \right \}.\]
	Since $\bigwedge^2 \Sigma_{\max}$ is finite, after finitely many steps we obtain $\D_{m}=\overline{\D}$. We will prove by induction that for every $\{\sigma,\tau\}\in \D_i$, $\omega_{\sigma\tau}$ is determined by \[\{\omega_{\sigma'\tau'}\ | \{\sigma',\tau'\} \in \D\}.\]
	Clearly, for $i=0$, this is true. Suppose that  $\{\sigma,\tau\}\in \D_i$ for some $i\geq 1$. If $\{\sigma,\tau\}\in \D_{i-1}$ the statement is true by induction. If $V_{\rho,\bu}\cap \sigma \cap  \tau=\emptyset$, then $\omega_{\sigma\tau}=0$. Thus, we may assume that $V_{\rho,\bu}\cap \sigma \cap  \tau\not=\emptyset$. Since $\omega \in \check{Z}^1(\sV_{\rho,\bu},\KK)$ and 
	$V_{\rho,\bu}\cap \sigma \cap\tau\cap  \kappa\not=\emptyset$, we have $\omega_{\sigma\tau}=-\omega_{\tau\kappa}+\omega_{\sigma\kappa}$. By the induction hypothesis, $\omega_{\tau\kappa}$ and $\omega_{\sigma\kappa}$ are already determined. Thus, the statement is true by induction.
\end{proof}

In situations where $\Sigma$ has reasonable geometry, there is a canonical choice of $\D$. For a cone $\tau\in \Sigma$, define
\[
	\Star(\tau,\Sigma)=\{\sigma\in\Sigma\ |\ \tau\subseteq \sigma\}.
\]
We say that $\Star(\tau,\Sigma)$ is \emph{connected in codimension one} if for any two maximal cones $\sigma,\sigma'\in \Star(\tau,\Sigma)$, there is a sequence of maximal cones $\tau_1=\sigma,\tau_2,\ldots,\tau_k=\sigma'$ such that $\tau_i\in \Star(\tau,\Sigma)$ and $\tau_i,\tau_{i+1}$ intersect in a common facet.

\begin{prop}\thlabel{prop:codimone}
	Let $\Sigma$ be a simplicial fan and let $\D\subseteq \bigwedge^2\Sigma_{\max}$ consist of those pairs of cones that intersect in a common facet. Suppose that for every $\tau\in\Sigma$, $\Star(\tau,\Sigma)$ is connected in codimension one. Then $\overline{\D}= \bigwedge^2 \Sigma_{\max}$. In particular, a cocycle $\omega \in \check{Z}^1(\sV_{\rho,\bu},\KK)$ is determined by $\{\omega_{\sigma\tau}\ |\ \{\sigma,\tau\} \in \D\}$. 
\end{prop}
\begin{proof}
	We will show that $\overline{\D}= \bigwedge^2 \Sigma_{\max}$; the second claim then follows from \thref{prop:closure}. More specifically, we will show that for $\{\sigma,\sigma'\}\in \bigwedge^2 \Sigma_{\max}$, $\{\sigma,\sigma'\}\in \overline{\D}$. We will induct on the dimension of $\tau=\sigma\cap\sigma'$. If $\sigma,\sigma'$ intersect in a common facet, then we are done by definition of $\D$. 
Otherwise, suppose that we have shown all pairs intersecting in a face of dimension larger than $\dim \tau$ belong to $\overline \D$.

Fixing $\tau$, we now show that any pair of maximal cones $\sigma,\sigma'$ from $\Star(\tau,\Sigma)$ belongs to $\overline \D$. For this, we induct on the length of a sequence $\tau_1=\sigma,\tau_2,\ldots,\tau_k=\sigma'$ in $\Star(\tau,\Sigma)$ connecting $\sigma,\sigma'$ in codimension one. If $k=2$, then again we are done by the definition of $\D$. For $k>2$, by induction we have that $\{\sigma,\tau_2\}$ and $\{\tau_2,\sigma'\}$ belong to $\overline{\D}$. Since $\sigma\cap\sigma'=\tau\subseteq \tau_2$, it follows that $\{\sigma,\sigma'\}\in\overline{\D}$. The claim now follows by induction.
\end{proof}
\begin{rem}
	It is straightforward to verify that if a simplicial fan $\Sigma$ is combinatorially equivalent to a fan with convex support, it satisfies the hypotheses of \thref{prop:codimone}.
\end{rem}

It follows from  \thref{thm:samehull} and \thref{prop:closure} that in order to compute the hull of $\Def_{X_{\Sigma}}$, we can use the functor $\cDef_{\Sigma',\Gamma}$  for some $\Sigma'$ covering $\Gamma$, and, when computing obstructions, choose $\D \subseteq \bigwedge^2 \Sigma'_{\max}$ such that $\overline \D=\bigwedge^2 \Sigma'_{\max}$. We will apply these constructions in several examples in \S\ref{sec:rank3ctd}.

\begin{example}\label{example:D}
	Let $\Sigma'$ be the fan from Example \ref{example:Gamma}.
	Define the set \[\D=\Big \{ \{\sigma_1,\sigma_2\}, \{\sigma_2,\sigma_3\}, \{\sigma_3,\sigma_4\}, \{\sigma_4,\sigma_1\} \Big\}.  \]
	We have $\{\sigma_i,\sigma_{i+1}\}, \{\sigma_{i+1},\sigma_{i+2}\} \in \D $ with 
	$\sigma_i\cap \sigma_{i+2}\subseteq \sigma_{i+1}$ for $i=1,2$. Thus, $\{\sigma_1,\sigma_{3}\}$ and $\{\sigma_2,\sigma_{4}\}$ are in $\overline{\D}$.
	It follows that $\overline{\D}= \bigwedge^2 \Sigma'_{\max}$.                                       

In fact, $\D$ is the set of all pairs of maximal cones intersecting in a common facet. Since
$\Sigma'$ has the property that $\Star(\sigma,\Sigma')$ is connected in codimension one for all $\sigma$, \thref{prop:codimone} guarantees that 
$\overline{\D}= \bigwedge^2 \Sigma'_{\max}$. 
\end{example}
\begin{rem}
When $\Sigma$ is the inner normal fan of some polytope $P$ (i.e.~$X_\Sigma$ is a projective toric variety), the interpretation of 
$\check{C}^i(\sV_{\rho,\bu},\mfm_A)$ for $i=0,1$ as simplicial cochains on the nerve $\nerve(\sV_{\rho,\bu})$ becomes especially attractive in light of Propositions \ref{prop:closure} and \ref{prop:codimone}. Indeed, for a ray $\rho\in\Sigma(1)$ and a degree $\bu\in M$, let $P_{\rho,\bu}$ consist of the union of those faces of $P$ on which the vertices of $V_{\rho,\bu}$ obtain their minimum. We may view the vertices of $\nerve(\sV_{\rho,\bu})$ as the vertices of $P_{\rho,\bu}$. Moreover, pairs of maximal cones of $\Sigma$ intersecting in a common facet that meets $V_{\rho,\bu}$ correspond to edges of $P_{\rho,\bu}$.
Thus, by Proposition \ref{prop:closure} any $1$-cocycle on $\nerve(\sV_{\rho,\bu})$ is determined by its values on edges of $P_{\rho,\bu}$ and we may view the computations involved in solving the combinatorial deformation equation as taking place in simplicial cochain complexes for the one-skeletons of the $P_{\rho,\bu}$.
\end{rem}

\subsection{Unobstructedness}\label{sec:unobstructed}
In this section, we will
provide a sufficient criterion for a toric variety to have unobstructed deformations.

\begin{thm}\thlabel{thm:unobstructed}
	Let $X=X_{\Sigma}$ be a complete toric variety that is smooth in codimension 2 and $\QQ$-factorial in codimension 3. Let
	\begin{equation*}
	\A=\{(\rho,\bu)\in\Sigma(1)\times M\ |\  \widetilde{H}^0(V_{\rho,\bu},\KK)\not=0\}
\end{equation*}
and let $\Gamma$ be as in Definition \ref{defn:A}.
	If ${H}^{1}(V_{\rho,\bu},\KK)=0$ for all pairs $(\rho,\bu)\in \Gamma$ satisfying $\rho(\bu)=-1$, then $X_{\Sigma}$ is unobstructed.
In particular, setting 
	\begin{align*}
		\B:= \left\{ (\rho,\bu+\bv)\in \Sigma(1)\times M\ \Big|
			(\rho,\bu)\in \mathcal{A};\  
			\bv\in \displaystyle\sum_{(\rho',\bu')\in \mathcal{A}} \ZZ_{\geq0}\cdot \bu' 
		\right\},
	\end{align*}
if ${H}^{1}(V_{\rho,\bu},\KK)=0$ for all pairs $(\rho,\bu)\in \B$ satisfying $\rho(\bu)=-1$, then $X_{\Sigma}$ is unobstructed.
\end{thm}	
\begin{proof}
By \thref{mildsing}, after possibly replacing $\Sigma$ be any simplicial subfan of $\Sigma$ containing all three-dimensional cones, we may assume that 
$\Def_X\cong \cDef_\Sigma$.
By \thref{thm:samehull}, the functors $\cDef_{\Sigma}$ and $\cDef_{\Gamma,\Sigma}$ have the same hulls. But an obstruction space for $\cDef_{\Gamma,\Sigma}$ is given by 
\[
	\bigoplus_{(\rho,\bu)\in\Gamma} H^1(V_{\rho,\bu},\KK),
\]
see Lemma \ref{lemma:obstruction}. Moreover, by \thref{cor:vanishdegree}, we know that for any $(\rho,\bu)\in \Gamma$,  $\rho(\bu)\not=-1$ implies ${H}^{1}(V_{\rho,u},\KK)=0$ and the first claim of the theorem follows.
To show the second claim, observe that $\B$ contains $\Gamma$.
\end{proof}

To illustrate the significance of the above theorem, we provide an example of an unobstructed toric threefold whose unobstructedness does not follow by degree reasons alone.

\begin{example}\label{ex:unobstructed2}
	We consider the smooth toric threefold from Example \ref{ex:unobstructed1}. 
	The set $\A$ from \thref{thm:unobstructed} consists of exactly
\[(\rho_3,(1,0,0)),(\rho_{11},(0,0,1)),(\rho_{10},(-1,0,-1)),(\rho_{10},(-1,0,0)).\]
As noted earlier, we also have $H^1(V_{\rho,\bu},\KK)=0$ except for $(\rho,\bu)=(\rho_{0},(0,0,-1))$.

There are infinitely many positive integer combinations of degrees \[(1,0,0),(0,0,1),(-1,0,-1),(-1,0,0)\] that sum to $(0,0,-1)$. For example,
\[(0,0,-1)=(-1,0,-1)+(1,0,0).\]
Hence, we cannot conclude by degree reasons alone that $X_\Sigma$ is unobstructed. However, since $\rho_0$ does not appear in any element of $\A$, $(\rho_0,(0,0,-1)$ is not in $\B$ and we conclude that $X_\Sigma$ is in fact unobstructed.
\end{example}

\section{Examples}\label{sec:Examples}
Throughout this section, we will assume that $\Sigma$ is a smooth complete fan.
\subsection{Primitive collections and rigidity}\label{sec:prim}

The data of a fan can be provided by specifying the ray generators and listing the maximal cones. Instead of specifying the maximal cones we can describe the fan using the notion of \emph{primitive collections}:

\begin{defn}[{\cite[Definition 2.6]{bat}}]
	A subset $\mcP\subseteq \Sigma(1)$ is a \emph{primitive collection} if the elements of $\mcP$ do not belong to a common cone in $\Sigma$, but the elements of every proper subset of $\mcP$ do.
\end{defn}
	Let $\mcP=\{\rho_1,\ldots,\rho_k\}$ be a primitive collection and let $\sigma \in \Sigma$ be the unique cone such that $n_{\rho_1}+\cdots+n_{\rho_k}$ lies in the relative interior of $\sigma$. Then there is a relation
	\[n_{\rho_1}+\cdots+n_{\rho_k}= \sum_{\rho\in \sigma\cap \Sigma(1)} c_{\rho}n_{\rho} \]
	where $c_{\rho}>0$ for all $\rho \in \sigma\cap \Sigma(1)$. This is the \textit{primitive relation} associated to $\mcP$, see {\cite[Definition 2.1.4]{bat1}}. The \emph{degree} of $\mcP$ is the integer \[\deg(\mcP)= k-\sum c_{\rho}.\]

\begin{thm}\thlabel{lemma:rigid}
	Let $\Sigma$ be a smooth complete fan. If $\deg(\mcP)>0$ for every primitive collection of cardinality 2, then $X_{\Sigma}$ is rigid.
\end{thm}
\begin{proof}
	We will show that $H^1(X_{\Sigma},\T_{X_{\Sigma}})=0$ in this case. According to Proposition \ref{prop:cohom2}, we have
	\begin{equation*}
		H^1(X_{\Sigma},\T_{X_{\Sigma}})\cong\bigoplus_{\substack{\rho\in\Sigma(1), \bu\in M \\ \rho(\bu)=-1}} \widetilde{H}^{0}(V_{\rho,\bu},\KK).
	\end{equation*}
	For any $\rho\in \Sigma(1)$ and $\bu\in M$, $\widetilde{H}^{0}(V_{\rho,\bu},\KK)\neq0$ only when the simplicial complex $V_{\rho,\bu}$ is disconnected. Hence, it is enough to show that $V_{\rho,\bu}$ is connected for all $\rho\in \Sigma(1)$ and $\bu\in M$.

	Assuming that  $V_{\rho,\bu}\neq \emptyset$, let $\rho_{\min}\in \Sigma(1)$ be any ray not equal to $\rho$ such that $\rho_{\min}(\bu)$ is minimal. Then for any $\rho'\neq \rho$ with $\rho'(\bu)<0$, we claim that   $\rho'$ shares a cone of $\Sigma$ with $\rho_{\min}$. Indeed, if not, $\mcP=\{\rho_{\min},\rho'\}$ forms a primitive collection, and since 
$\deg(\mcP)>0$ we have only two possibilities:
\[n_{\rho_{\min}}+n_{\rho'}=0; \quad n_{\rho_{\min}}+n_{\rho'}= n_{\rho''}.\]
The first case is impossible since both $\rho_{\min}(\bu)$ and $\rho'(\bu)$ are less than zero. The second case is impossible since it would follow that $\rho''(\bu)<\rho_{\min}(\bu)$, contradicting the choice of $\rho_{\min}$. This implies the claim, and the connectedness of $V_{\rho,\bu}$ follows.	
\end{proof}

\begin{rem}
	A smooth toric variety $X_{\Sigma}$ is Fano (respectively weak Fano), if and only if $\deg(\mcP)>0$ (respectively $\deg(\mcP)\geq0$) for every primitive collection (see \cite[Proposition 2.3.6]{bat1}). As a corollary, we obtain the well-known result that every smooth toric Fano variety is rigid (see \cite[Proposition 4.2]{BB}).
	We also obtain the rigidity for smooth toric weak Fano varieties with no degree zero primitive collections of cardinality 2 (cf. \cite[Corollary 1.7]{ilten1} for a similar result).
\end{rem}

It is well-known that $\PP^n$ is the only smooth complete toric variety with Picard rank 1. Since it is Fano, it is rigid by the previous remark. Thus we may focus on smooth complete toric varieties with higher Picard rank.

\subsection{Picard rank two}\label{sec:ranktwo}
In this section, we will prove that every smooth complete toric variety with Picard rank 2 is unobstructed. Any $n$-dimensional toric variety $X$ of this type can be expressed as 
\[X \cong \PP(\CO_{\PP^{s}}\oplus \CO_{\PP^{s}}(a_1)\oplus \cdots \oplus \CO_{\PP^{s}}(a_r)),\]
for $r,s\geq 1, r+s=n$ and $0\leq a_1\leq \cdots \leq a_r$, see \cite{kleinschmidt}, \cite[Theorem 7.3.7]{cls}.

A fan $\Sigma$ with $X=X_\Sigma$ may be described as follows, see also \cite[Example 7.3.5]{cls}. Fix the lattice $N=\ZZ^n$, and consider the ray generators given by the columns of the following matrix:
\NiceMatrixOptions{code-for-first-col = \color{darkgray},code-for-first-row = \color{darkgray}}
\setcounter{MaxMatrixCols}{20}
	\begin{equation*}
		%\hspace*{-8em}
		A=
		%	\begin{pNiceMatrix}[small,first-row,first-col,margin,hvlines]
			\begin{pNiceMatrix}[small,first-row,first-col,margin]
				& \rho_1&\cdots&\rho_s&\rho_{s+1}&\cdots&\rho_n&\rho_{n+1}&\rho_{n+2}\\
				&&&&&&&&&\\
				&\Block{3-3}<\LARGE>{I_s} & & &\Block{3-3}<\LARGE>{  0}&&&0&-1\\
				&&&&&&&\vdots&\vdots\\
				&&&&&&&0&-1\\\\
				&\Block{3-3}<\LARGE>{ 0} & & &\Block{3-3}<\LARGE>{I_{r}}&&&-1&a_1\\
				&&&&&&&\vdots&\vdots\\
				&&&&&&&-1&a_r\\
			\end{pNiceMatrix}
			\label{matrix}
		\end{equation*}
		The primitive collections for $\Sigma$ are $\{\rho_1,\ldots,\rho_{s},\rho_{n+2}\}$ and $\{\rho_{s+1},\ldots,\rho_{n},\rho_{n+1}\}$.

\begin{thm}\thlabel{thm:ranktwo}
	Let 
	\[X= \PP(\CO_{\PP^{s}}\oplus \CO_{\PP^{s}}(a_1)\oplus \cdots \oplus \CO_{\PP^{s}}(a_r)).\]
	Then:
	\begin{enumerate}[label={(\roman*)}]
		\item $X$ is rigid if and only if $s>1$, or $s=1$ and $a_r\leq1$. If $s=1$, then 
		\[\dim_{\KK} H^1(X,\T_{X})= \sum_{j=1}^{r} \max \{a_{j}-1,0\}+ \sum_{k<j} \max \{a_j-a_k-1,0\};\] \label{thm:ranktwo:claim1}
		\item $X$ is unobstructed.
		\label{thm:ranktwo:claim2}
	\end{enumerate}
\end{thm}
\begin{proof}

	 We will make use of the isomorphism
	 \begin{equation*}
	 	H^k(X,\T_{X})\cong\bigoplus_{\substack{\rho\in\Sigma(1), \bu\in M \\ \rho(\bu)=-1}} \widetilde{H}^{k-1}(V_{\rho,\bu},\KK),
	 \end{equation*}
	 for $k=1,2$ (see Proposition \ref{prop:cohom2}).

	 We will first show claim \ref{thm:ranktwo:claim1}.  
	 Suppose that $H^1(X,\T_{X})\neq0$. In this case, there must exist a ray $\rho\in \Sigma(1)$ with $\rho(\bu)=-1$ and a primitive collection $\{\rho',\rho''\}$ with exactly two rays which are distinct from $\rho$ and satisfy $\rho'(\bu)<0, \rho''(\bu)<0$. Consequently, either $r$ or $s$ must equal one. However, if $r=1$, then $\rho'(\bu)=-\rho''(\bu)$. Hence, it must be the case that $s=1$ and the rigidity for $s>1$ follows.
	 
	 Assume that $s=1$.  From the primitive collections, it follows that the simplicial complex $V_{\rho,\bu}$ is disconnected only when it consists of the two vertices $n_{\rho_{1}}$ and $n_{\rho_{n+2}}$. This occurs if and only if the following conditions on the ray-degree pairings are satisfied:  there exists $j \in \{1,\ldots,r+1\}$ such that $\rho_{s+j}(\bu)=-1$ and 
	 \begin{align*}
		\rho_1(\bu)<0, \quad \rho_{n+2}(\bu)<0, \quad \rho_{s+i}(\bu)\geq 0 \quad \textrm{ for all} \; i \in \{1,\ldots,r+1\}\setminus \{\ell\}. 
	\end{align*}
	If $a_{r}\geq 2$, then for $\rho=\rho_{s+r}=\rho_{n}$ and $\bu=(-1,0,\ldots,0,-1)$, the above set of conditions is satisfied, and thus we obtain  $V_{\rho,\bu}$ is disconnected. In fact, all degree-ray pairs $(\rho,\bu)$ satisfying the above conditions are given by choosing  $\rho=\rho_{s+j}$ for $ j\in \{1,\ldots,r\}$ and taking $\bu$ of one of the following two forms:
	\begin{align*}
		\bu &= d\cdot e_1-e_{s+j}, 
		&& \mathrm{where}\ 1-a_j\le d\le -1,\\
		\bu &= d\cdot e_1+e_{s+k}-e_{s+j}, 
		&& \mathrm{where}\ 1-a_j+a_k\le d\le -1,\ \ 1\leq k<j.
	\end{align*}
	
	For each $j=1,\ldots,r$, taking the direct sum over all $\bu$ of the first form, we obtain
	\[
	\dim_{\KK}(\bigoplus
	\widetilde{H}^{0}(V_{\rho_{s+j},\bu},\KK))
	=
	\max\{a_j-1,0\},
	\]
	and for each pair $1\le k<j\le r$, taking the direct sum over all $\bu$ of the second form, we obtain
	\[
	\dim_{\KK}(\bigoplus
	\widetilde{H}^{0}(V_{\rho_{s+j},\bu},\KK))
	=
	\max\{a_j-a_k-1,0\}.
	\]
	Summing over all $j$ and all pairs $(k,j)$ with $1\le k<j\le r$, we obtain
	\[
	\dim_{\KK} H^1(X,\T_X)
	=
	\sum_{j=1}^{r}\max\{a_j-1,0\}
	+
	\sum_{k<j}\max\{a_j-a_k-1,0\},
	\]
	and claim~\ref{thm:ranktwo:claim1} follows.

	To prove claim \ref{thm:ranktwo:claim2}, we show that if $X$ is not rigid, then $H^2(X,\T_X)=0$. Hence, we may again assume that $s=1$. It is suffices to show that for all ray-degree pairs $(\rho,\bu)$, every connected component of $V_{\rho,\bu}$ is contractible to a point.
	Given that
	\[n_{\rho_{s+1}}+\ldots+n_{\rho_{n+1}}=0,\]
	for every $\bu \in M$ there exists at least one $j\in \{1,\ldots,r+1\}$ such that $\rho_{s+j}(\bu)\geq 0$. Let $\sigma$ be the cone in $\Sigma$ generated by $\{\rho_{s+1}, \ldots, \rho_{n+1}\} \setminus \{\rho_{s+j}\}$. Then $V_{\rho,\bu}$ is the join of the simplicial complexes
	 $\left( V_{\rho,\bu}\cap \{n_{\rho_1},n_{\rho_{n+2}}\} \right )$ and $\left (V_{\rho,\bu}\cap \sigma \right)$.
	Since $V_{\rho,\bu}\cap \sigma$ is either empty or a simplex, it follows that every connected component of $V_{\rho,\bu}$ is contractible to a point. This completes the proof of claim \ref{thm:ranktwo:claim2}.
\end{proof}

\subsection{Split \texorpdfstring{$\mathbb{P}^1$}{P1}-bundles over Hirzebruch surfaces}\label{sec:rank3}
As shown in \cite{FPR}, there exists a  $\PP^1$-bundle over the second Hirzebruch surface $\FF_2$ that exhibits quadratic obstructions. In fact, Picard rank three toric threefolds represent the minimal cases in terms of both dimension and Picard rank where obstructions can occur. This is because smooth complete toric varieties of dimension at most 2 (\cite[Corollary 1.5]{ilten1}) and those with Picard rank at most 2 (\thref{thm:ranktwo}) are unobstructed. 

In this section and the next, we examine toric threefolds that are $\PP^1$-bundles over the $e$th Hirzebruch surface \[\FF_e=\PP(\CO_{\PP^1}\oplus \CO_{\PP^1}(e)),\]
where $e\geq 0$. Every toric threefold of Picard rank 3 is either of this form, or the blowup of a toric threefold of Picard rank 2 in a point or a $\PP^1$ (see \cite[Chapter \Romannum{7}, Theorem 8.2]{Ewald}, \cite[Theorem 0.1]{RL}). 

 Any toric $\PP^1$-bundle over $\FF_e$ can be expressed as
\[ X\cong \PP(\mathcal{O}_{\mathbb{F}_e}\oplus \mathcal{O}_{\mathbb{F}_e}(aF+bH)),\]
where $a,b\in \ZZ$, and $F$ and $H$ respectively represent the classes in $\Pic(\FF_e)$ of the fiber and $\CO_{\mathbb{F}_e}(1)$ in the $\PP^1$-bundle fibration of $\FF_e$ over $\PP^1$. 
In other words, 
\[\Pic(\mathbb{F}_e)= \mathbb{Z}F\oplus \mathbb{Z}H \quad  \mathrm{with}\quad  F^2=0,\; F\cdot H=1,\;  H^2=e.\]
Since we are considering $X$ up to isomorphism, we can take $e,b\geq0$, see  e.g.~\cite[Theorem 3.6]{robins}. Under this assumption, the fan $\Sigma$ from Example \ref{example:Gamma} describes $X$, that is, $X=X_\Sigma$. The primitive collections are given by $\{\rho_1,\rho_3\}$, $\{\rho_2,\rho_4\}$, and $\{\rho_5,\rho_6\}$. From this, we obtain the primitive relations:

\begin{equation}\label{eq:primrel}
	\begin{aligned}
	n_{\rho_5}+ n_{\rho_6}&=0;\\
	n_{\rho_2}+ n_{\rho_4}&= b\cdot n_{\rho_5};\\
	n_{\rho_1}+n_{\rho_3}&= 
	\begin{cases}
		e \cdot n_{\rho_2} + a\cdot n_{\rho_5} \; \textrm{if} \; a\geq 0 \\
		e \cdot n_{\rho_2} + (-a)\cdot  n_{\rho_6} \; \textrm{if} \; a< 0. \\
	\end{cases} 
\end{aligned}
\end{equation}

As a first step for computing the hull for several examples of $\PP^1$-bundles over $\FF_e$, we will describe $H^1(X,\T_{X})$ and $H^2(X,\T_{X})$ in the following two lemmas.

\begin{lemma}\thlabel{lemma:H2deg}
	Let $X= \PP(\mathcal{O}_{\mathbb{F}_e}\oplus \mathcal{O}_{\mathbb{F}_e}(aF+bH))$ with $e,b\geq0$. Then
	\[H^2(X,\T_{X})\cong \bigoplus_{\substack{\rho\in\Sigma(1), \bu\in M \\ \rho(\bu)=-1}} H^1(V_{\rho,\bu},\KK)\ \textrm{is non-zero}\]
	if and only if  $b\geq2$ and $(b-1)e+a\geq 2$. Moreover, $H^1(V_{\rho,\bu},\KK)\not=0$ if and only if $\rho=\rho_5$ and $\bu=(x,y,-1)$ satisfies
	\begin{align*}
		-b+1\leq y \leq -1, \quad	ey-a+1\leq x \leq -1.
	\end{align*}
\end{lemma}

\begin{proof}
	Clearly, when $V_{\rho,\bu}$ is the simple cycle formed by the vertices $\rho_1,\rho_2,\rho_3,\rho_4$, we have $H^1(V_{\rho,\bu},\KK)\not=0$. 
	We claim that for any other simplicial complex of the form $V_{\rho,\bu}$, its connected components are contractible to a point. Indeed,  any such simplicial complex $V_{\rho,\bu}$ can be expressed as the join of the simplicial complexes:
	\[V_1= V_{\rho,\bu} \cap \{n_{\rho_1},n_{\rho_3}\}, \quad V_2=V_{\rho,\bu} \cap \{n_{\rho_2},n_{\rho_4}\}, \quad V_3=V_{\rho,\bu} \cap \{n_{\rho_5},n_{\rho_6}\}.\]
	From \eqref{eq:primrel}, $V_3$ is either empty or one of the vertices. In the latter case, $V_{\rho,\bu}$ is contractible to that vertex. Therefore, we may restrict to the case $V_{\rho,\bu}$ is the join of the simplicial complexes $V_1$ and $V_2$. In that case, the connected components of $V_{\rho,\bu}$ are contractible unless $V_1$ and $V_2$ each consist of two vertices.

	This scenario occurs if and only if the following conditions on ray-degree pairing are satisfied: either $\rho_5(\bu)=-1$ or $\rho_6(\bu)=-1$, and 
	\[\rho_i(\bu)<0 \quad \textrm{for all} \quad i=1,2,3,4.\] 
	There is no $\bu$ satisfying the second case (this is immediate from primitive relations \eqref{eq:primrel}). Since $\rho_5(\bu)=-1$, we can represent $\bu=(x,y,-1)$. Under this representation, the inequalities can be expressed as:
	\begin{align*}
		ey-a+1&\leq x \leq -1\\
		-b+1&\leq y \leq -1.
	\end{align*}
This system of inequalities has an integral solution (namely $(-1,-b+1,-1)$) if and only if $b\geq 2$ and $(b-1)e+a\geq 2$. 
\end{proof}

\begin{lemma}\thlabel{lemma:H1deg}
	Let $X\cong \PP(\mathcal{O}_{\mathbb{F}_e}\oplus \mathcal{O}_{\mathbb{F}_e}(aF+bH))$ with $e,b\geq0$. Then
	\[\widetilde H^0(V_{\rho,\bu},\KK) \not=0\]
	in exactly the following mutually disjoint cases:
	\begin{enumerate}[label=Type \Roman*:,  left=0pt, align=left]
		\item The pair $(\rho,\bu)$ takes the form $(\rho_2,(x,-1,0))$, where $x$ satisfies
		\[-e+1\leq x \leq -1.\]
		Type \Romannum{1} occurs if and only if $e\geq 2$;
		
		\item The pair $(\rho,\bu)$ takes the form $(\rho_6,(x,y,1))$, where $x,y$ satisfies
		 \begin{align*}
		 	0\leq y\leq b, \quad 	ey+a+1 \leq x \leq -1.
		 \end{align*}
		 Type \Romannum{2} occurs if and only if $a\leq -2$;
		
		\item The pair $(\rho,\bu)$ takes the form $ (\rho_5,(x,y,-1))$, where $x,y$ satisfies
		\begin{align*}
			-b+1\leq y\leq -1, \quad 0 \leq x \leq ey-a.
		\end{align*}
		Type \Romannum{3} occurs if and only if $e+a\leq0$ and $b\geq 2$;
		
		\item The pair $(\rho,\bu)$ takes the form $(\rho_5,(x,0,-1))$, where $x$ satisfies
		\begin{align*}
			b=0, \quad -a+1 &\leq x \leq -1.
		\end{align*}
		Type \Romannum{4} occurs if and only if
		$a\geq 2$ and $b=0$.
	\end{enumerate}
\end{lemma}

\begin{proof}
	Any simplicial complex $V_{\rho,\bu}$ can be expressed as the join of the simplicial complexes:
	\[V_1= V_{\rho,\bu} \cap \{n_{\rho_1},n_{\rho_3}\}, \quad V_2=V_{\rho,\bu} \cap \{n_{\rho_2},n_{\rho_4}\}, \quad V_3=V_{\rho,\bu} \cap \{n_{\rho_5},n_{\rho_6}\}.\]
	If at least two of $V_1,V_2$ and $V_3$ are non-empty, then $V_{\rho,\bu}$ is connected. According to \eqref{eq:primrel}, $V_3$ is either empty or consists of single vertex. Hence, $V_{\rho,\bu}$ can only have more than one connected component when it consists of exactly two vertices, either $\{n_{\rho_1}, n_{\rho_3}\}$ or $\{n_{\rho_2}, n_{\rho_4}\}$. There are four distinct scenarios in which this occurs:
	
	\noindent
	\textbf{Type \Romannum{1}:}
	\[\rho_2(\bu)=-1, \quad \rho_1(\bu)<0, \quad \rho_3(\bu)<0, \quad \rho_4(\bu)\geq 0, \quad \rho_5(\bu)\geq0 ,\quad \rho_6(\bu)\geq 0. \]
	Thus, we can represent $\bu=(x,-1,0)$, and the above 
	inequalities lead to
	\[-e+1\leq x \leq -1.\]
	It is immediate to see that Type \Romannum{1} occurs if and only if $e\geq 2$.
		
	\noindent
	\textbf{Type \Romannum{2}:}
	\[\rho_6(\bu)=-1, \quad \rho_1(\bu)<0, \quad \rho_3(\bu)<0, \quad \rho_2(\bu)\geq 0, \quad \rho_4(\bu)\geq0 . \]
	Thus, we can represent $\bu=(x,y,1)$, and the above inequalities lead to
	\begin{align*}
		0\leq y\leq b, \quad 	ey+a+1 \leq x \leq -1.
	\end{align*}
	If Type \Romannum{2} occurs, then $\bu=(-1,0,1)$ is always included, which requires $a\leq -2$.

	\noindent
	\textbf{Type \Romannum{3}:}
	\[\rho_5(\bu)=-1, \quad \rho_2(\bu)<0, \quad \rho_4(\bu)<0, \quad \rho_1(\bu)\geq 0, \quad \rho_3(\bu)\geq0 . \]
	Thus, we can represent $\bu=(x,y,-1)$, and the above inequalities lead to
	\begin{align*}
		-b+1\leq y\leq -1, \quad 0 \leq x \leq ey-a.
	\end{align*}
	If Type \Romannum{3} occurs, then $\bu=(0,-1,-1)$ is always included, which requires $e+a\leq0$.
	
	\noindent
	\textbf{Type \Romannum{4}:}
	\[\rho_5(\bu)=-1, \quad \rho_1(\bu)<0, \quad \rho_3(\bu)<0, \quad \rho_2(\bu)\geq 0, \quad \rho_4(\bu)\geq0 . \]
	Thus, we can represent $\bu=(x,0,-1)$, and the above inequalities lead to
	\begin{align*}
		b=0, \quad -a+1 &\leq x \leq -1.
	\end{align*}
	It is immediate to see that Type \Romannum{4} occurs if and only if $a\geq 2$ and $b=0$.
	
\end{proof}

Using \thref{lemma:H2deg} and \thref{lemma:H1deg}, we will apply the unobstructedness result (\thref{thm:unobstructed}) based on the ray-degree pairs.

\begin{lemma}\thlabel{lemma:possibleobs}
	Let \(X\cong \PP(\mathcal{O}_{\mathbb{F}_e}\oplus \mathcal{O}_{\mathbb{F}_e}(aF+bH))\) with \(e, b \geq 0\). Then $X$ is unobstructed, except possibly in the following cases:
	\begin{enumerate}[label={(\roman*)}]
		\item $e=1$, $a\leq -2$ and $b\geq 3-a$; \label{case:possible1}
		\item $e\geq 2$, $a\leq -e$ and $b \geq  1+\dfrac{2-a}{e}$. \label{case:possible2} 
	\end{enumerate}

\end{lemma}

\begin{proof}
	Suppose first that $e=0$. By \thref{lemma:H2deg} $H^2(X,\T_X)\not=0$ if and only if $a\geq 2$ and $b\geq 2$. In this case, \thref{lemma:H1deg} implies that $H^1(X,\T_X)=0$. Therefore, when $e=0$,  $X$ is always unobstructed.
	
	Now assume that $e=1$. Then by \thref{lemma:H2deg} $H^2(X,\T_X)\not=0$ if and only if 
	\begin{align}\label{eq:ineq1}
	 b \geq \max \left\{2, 3-a \right\}.
	\end{align}
	 In this case, Type \Romannum{1} and Type \Romannum{4} of \thref{lemma:H1deg} do not occur.
	By \thref{thm:unobstructed}, to have obstructions, we must have Type \Romannum{2} and Type \Romannum{3} elements of \thref{lemma:H1deg}. Hence, we have $a\leq-2$ and combining this with \eqref{eq:ineq1} yields \ref{case:possible1}.

	Now assume that $e\geq 2$. By \thref{lemma:H2deg} $H^2(X,\T_X)\not=0$ if and only if  \begin{align}\label{eq:ineq2}
		 b \geq \max \left\{2, 1+\dfrac{2-a}{e}\right\}.
	\end{align}
	In this case, Type \Romannum{4} of \thref{lemma:H1deg} does not occur and Type \Romannum{1} does occur. By \thref{thm:unobstructed}, Type \Romannum{3} must occur. 
	However, the occurrence of both Type \Romannum{1} and Type \Romannum{3} together implies Type \Romannum{2} also occurs. Hence, we have $a\leq-e$; combining this with \eqref{eq:ineq2} yields \ref{case:possible2}.
\end{proof}

\begin{figure}
	\begin{tikzpicture}[scale=0.6]
		% Draw axis
		\draw[dotted,->] (-8,0) -- (5,0) node[right] {$x$};
		\draw[dotted,->] (0,-9) -- (0,2) node[above] {$y$};
		
		% in diagram: e=1,a=-4, b=10
	
		% T^2 degree	
			
		\node[anchor=south west,font=\tiny] at (-7,-7) {$H^2(X,\T_X)_{\bu}\not=0$};
		
		\filldraw[fill=red!20, draw=red] (-4,-9)--(-1,-9)--(-1,-6)--cycle;
		
		\draw[fill=red] (-1,-6) circle (2pt);
		\node[anchor=south east,font=\tiny] at (-1,-6) {$(-1,a-2)$};

		\draw[fill=red] (-1,-9) circle (2pt);
		\node[anchor=north east,font=\tiny] at (0,-9) {$(-1,-b+1)$};

		\draw[fill=red] (-4,-9) circle (2pt);
		\node[anchor=north east,font=\tiny] at (-4,-9) {$(-a-b+2,-b+1)$};
		
		\node[anchor=south west,font=\tiny] at (-3,-8.5) {$z=-1$};

		%\draw[dash pattern=on 1pt off 1pt,red] (-1,-6)--(0,-5); 
		
		%T^1 degree
			
		% Case 2
		
		\node[anchor=south west,font=\tiny] at (-6,1) {Type \Romannum{2} };

		\filldraw[fill=blue!20, draw=blue] (-3,0)--(-1,0)--(-1,2)--cycle;

		\draw[fill=blue] (-1,0) circle (2pt);
		\node[anchor=north west,font=\tiny] at (-2,0) {$(-1,0)$};
		
		\draw[fill=blue] (-3,0) circle (2pt);
		\node[anchor=north east,font=\tiny] at (-3,0) {$(a+1,0)$};
		
		\draw[fill=blue] (-1,2) circle (2pt);
		\node[anchor=south east,font=\tiny] at (-1,2) {$(-1,-a-2)$};
		
		\node[anchor=north west,font=\tiny] at (-2.3,0.8) {$z=1$};

		% case 3

		\node[anchor=south west,font=\tiny] at (2,-3) {Type \Romannum{3}};

		\filldraw[fill=black!20, draw=black] (0,-1)--(3,-1)--(0,-4)--cycle;
		
		\draw[fill=black] (0,-1) circle (2pt);
		\node[anchor=south west,font=\tiny] at (0,-1) {$(0,-1)$};
		
		\draw[fill=black] (3,-1) circle (2pt);
		\node[anchor=south west,font=\tiny] at (3,-1) {$(-a-1,-1)$};
		
		\draw[fill=black] (0,-4) circle (2pt);
		\node[anchor=north west,font=\tiny] at (0,-4) {$(0,a)$};
		
		\node[anchor=south east,font=\tiny] at (2,-2) {$z=-1$};
	
		%\draw[dash pattern=on 1pt off 1pt] (-2,-6)--(0,-4); 

	\end{tikzpicture}
	\caption{The projections onto the $xy$-coordinates of the degrees $\bu$ where $H^1(X,\T_{X})_{\bu}\not=0$ and $H^2(X,\T_{X})_{\bu}\not=0$ for the case $e=1$ with $a \leq -2$ and $b\geq -a+3$.}
	\label{fig:type1}
\end{figure}

\begin{figure}
	\begin{tikzpicture}[scale=0.67]
		% Draw axis
		\draw[dotted,->] (-8,0) -- (5,0) node[right] {$x$};
		\draw[dotted,->] (0,-6) -- (0,4) node[above] {$y$};
		
		% Define variable values
		\def\e{4}
		\def\a{-9}
		\def\b{6}
		
		% compute and define
		\def\eeta{-3}	
		\def\xxi{-2}
		\def\mmu{1}

		% T^2 degree	
		
		\node[anchor=south west,font=\tiny] at (-10,-4) {$H^2(X,\T_X)_{\bu}\not=0$};
		
		% define coordinates
		
		\coordinate (A1) at ({(-\b+1)*\e-\a+1}, {-\b+1});
		\coordinate (A2) at(-1,{-\b+1});
		\coordinate (A3) at(-1,\eeta);
		\coordinate (A4) at ({\eeta*\e-\a+1},\eeta);

		\filldraw[fill=red!20, draw=red] (A1)--(A2)--(A3)--(A4)--cycle;
		
		\draw[fill=red] (A3) circle (2pt);
		\node[anchor=north west,font=\tiny] at (A3) {$(-1,\eta)$};

		\draw[fill=red] (A4) circle (2pt);
		\node[anchor=south east,font=\tiny] at (A4) {$(\eta \cdot e-a+1,\eta)$};
		
		\draw[fill=red] (A2) circle (2pt);
		\node[anchor=north east,font=\tiny] at (A2) {$(-1,-b+1)$};

		\draw[fill=red] (A1) circle (2pt);
		\node[anchor=north east,font=\tiny] at ($(A1)+(5,0)$) {$((-b+1)\cdot e-a+1,-b+1)$};
		
		\draw[dash pattern=on 1pt off 1pt,red] (A3)--(-1,{(\a-2)/\e})--(A4);
		
		\node[anchor=north east,font=\tiny] at ($(A2)+(-1,1)$) {$z=-1$};

		%T^1 degree
		
		%Type 1

		\node[anchor=south west,font=\tiny] at (-8,-1.5) {Type \Romannum{1}};
		
		% define coordinates
		
		\coordinate (B1) at (-1, -1);
		\coordinate (B2) at({-\e+1},-1);
		
		\draw[thick,green] (B2) -- (B1); % y = -1
		
		\draw[fill=green] (B1) circle (2pt);
		\node[anchor=south east,font=\tiny] at (0,-1) {$(-1,-1)$};
		
		\draw[fill=green] (B2) circle (2pt);
		\node[anchor=south east,font=\tiny] at (B2) {$(-e+1,-1)$};
		
		\node[anchor=north east,font=\tiny] at (-1.5,-1) {$z=0$};

		% Type 2
		
		\node[anchor=south west,font=\tiny] at (-8,2) {Type \Romannum{2} };

		% define coordinates
		
		\coordinate (C1) at ({\a+1}, 0);
		\coordinate (C2) at(-1,0);
		\coordinate (C3) at(-1,\mmu);
		\coordinate (C4) at ({\mmu*\e+\a+1},\mmu);

		\filldraw[fill=blue!20, draw=blue] (C1)--(C2)--(C3)--(C4)--cycle;

		\draw[fill=blue] (C2) circle (2pt);
		\node[anchor=south west,font=\tiny] at (-1,0) {$(-1,0)$};
		
		\draw[fill=blue] (C1) circle (2pt);
		\node[anchor=south east,font=\tiny] at (C1) {$(a+1,0)$};
		
		\draw[fill=blue] (C3) circle (2pt);
		\node[anchor=south west,font=\tiny] at (C3) {$(-1,\mu)$};

		\draw[fill=blue] (C4) circle (2pt);
		\node[anchor=south east,font=\tiny] at (C4) {$(\mu\cdot e+a+1,\mu)$};

		\draw[dash pattern=on 1pt off 1pt,blue] (C3)--(-1,{(-\a-2)/\e})--(C4);

		\node[anchor=north east,font=\tiny] at ($(C2)+(-1,1)$) {$z=1$};

		% case 3
		
		\node[anchor=south west,font=\tiny] at (2,-4) {Type  \Romannum{3} };

		% define coordinates
		
		\coordinate (D1) at (0, -1);
		\coordinate (D2) at ({-\e-\a},-1);
		\coordinate (D3) at({\xxi*\e-\a},\xxi);
		\coordinate (D4) at (0,\xxi);

		\filldraw[fill=black!20, draw=black] (D1)--(D2)--(D3)--(D4)--cycle;
		
		\draw[fill=black] (D1) circle (2pt);
		\node[anchor=south west,font=\tiny] at (D1) {$(0,-1)$};
		
		\draw[fill=black] (D2) circle (2pt);
		\node[anchor=south west,font=\tiny] at ($(D2)+(-1.5,0)$) {$(-e-a,-1)$};
		
		\draw[fill=black] (D4) circle (2pt);
		\node[anchor=north east,font=\tiny] at (D4) {$(0,\xi)$};
		
		\draw[fill=black] (D3) circle (2pt);
		\node[anchor=north west,font=\tiny] at (D3) {$(\xi \cdot e-a,\xi)$};
		
		\draw[dash pattern=on 1pt off 1pt] (D4)--(0,{\a/\e})--(D3)--cycle;
		
		\node[anchor=north east,font=\tiny] at ($(D1)+(2,-0.2)$) {$z=-1$};

	\end{tikzpicture}
	\caption{The projections onto the $xy$-coordinates of the degrees $\bu$ where $H^1(X,\T_{X})_{\bu}\not=0$ and $H^2(X,\T_{X})_{\bu}\not=0$ for the case $e\geq 2$, $a\leq -e$, and $b\geq 1+\dfrac{-a+2}{e}$.}
	\label{fig:type2}
\end{figure}

For the cases outlined in \thref{lemma:possibleobs}, we will review the inequalities that define  the sets of $\bu$ with $H^1(X,\T_{X})_{\bu}\not=0$ and $H^2(X,\T_{X})_{\bu}\not=0$,  as some of the inequalities are redundant.

\vspace{1em}
\noindent
\textbf{Case: $e=1$, $a\leq -2$, and  $b\geq 3-a$. }
In Figure \vref{fig:type1}, we depict the projections onto the $xy$-coordinates of the degrees $\bu$ where $H^1(X,\T_{X})_{\bu}\not=0$ (split into types \Romannum{2} and \Romannum{3}) and $H^2(X,\T_{X})_{\bu}\not=0$. 
In this situation, the degrees with $H^2(X,\T_{X})_{\bu}\not=0$ satisfy the inequalities
\[y\geq -b+1, \quad x\leq -1, \quad x\geq y-a+1.\]
This defines a triangular region, which degenerates to a single point if $b=3-a$. 

Type \Romannum{1} and Type \Romannum{4} of $H^1(X,\T_X)_{\bu}$ do not occur. The region for Type \Romannum{2} is defined by the inequalities
\[y\geq 0, \quad x\leq -1, \quad x\geq y+a+1 \]
resulting in a triangular region, which degenerates to a single point if $a=-2$. The region for Type \Romannum{3} is defined by the inequalities
\[y\leq -1, \quad x\geq 0, \quad x\leq y-a \]
resulting in a triangular region.
We always have the following relation among the degrees:
\begin{equation}\label{eq:relation1}
	(-1,a-2,-1)=(0,a,-1)+(0,-2,-1)+(-1,0,1).
\end{equation}
For the degree on the left hand side, $H^2(X,\T_{X})$ is non-zero, while $H^1(X,\T_X)$ is non-vanishing for each of the degrees on the right hand side.

\vspace{1em}
\noindent
\textbf{Case: $e\geq2$, $a\leq -e$, and  $b\geq 1+\dfrac{2-a}{e}$. }
In Figure \vref{fig:type2}, we depict the projections onto the $xy$-coordinates of the degrees $\bu$ where $H^1(X,\T_{X})_{\bu}\not=0$ (split into types \Romannum{1}, \Romannum{2}, and \Romannum{3}) and $H^2(X,\T_{X})_{\bu}\not=0$. 
In this situation, the degrees with $H^2(X,\T_{X})_{\bu}\not=0$ satisfy the inequalities
\[y\geq -b+1, \quad x\leq -1, \quad x\geq ey-a+1. \]
If $(a-2)/e$ is an integer, the convex hull of these degrees is a triangle, which degenerates to a single point when $b= 1+(2-a)/e$. If $(a-2)/e$ is not an integer, we define \[\eta=\left\lfloor \dfrac{a-2}{e} \right\rfloor,\] and obtain a trapezoidal region (which degenerates to a line when $\eta=-b+1$) by eliminating the top portion of the triangle where there are no lattice points. 

Type \Romannum{4} of $H^1(X,\T_X)_{\bu}$ does not occur.
The region for Type \Romannum{1} is defined by the (in)equalities
\[y=-1, \quad x\leq -1, \quad x\geq -e+1, \]
and forms a line. The region for Type \Romannum{2} is defined by the inequalities
\[y\geq 0, \quad x\leq -1, \quad x\geq ey+a+1. \]
If $(-a-2)/e$ is an integer, we obtain a triangle, which degenerates to a single point when $a=-2$. If $(-a-2)/e$ is not an integer, we define
\[\mu= \left\lfloor \dfrac{-a-2}{e} \right\rfloor\]
and obtain a trapezoidal region (which degenerates to a line  when $\mu=0$) by eliminating the top portion of the triangle where there are no lattice points. 

The region for Type \Romannum{3} is defined by the inequalities
\[y\leq -1, \quad x\geq 0, \quad x\leq ey-a. \]
If $a/e$ is an integer, we obtain a triangle, which degenerates to a single point when $a=-e$. If $a/e$ is not an integer, we define
\[\xi=\left\lceil \dfrac{a}{e} \right\rceil\]
and obtain a trapezoidal region (which degenerates to a line when $\xi=-1$) by eliminating the bottom portion of the triangle where there are no lattice points.

If $a\not\equiv 1 \bmod e$, then $\xi-\eta=1$, and we have the following relation:
\begin{equation}\label{eq:relation2}
	(-1,\eta,-1)=(0,\xi,-1)\;+\;(-1,-1,0).
\end{equation}
If $a\equiv 1 \bmod e$, then $\xi-\eta=2$, and we have the following relation:
\begin{equation}\label{eq:relation3}
	(-1,\eta,-1)= (1,\xi,-1)\;+\;2(-1,-1,0).
\end{equation}
In both cases, for the degree on the left hand side, $H^2(X,\T_{X})$ is non-zero, while $H^1(X,\T_X)$ is non-vanishing for each of the degrees on the right hand side.

\subsection{Split \texorpdfstring{$\mathbb{P}^1$}{P1}-bundles: obstruction computations}\label{sec:rank3ctd}
Our next goal is to compute the hull for several examples. Instead of directly working with $\cDef_{\Sigma}$, we will apply the results from \S \ref{sec:rc}. Specifically, we will identify $\Sigma'$ and $\Gamma$ such that $\cDef_{\Sigma',\Gamma}$ has the same hull as $\cDef_{\Sigma}$. Additionally, we will determine a minimal $\D$ such that $\overline{\D}= \bigwedge^2 \Sigma'$.

\begin{lemma}\thlabel{lemma:SigmaD}
	Let $X_{\Sigma}\cong \PP(\mathcal{O}_{\mathbb{F}_e}\oplus \mathcal{O}_{\mathbb{F}_e}(aF+bH))$ with $e,b\geq0$. Then the fan $\Sigma'$ with the maximal cones
	\begin{align*}
		\sigma_1&= \cone(\rho_1,\rho_4,\rho_5), \quad \sigma_2= \cone(\rho_1,\rho_2,\rho_5), \\
		\sigma_3&= \cone(\rho_2,\rho_3,\rho_5), \quad  \sigma_4= \cone(\rho_3,\rho_4,\rho_5).
	\end{align*}
	 covers $\Gamma$. Moreover, the set
	 
	 \[\D=\Big \{ \{\sigma_1,\sigma_2\}, \{\sigma_2,\sigma_3\}, \{\sigma_3,\sigma_4\}, \{\sigma_4,\sigma_1\} \Big\} \]
		satisfies $\overline{\D}= \bigwedge^2 \Sigma'_{\max}$.                                        
	
\end{lemma}
\begin{proof}
From \thref{lemma:H1deg}, we observe that $(\rho,\bu)\in \A$ must be of the form 
\[ \Big( \rho_2, (*,*,0) \Big) \quad  \textrm{or} \quad \Big(\rho_5, (*,*,-1) \Big) \quad  \text{or} \quad \Big(\rho_6, (*,*,1) \Big).\] 
The proof then follows from Example \ref{example:Gamma} and Example \ref{example:D}.
\end{proof}
 
For notational simplicity, we will index cochains by the numbers $1,2,3,4$ instead of the cones $\sigma_1,\ldots,\sigma_4$, e.g. for $1\leq i,j\leq 4$ we write $\alpha_i$ instead of $\alpha_{\sigma_i}$ and $\omega_{ij}$ instead of $\omega_{\sigma_i\sigma_j}$.
Recall that in \S\ref{sec:cde}, we needed to make several choices while computing the hull of $\cDef_{\Sigma}$ (or similarly $\cDef_{\Sigma',\Gamma}$). 
In the examples below, we will make these choices as follows:
\begin{enumerate}
	\item We always use the graded lexicographic local monomial order.
	\item For constructing the map $\psi$, we choose the ordering of cones in $\Sigma'$ as $\sigma_1<\sigma_2<\sigma_3<\sigma_4$.
	\item For Type \Romannum{1} and Type \Romannum{2} first order deformations, $V_{\rho,\bu}$ is given by two vertices $n_{\rho_1}$ and $n_{\rho_3}$. When choosing a basis of $\rT^1 \cDef_{\Sigma',\Gamma}$, we will always take the connected component $n_{\rho_3}$.
		\item   For Type \Romannum{3} first order deformations, $V_{\rho,\bu}$ is given by two vertices $n_{\rho_2}$ and $n_{\rho_4}$. 
When choosing a basis of $\rT^1 \cDef_{\Sigma',\Gamma}$, we will always take the connected component $n_{\rho_2}$.
	\item For ray-pairs such that $H^1(V_{\rho,\bu},\KK)\not=0$, we know $V_{\rho,\bu}$ is the simple cycle with vertices $n_{\rho_1},n_{\rho_2}, n_{\rho_3},n_{\rho_4}$ ordered cyclically.
	In these cases, we choose the cocycle
	$ \omega \in \check{Z}^1(\sV'_{\rho,\bu},\KK)$ by setting
	\[\omega_{34}=1, \quad \omega_{12}=\omega_{23}=\omega_{41}=0. \]
	Then $\psi(\omega)_i=0$ for $i=1,\ldots,4$, and thus
	$d(\psi(\omega))=0$ as required.
\end{enumerate}
With the above choices, it follows that $\alpha_{1}^{(r)}$ is always zero.

%Examples cases

\begin{figure}[htbp]
	\centering
	
	%subfigure A
	\begin{subfigure}{0.45\textwidth}
		\centering
		\fbox{
			\begin{tikzpicture}[scale=0.6]
				\foreach \x in {-2,-1,0,1,2} {
					\foreach \y in {-4,-3,-2,-1,0} {
						\draw (\x,\y) circle (2pt);
					}
				}
				
				% Draw axis
				\draw[dotted,->] (-2.5,0) -- (2.5,0) node[right] {$x$};
				\draw[dotted,->] (0,-4.5) -- (0,1) node[above] {$y$};
				
				\filldraw[blue] (-1,0) circle (2pt); \node[anchor=south east,font=\tiny] at (-1,0) {$\bu_4$};			
				
				\filldraw[fill=black!20, draw=black]	
				(0,-1)-- (1,-1)--(0,-2)--cycle;

				\filldraw[black] (0,-1) circle (2pt); \node[anchor=south west,font=\tiny] at (0,-1) {$\bu_1$};
				
				\filldraw[black] (1,-1) circle (2pt); \node[anchor=south west,font=\tiny] at (1,-1) {$\bu_2$};
				
				\filldraw[black] (0,-2) circle (2pt); \node[anchor=north west,font=\tiny] at (0,-2) {$\bu_3$};

				\filldraw[red] (-1,-4) circle (2pt);
				\node[anchor=north east,font=\tiny] at (-1,-4) {$\bv$};

			\end{tikzpicture}
		}
		\caption{Example \ref{example:1,-2,5}}
		\label{fig:1,-2,5}
	\end{subfigure}\hspace{0.1cm}
	%subfigure B
	\begin{subfigure}{0.45\textwidth}
		\centering
		\fbox{	
			\begin{tikzpicture}[scale=0.6]
				\foreach \x in {-3,-2,-1,0,1} {
					\foreach \y in {-4,-3,-2,-1,0} {
						\draw (\x,\y) circle (2pt);
					}
				}

				% Draw axis
				\draw[dotted,->] (-3,0) -- (1.5,0) node[right] {$x$};
				\draw[dotted,->] (0,-4.5) -- (0,1) node[above] {$y$};
				
				\filldraw[blue] (-1,0) circle (2pt); \node[anchor=south east,font=\tiny] at (-1,0) {$\bu_4$};			
				
				\filldraw[blue] (-2,0) circle (2pt); \node[anchor=south east,font=\tiny] at (-2,0) {$\bu_5$};			
				
				\draw[blue]	(-1,0)-- (-2,0);
				
				\filldraw[fill=black!20, draw=black]	
				(0,-1)-- (1,-1);

				\filldraw[black] (0,-1) circle (2pt); \node[anchor=south west,font=\tiny] at (0,-1) {$\bu_1$};
				
				\filldraw[black] (1,-1) circle (2pt); \node[anchor=south west,font=\tiny] at (1,-1) {$\bu_2$};
				
				\filldraw[green] (-1,-1) circle (2pt); \node[anchor=south west,font=\tiny] at (-1,-1) {$\bu_3$};
				
				\draw[red]	(-1,-3)-- (-2,-3);
				
				\filldraw[red] (-1,-3) circle (2pt);
				\node[anchor=north east,font=\tiny] at (-1,-3) {$\bv_1$};
				
				\filldraw[red] (-2,-3) circle (2pt);
				\node[anchor=north east,font=\tiny] at (-2,-3) {$\bv_2$};

			\end{tikzpicture}
		}
		\caption{Example \ref{example:2,-3,4}}
		\label{fig:2,-3,4}

	\end{subfigure}
	%subfigure C
	\begin{subfigure}{0.45\textwidth}
		\centering
		\fbox{	
			\begin{tikzpicture}[scale=0.6]
				\foreach \x in {-3,-2,-1,0,1} {
					\foreach \y in {-2,-1,0,1} {
						\draw (\x,\y) circle (2pt);
					}
				}

				% Draw axis
				\draw[dotted,->] (-3,0) -- (1.5,0) node[right] {$x$};
				\draw[dotted,->] (0,-2.5) -- (0,1.5) node[above] {$y$};
				
				\filldraw[blue] (-1,0) circle (2pt); \node[anchor=south west,font=\tiny] at (-1,0) {$\bu_3$};			
				\filldraw[blue] (-2,0) circle (2pt);
				\node[anchor=south west, font=\tiny] at (-2,0) {$\bu_5$};
				\filldraw[blue] (-3,0) circle (2pt); \node[anchor=south east, font=\tiny] at (-2.5,0) {$\bu_{7}$};

				\filldraw[black] (0,-1) circle (2pt); \node[anchor=north west,font=\tiny] at (0,-1) {$\bu_1$};
				
				\filldraw[black] (1,-1) circle (2pt); \node[anchor=north west,font=\tiny] at (1,-1) {$\bu_2$};

				\filldraw[green] (-1,-1) circle (2pt); \node[anchor=north west,font=\tiny] at (-1,-1) {$\bu_4$};
				
				\filldraw[green] (-2,-1) circle (2pt); \node[anchor=north east,,font=\tiny] at (-2,-1) {$\bu_6$};
				
				\filldraw[red] (-1,-2) circle (2pt);
				\node[anchor=north east,font=\tiny] at (-1,-2) {$\bv$};

				\draw[blue] (-3,0)-- (-1,0);
				\draw[green] (-2,-1)-- (-1,-1); 
				\draw[black] (0,-1)-- (1,-1);
			\end{tikzpicture}
		}
		\caption{Example \ref{example:3,-4,3}}
		\label{fig:3,-4,3}
	\end{subfigure}
	\hspace{.1cm}	
	%subfigure D
	\begin{subfigure}{0.45\textwidth}
		\centering
		\fbox{
			\begin{tikzpicture}[scale=0.6]
				\foreach \x in {-5,-2,-1,0,1,2} {
					\foreach \y in {-2,-1,0,1} {
						\draw (\x,\y) circle (2pt);
					}
				}
				
				% Draw axis
				\draw[dotted,->] (-5,0) -- (2.5,0) node[right] {$x$};
				\draw[dotted,->] (0,-2) -- (0,1.5) node[above] {$y$};
				
				\filldraw[blue] (-1,0) circle (2pt); \node[anchor=south west,font=\tiny] at (-1,0) {$\bu_3$};			
				\filldraw[blue] (-2,0) circle (2pt);
				\node[anchor=south west, font=\tiny] at (-2,0) {$\bu_5$};
				\filldraw[blue] (-5,0) circle (2pt); \node[anchor=south west, font=\tiny] at (-5,0) {$\bu_{2e-1}$};
				
				\filldraw[black] (0,-1) circle (2pt); \node[anchor=north west,font=\tiny] at (0,-1) {$\bu_1$};

				\filldraw[green] (-1,-1) circle (2pt); \node[anchor=north west,font=\tiny] at (-1,-1) {$\bu_2$};

				\filldraw[green] (-2,-1) circle (2pt); \node[anchor=north west,,font=\tiny] at (-2,-1) {$\bu_4$};
				\filldraw[green] (-5,-1) circle (2pt); \node[anchor=north west,font=\tiny] at(-5,-1) {$\bu_{2e-2}$};

				\filldraw[red] (-1,-2) circle (2pt);
				\node[anchor=north west,font=\tiny] at (-1,-2) {$\bv_1$};

				\filldraw[red] (-2,-2) circle (2pt); \node[anchor=north west,font=\tiny] at (-2,-2) {$\bv_2$};
				
				\filldraw[red] (-5,-2) circle (2pt); \node[anchor=north west,font=\tiny] at (-5,-2) {$\bv_{e-1}$};
				
				\draw[blue] (-2,0)-- (-1,0);
				\draw[blue,dashed] (-5,0)-- (-2,0);
				\draw[green] (-2,-1)-- (-1,-1); 
				\draw[green,dashed] (-5,-1)-- (-2,-1);
				\draw[red] (-2,-2)-- (-1,-2);
				\draw[red,dashed] (-5,-2)-- (-2,-2);
			\end{tikzpicture}
		}
		\caption{Example \ref{example:e,-e,b}}
		\label{fig:e,-e,b}
	\end{subfigure}

	\caption{ The $xy$ coordinates of degrees at which $H^1(X,\T_{X})_{\bu}\not=0$ and $H^2(X,\T_{X})_{\bv}\not=0$ 
	}
	\label{fig:examples}

\end{figure}

\begin{example}[Case $(e,a,b)=(1,-2,5)$]\label{example:1,-2,5}
	Here, we explain how to use the setup in \S\ref{sec:cde} and \S\ref{sec:rc} to compute the hull of $\Def_X$ for 
	\[X=\PP(\mathcal{O}_{\mathbb{F}_1}\oplus \mathcal{O}_{\mathbb{F}_1}(-2F+5H)).\]
	This example possesses the minimal $\rT^1$ dimension among those $\PP^1$-bundles in \thref{lemma:possibleobs} that have no quadratic obstructions, but may have third order obstructions. We will compute its hull, showing that it indeed has a third order obstruction.

	By \thref{lemma:H1deg}, $H^1(X_{\Sigma},T_{X_{\Sigma}})$ is non-zero only in the following degrees:
	\[\bu_{1}=(0,-1,-1), \quad \bu_{2}=(1,-1,-1), \quad \bu_{3}=(0,-2,-1), \quad \bu_{4}=(-1,0,1).\]
	By \thref{lemma:H2deg}, $H^2(X_{\Sigma},T_{X_{\Sigma}})$ is non-zero only in the degree 
	\[\bv=(-1,-4,-1).\] 
	See Figure \vref{fig:1,-2,5} for an illustration.
	The only non-negative integer combinations of the $\bu_i$ giving $\bv$ are of the form
	\[\bv= 2\bu_{3}+\bu_{4}=\bu_{1}+\bu_{2}+\bu_{3}+2\bu_{4}=2\bu_{1}+2\bu_{2}+3\bu_{4}.\]
	Therefore, the hull must have the form
	\[\KK[[t_{1},\ldots,t_{4}]]\Big/J; \quad J=\langle a_1t_{3}^2t_{4} +a_2t_{1}t_{2}t_{3}t_{4}^2+a_3  t_{1}^2t_{2}^2t_{4}^3 \rangle\]
	for some $a_1,a_2,a_3\in \KK$. Here, $t_i$ is the deformation parameter with degree $\bu_i$.
	
	In Table \vref{table:1,-2,5} we summarize the deformation equation computations. See the ancillary files \cite{ancil} for code carrying out these computations in Macaulay2 \cite{M2}. The monomials highlighted in blue are those for which obstructions are possible. 
	The deformation data is interpreted as follows. For each $\alpha_i^{(r)}$, the coefficient of a monomial $t^w$ is obtained by multiplying the entries in the $\sigma_i$ row by the coefficient of $t^w$ in the first column of the table. The expression for $\alpha_i^{(r)}$ is then the sum of these products, considering only monomials of total degree less than or equal to $r$. As explained in Example \ref{example:Gamma}, every element of $\Gamma$ must be of the form 
	\[ \Big( \rho_5, (*,*,-1) \Big) \quad  \textrm{or} \quad \Big(\rho_6, (*,*,1) \Big) \quad  \text{or} \quad \Big(\rho_5, (*,*,0) \Big) \quad  \text{or} \quad \Big(\rho_6, (*,*,0) \Big).\]
	This implies that only monomials of the form $t_1^{w_1}t_2^{w_2}t_3^{w_3}t_4^{w_4}$ with \[|w_1+w_2+w_3-w_4|\leq 1\]
	 will appear. Additionally, we restrict our attention to relevant monomials, see Remark \ref{rem:relevant}:
we only need to consider monomials dividing 	$t_{3}^2t_{4}$, $t_{1}t_{2}t_{3}t_{4}^2$, or $t_{1}^2t_{2}^2t_{4}^3$.
		
The obstruction data should be interpreted as follows. For the normal form of $(\mfoc(\alpha^{(r)})-\sum_{\ell=1}^q g_{\ell}^{(r)}\cdot \omega_\ell)_{ij}$ with respect to $\mfm\cdot J_{r}$,	the coefficient of $t^w\in \mfm_{r+1}$ is the product of the entries in the row labeled $\sigma_{i}\sigma_j$  with the coefficient of $t^w$ found in the first column of the table. We also list the coefficients of the $\gamma_{\ell}^{(r+1)}$ in the rightmost column of the table.

After computing $\alpha^{(7)}$, we obtain the deformation equation
	\begin{align*}
		\mfoc(\alpha^{(7)})&\equiv 0 \mod  \mfm^8+ \langle  t_{3}^2t_{4} -2t_{1}t_{2}t_{3}t_{4}^2+  t_{1}^2t_{2}^2t_{4}^3 \rangle.
	\end{align*}
	Consequently, we conclude that $a_1=a_3=1$ and $a_2=-2$.
	After applying the change of variables  $t_{3}'=t_{3}-t_{1}t_{2}t_{4}$, we observe that
	\[t_{3}'^2t_{4}= t_{3}^2t_{4} -2t_{1}t_{2}t_{3}t_{4}^2+  t_{1}^2t_{2}^2t_{4}^3. \]
	Thus, the hull of $\Def_X$ is given by
	\[
	\KK[[t_1,t_2,t_3',t_4]]/\langle t_3'^2 t_4 \rangle.
	\]
	We obtain that the spectrum of the hull has two irreducible components, both of dimension three. One is smooth and the other generically non-reduced of multiplicity two.
	\end{example}
	
	\begin{rem}
		The entries highlighted in red in Table~\ref{table:1,-2,5} and subsequent Tables \ref{table:3,4,-3} and \ref{table:case: e,-e,3} allow us to read off the nerves $\nerve(\sV(\rho_i,\bu))$ for various relevant choices of $\rho_i$ and $\bu$. Indeed, the red entries
	correspond to those cones $\sigma_j$ or pairs $\sigma_j\sigma_k$ whose intersection with $V_{\rho_i,\bu}$ is empty. Here, $\bu$ is the degree of the monomial in that row, and $i$ is the index with $f_i$ appearing in that row. (For rows with $(f_5-f_6)$, the complexes $V_{\rho_5,\bu}$ and $V_{\rho_6,\bu}$ coincide).
	Note that for the monomials $t_2t_4$ and $t_2^2t_4^2$ in Table~\ref{table:1,-2,5}, $\sigma_1$ also has trivial intersection with $V_{\rho_i,\bu}$.

	The red entries for deformation data and obstruction data are always necessarily zero. On the other hand, the non-red entries mark the vertices and edges of $\nerve(\sV_{\rho_i,\bu})$ (which also include the vertex $\sigma_1$ except for the cases of $t_2t_4$ and $t_2^2t_4^2$ in Table~\ref{table:1,-2,5}). 
	\end{rem}

\begin{table}
	\caption{Deformation data for the case $(e,a,b)=(1,-2,5)$}
	\label{table:1,-2,5}
	\scriptsize
	\begin{tabular}{|c|c|c|c|c|c|c|c|c|c|}
		\hline
		&\multicolumn{3}{c|}{\makecell{Deformation data}}& \multicolumn{4}{c|}{\makecell{Obstruction data}}& \makecell{$\gamma$} \\
		\hline
		\diagbox{$t^{w}$}{$\sigma$} & $\sigma_2$ & $\sigma_3$ & $\sigma_4$ & $\sigma_{1}\sigma_2$ & $\sigma_{2}\sigma_3$ & $\sigma_{3}\sigma_4$ & $\sigma_{4}\sigma_1$ &\\
		\hline
		$t_{1}\cdot f_5$   & 1 & 1 & 0 & \zero & 0 & \zero & 0&0 \\
		\hline
		$t_{2}\cdot f_5$   & 1 & 1 & 0 & \zero & 0 & \zero & 0&0 \\
		\hline
		$t_{3}\cdot f_5$   & 1 & 1 & 0 & \zero & 0 & \zero & 0&0 \\
		\hline
		$t_{4}\cdot f_6$   & 0 & 1 & 1 & 0 & \zero & 0 & \zero&0 \\
		\hline

		%second-order corrections
		\multicolumn{9}{|c|}{\textbf{2nd order}}\\
		\hline
		
		$t_{1}t_{4}\cdot (f_5-f_6)$  & 0 & -1/2 & \zero & 0& 1/2 & \zero & \zero&0 \\
		\hline
		$t_{2}t_{4}\cdot (f_5-f_6)$   & 0 & -1/2 & -1 & \zero & 1/2 & 1/2 & \zero&0 \\
		\hline
		$t_{3}t_{4}\cdot (f_5-f_6)$ & 0 & -1/2 & -1 & 0 & 1/2 & 1/2 & \zero&0 \\
		\hline

		%third-order corrections
		\multicolumn{9}{|c|}{\textbf{3rd order}}\\
		\hline
		$t_{1}^2t_{4}\cdot f_5$     	& 0 & -1/6 & 0  & 0 & 1/6 & \zero  & 0&0 \\
		\hline
		$t_{1}t_{2}t_{4}\cdot f_5$  	& 0 & -1/3 & 0  & \zero & 1/3 & \zero  & 0&0 \\
		\hline
		$t_{1}t_{3}t_{4}\cdot f_5$  	& 0 & -1/3 & 0  & 0 & 1/3 & \zero  & 0&0 \\
		\hline
		$t_{1}t_{4}^2\cdot f_6$     	& 0 & -1/6 & -1 & 0 & 1/6 &5/6& \zero&0 \\
		\hline
		$t_{2}^2t_{4}\cdot f_5$     	& 1 &  5/6 & 0  & \zero & 1/6 &5/6& 0&0 \\
		\hline
		$t_{2}t_{3}t_{4}\cdot f_5$  	& 2 &  5/3 & 0  & \zero & 1/3 &5/3& 0&0 \\
		\hline
		$t_{2}t_{4}^2\cdot f_6$     	& 0 & -1/6 & 0  & 0 & 1/6 & -1/6& \zero&0 \\
		\hline
\color{blue}{$t_{3}^2t_{4}\cdot f_5$}   & 0 & -1/6 & 0  & 0 & 1/6 &5/6& 0&1 \\
		\hline
		$t_{3}t_{4}^2\cdot f_6$  		& 0 & -1/6 & 0  & 0 & 1/6 & -1/6& \zero&0 \\
		\hline

		%fourth-order corrections
		\multicolumn{9}{|c|}{\textbf{4th order}}\\
		\hline
		$t_{1}^2t_{4}^2\cdot (f_5-f_6)$     & 0 & 1/12& \zero  & 0 &-1/12 & \zero & \zero&0 \\
		\hline
		$t_{1}t_{2}t_{4}^2\cdot (f_5-f_6)$  & 0 & 1/6 &1  & 0 & -1/6 & -5/6 & \zero&0 \\
		\hline
		$t_{1}t_{3}t_{4}^2\cdot (f_5-f_6)$  & 0 & 1/6 &1  & 0 & -1/6 & -5/6 & \zero&0 \\
		\hline
		$t_{2}^2t_{4}^2\cdot (f_5-f_6)$     & 0 &  -5/12&-1/2 & \zero & 5/12 & 1/12 & \zero&0 \\
		\hline
		$t_{2}t_{3}t_{4}^2\cdot (f_5-f_6)$  & 0 &  -5/6 & -1  & 0 & 5/6  & 1/6 & \zero&0 \\
		\hline
		
		%fifth-order corrections
		\multicolumn{9}{|c|}{\textbf{5th order}}\\
		\hline
		$t_{1}^2t_{2}t_{4}^2\cdot f_5$          & 0 & 1/10   & 0  & 0 & -1/10 & \zero    & 0&0 \\
		\hline
		$t_{1}^2t_{4}^3\cdot f_6$               & 0 & 1/30   & 1  & 0 & -1/30 &-29/30 & \zero&0 \\
		\hline
		$t_{1}t_{2}^2t_{4}^2\cdot f_5$          &-1 & -37/30 & 0  & \zero & 7/30 &-37/30& 0 &0 \\
		\hline
\color{blue}{$t_{1}t_{2}t_{3}t_{4}^2\cdot f_5$} & 0 & -7/15  & 0  & 0 & 7/15 &-37/15& 0&-2 \\
		\hline
		$t_{1}t_{2}t_{4}^3\cdot f_6$            & 0 &  1/15  & 0  & 0 & -1/15 & 1/15& \zero&0 \\
		\hline
		$t_{2}^2t_{4}^3\cdot f_6$               & 0 & -2/15  &-1/6& 0 & 2/15 &1/30& \zero&0 \\
		\hline
		
		%sixth-order corrections
		\multicolumn{9}{|c|}{\textbf{6th order}}\\
		\hline
		$t_{1}^2t_{2}t_{4}^3\cdot (f_5-f_6)$  & 0 & -1/20 & -1 & 0 & 1/20 & 19/20 & \zero&0 \\
		\hline
		$t_{1}t_{2}^2t_{4}^3\cdot (f_5-f_6)$  & 0 &37/60& 1& 0 &-37/60 &  -23/60 & \zero&0 \\
		\hline
		%seventh-order corrections
		\multicolumn{9}{|c|}{\textbf{7th order}}\\
		\hline
		\color{blue}{$t_{1}^2t_{2}^2t_{4}^3\cdot f_5$}  & 0 & 41/105 & 0 & 0 & -41/105 & 146/105 & 0&1 \\
		\hline		
	\end{tabular}
\end{table}

\begin{example}[Case $(e,a,b)=(2,-3,4)$]\label{example:2,-3,4}
	Here, we compute the hull of $\Def_X$ for 
	\[X=\PP(\mathcal{O}_{\mathbb{F}_2}\oplus \mathcal{O}_{\mathbb{F}_2}(-3F+4H)).\]
	 We will show that the hull  does not have any quadratic obstructions, but does have third order obstructions. We will utilize this example in the proof of \thref{thm:obstructed}. 
	
	By \thref{lemma:H1deg}, $H^1(X_{\Sigma},T_{X_{\Sigma}})$ is non-zero only in the following degrees:	 
	\begin{align*}
	\bu_1&=(0,-1,-1), & \bu_2&=(1,-1,-1), & \bu_3&=(-1,-1,0),\\
	\bu_4&=(-1,0,1), & \bu_5&=(-2,0,1).	
	\end{align*}
	By \thref{lemma:H2deg}, $H^2(X_{\Sigma},T_{X_{\Sigma}})$ is non-zero only in the degrees 
	\[ \bv_1=(-1,-3,-1),\quad \bv_2=(-2,-3,-1).\] 
	See Figure \vref{fig:2,-3,4} for an illustration.
	The only non-negative integer combinations are of the form
	\begin{align*}
		\bv_1&= \bu_2+2\bu_3=\bu_1+\bu_2+\bu_3+\bu_4=2\bu_2+\bu_3+\bu_5\\
		&= 3\bu_2+ 2\bu_5=   \bu_2+ 2\bu_1+2\bu_4= \bu_1+2\bu_2+\bu_4+\bu_5  \\
		\bv_2&= \bu_1+2\bu_3=\bu_1+\bu_2+\bu_3+\bu_5=2\bu_1+\bu_3+\bu_4\\
		&=3\bu_1+ 2\bu_4= \bu_1+2\bu_2+2\bu_5 =\bu_2+ 2\bu_1+\bu_4+\bu_5 
	\end{align*}
	Therefore the hull must have the form $\KK[[t_1,\ldots,t_5]]\Big/ \langle g_1,g_2 \rangle 
	$, where
	\begin{align*}
		g_1&=a_1t_2t_3^2+a_2t_1t_2t_3t_4+ a_3t_2^2t_3t_5+a_4t_2^3t_5^2+a_5t_1^2t_2t_4^2+a_6t_1t_2^2t_4t_5;\\
		g_2&=b_1t_1t_3^2+b_2t_1t_2t_3t_5+ b_3t_1^2t_3t_4+b_4t_1^3t_4^2+b_5t_2^2t_1t_5^2+b_6t_2t_1^2t_4t_5.
	\end{align*}
	for some $a_1,\ldots,a_6,b_1,\ldots,b_6\in \KK$.
	
	Doing a computation similar to Example \ref{example:1,-2,5}, we obtain
	\begin{align*}
		a_1=-b_1=a_4=-b_4=a_5=-b_5=1;\\
		-a_2=b_2=-a_3=b_3=a_6=-b_6=2,
	\end{align*}
	see \cite{ancil} for details.
	After applying the change of variables $t_3'=-t_3+t_1t_4+t_2t_5$ we observe that
	\begin{align*}
		t_2t_3'^2=g_1\qquad	t_1t_3'^2=-g_2.
	\end{align*}
	Thus, the hull of $\Def_X$ is given by
	\[
	\KK[[t_1,t_2,t_3',t_4,t_5]]/\langle t_2t_3'^2, t_1t_3'^2 \rangle.
	\]
	As in Example \ref{example:1,-2,5}, the spectrum of the hull has two irreducible components. One is smooth with dimension 3, and the other is a generically non-reduced component of multiplicity 2 and dimension 4.
\end{example}

\begin{example}[Case $(e,a,b)=(3,-4,3)$]\label{example:3,-4,3}
	Here we compute the hull of $\Def_X$ for 
	\[X=\PP(\mathcal{O}_{\mathbb{F}_3}\oplus \mathcal{O}_{\mathbb{F}_3}(-4F+3H)).\]
	We will see that the hull is irreducible and singular at the origin, and already determined by the quadratic obstructions. We note that these quadratic obstructions could have been computed using the methods of \cite{ilten3}.  

	By \thref{lemma:H1deg}, $H^1(X_{\Sigma},T_{X_{\Sigma}})$ is non-zero only in the following degrees:
	\begin{align*}
		\bu_1&=(0,-1,-1), &  \bu_2&=(1,-1,-1), & \bu_3&=(-1,0,1), &  \bu_4&=(-1,-1,0),\\
		\bu_5&=(-2,0,1), &  \bu_6&=(-2,-1,0), &  \bu_7&=(-3,0,1).
	\end{align*}
	By \thref{lemma:H2deg}, $H^2(X_{\Sigma},T_{X_{\Sigma}})$ is non-zero only in the degree 
	\[\bv=(-1,-2,-1).\] 
	See Figure \vref{fig:3,-4,3} for an illustration. The only non-negative integer combinations are of the form
	\[\bv= \bu_1+\bu_4=\bu_2+\bu_6=2\bu_1+\bu_3=2\bu_2+\bu_7=\bu_1+\bu_2+\bu_5.\]
	Therefore, the hull must have the form
	\[\KK[[t_1,\ldots,t_7]]\Big/J; \quad J=\langle a_1t_1t_4+ a_2t_2t_6+ a_3t_1^2t_3+a_4t_2^2t_7+a_5t_1t_2t_5 \rangle,\]
	for some $a_1,\ldots,a_5\in \KK$.

	By computing the quadratic obstructions, we obtain that $a_1=a_2=-1$, see Table \vref{table:3,4,-3} or \cite{ancil}. After applying the change of variables $t_4'=-t_4+a_3t_1t_3+a_5t_2t_5$ and $t_6'=-t_6+a_4t_2t_7$, we obtain 
	\[-t_1t_4- t_2t_6+ a_3t_1^2t_3+a_4t_2^2t_7+a_5t_1t_2t_5= t_1t_4'+t_2t_6'.\]
	Thus, the hull of $\Def_X$ is given by
	\[
	\KK[[t_1,t_2,t_3,t_4',t_5,t_6']]/\langle t_1t_4'+t_2t_6' \rangle.
	\]
	The spectrum of the hull is irreducible and has dimension 6; it is the formal completion of the product of $\Aff^3$ with the affine cone over a smooth quadric surface.
\end{example}

\begin{table}[htbp]
	\caption{Deformation data for the case $(e,a,b)=(3,-4,3)$}
	\label{table:3,4,-3}
	\small
	\begin{tabular}{|c|c|c|c|c|c|c|c|c|c|}
		\hline
		&
		\multicolumn{3}{c|}{\makecell{Deformation data}}& \multicolumn{4}{c|}{\makecell{Obstruction data}}& {\makecell{$\gamma$}}\\
		\hline
		\diagbox{$t^{w}$}{$\sigma$} & $\sigma_2$ & $\sigma_3$ & $\sigma_4$ & $\sigma_1\sigma_{2}$& $\sigma_2\sigma_{3}$ & $\sigma_3\sigma_{4}$ & $\sigma_{4}\sigma_1$ &\\
		\hline
		$t_1\cdot f_5$   & 1 & 1 & 0 & \zero & 0&\zero &0&0 \\
		\hline
		$t_2\cdot f_5$   & 1 & 1 & 0 & \zero & 0&\zero&0&0 \\
		\hline
		$t_3\cdot f_6$   & 0 & 1 & 1 & 0 & \zero&0&\zero&0 \\
		\hline
		$t_4\cdot f_2$   & 0 & 1 & 1 & 0 & \zero&0&\zero&0 \\
		\hline
		$t_5\cdot f_6$   & 0 & 1 & 1 & 0 & \zero&0&\zero&0 \\
		\hline
		$t_6\cdot f_2$   & 0 & 1 & 1 & 0 & \zero&0&\zero&0 \\
		\hline
		$t_7\cdot f_6$   & 0 & 1 & 1 & 0 & \zero&0 &\zero&0 \\
		\hline
		
		%second-order corrections
		\multicolumn{9}{|c|}{\textbf{2nd order}}\\
		\hline
		
		$t_1t_3\cdot (f_5-f_6)$  & 0 & -1/2& -1 &  0&1/2&1/2&\zero&0 \\
		\hline
		\color{blue}{$t_1t_4\cdot f_5$}   & 0 & 1/2 & 0 & 0&-1/2&-1/2&0&-1 \\
		\hline
		$t_1t_5\cdot (f_5-f_6)$ & 0 & -1/2& -1 &  0&1/2&1/2&\zero&0 \\
		\hline
		$t_2t_5\cdot (f_5-f_6)$ & 0 & -1/2& -1 &  0&1/2&0&\zero&0 \\
		\hline
		\color{blue}{$t_2t_6\cdot f_5$} & 0 & 1/2& 0 &  0&-1/2&-1/2&0&-1 \\
		\hline
		$t_2t_7\cdot (f_5-f_6)$ & 0 & -1/2& -1 &  0&1/2&1/2&\zero&0 \\
		\hline	
	\end{tabular}
\end{table}

\begin{example}[Case $e\geq2, a=-e, b=3$]\label{example:e,-e,b}
	Here, we compute the hull of $\Def_X$ for 
	\[X=\PP(\mathcal{O}_{\mathbb{F}_e}\oplus \mathcal{O}_{\mathbb{F}_e}(-eF+3H)),\]
	where $e\geq 2$. We will show that
	the spectrum of the hull consists of two irreducible components, with an arbitrarily large difference in their dimensions. The case $e=2$ was first analyzed in \cite[Example 5.2]{FPR} and has minimal $\rT^1$ dimension among obstructed $\PP^1$-bundles.
	Similar to Example \ref{example:3,-4,3}, the hull is already determined by the quadratic obstructions, which could have been computed using the methods of \cite{ilten3}.  

	By \thref{lemma:H1deg}, $H^1(X_{\Sigma},T_{X_{\Sigma}})$ has dimension $2e-1$ and is non-zero in the degrees
	\[\bu_1=(0,-1,-1),\quad \bu_{2k}=(-k,-1,0), \quad \bu_{2k+1}=(-k,0,1)\]
	for $k=1,\ldots,e-1$. By \thref{lemma:H2deg}, $H^2(X_{\Sigma},T_{X_{\Sigma}})$ has dimension $e-1$  and is non-zero in the degrees 
	\[\bv_k=(-k,-2,-1)\] 
	for $k=1,\ldots,e-1$. See Figure \vref{fig:e,-e,b}. The only non-negative integer combinations are of the form
	\[\bv_k= \bu_1+\bu_{2k}=2\bu_1+\bu_{2k+1},\]
	for $k=1,\ldots,e-1$. 
	
	Therefore, the hull must have the form
	\[\KK[[t_{1},\ldots,t_{2k+1}]]\Big/J; \quad J=\langle a_{1k}\cdot t_1t_{2k}+ a_{2k}\cdot t_1^2t_{2k+1}:k=1,\ldots,e-1 \rangle,\]
	for some $a_{1k},a_{2k}\in \KK$ where $k=1,\ldots,e-1$.
	
	The deformation data restricted to the deformation parameters $t_1,t_{2k}$ is shown in Table \vref{table:case: e,-e,3}. Therefore, we obtain that $a_{1k}=-1$ for $k=1,\ldots,e-1$. After applying the change of variables
	\[t_{2k}'= -t_{2k}+a_{2k}\cdot t_1 t_{2k+1},\]
	 we can express the hull as
	\[ \KK[[t_1,t_2',t_3, \ldots, t'_{2e-1} ]]\Big /\langle t_1 \rangle \cdot \langle t_2', t_4',\ldots, t'_{2e-2}\rangle.\]
	The spectrum of the hull has  two irreducible components, both smooth: one with dimension $2e-2$, and the other with dimension $e$.
	
\end{example}

\begin{table}
	\caption{Deformation data for the case $e\geq 2, a=-e$ and $b=3$}
	\label{table:case: e,-e,3}
	\begin{tabular}{|c|c|c|c|c|c|c|c|c|c|}
		\hline
		&
		\multicolumn{3}{c|}{\makecell{Deformation data}}& \multicolumn{4}{c|}{\makecell{Obstruction data}}&{\makecell{$\gamma$}}\\
		\hline
		\diagbox{$t^{w}$}{$\sigma$} & $\sigma_2$ & $\sigma_3$ & $\sigma_4$ & $\sigma_1\sigma_2$& $\sigma_2\sigma_3$ & $\sigma_3\sigma_4$ & $\sigma_4\sigma_1$ & \\
		\hline
		$t_1\cdot f_5$     & 1 & 1& 0  &  \zero&0&\zero &0&0 \\
		\hline
		$t_{2k}\cdot f_2$  & 0 & 1 & 1 & 0&\zero&0&\zero&0 \\
		\hline
		\color{blue}{$t_1t_{2k} \cdot f_5$}  & 0 & 1/2& 0 &  0&-1/2&-1/2&0&-1 \\
		\hline
	\end{tabular}
\end{table}

\begin{rem}
	We have carefully chosen the previous four examples so that degree constraints imply that the obstructions equations $g_{\ell}$ are polynomials instead of power series. Such constraints do not apply in general, such as in the case of $(e,a,b)=(2,-4,4)$.
\end{rem}

Using computations similar to those from the examples above, we now confirm that the cases in \thref{lemma:possibleobs} do indeed yield obstructions, proving \thref{thm:obstructed}.

\begin{proof}[Proof of \thref{thm:obstructed}]
	By \thref{lemma:possibleobs}, the cases not listed in the theorem are unobstructed.
It remains to show that the cases of the theorem are indeed obstructed, with minimal degree of obstruction as claimed.
	Suppose first that $e=1$, $a\leq -2$ and $b\geq 3-a$. By Figure \vref{fig:type1} and the discussion following the proof of \thref{lemma:possibleobs}, there is no degree of Type \Romannum{1}, so it is not possible to have relation among degrees that could provide a quadratic obstruction. However, we have
	\[(-1,a-2,-1)=(0,-a,-1)+(0,-2,-1)+(-1,0,1),\]
	 see \eqref{eq:relation1}. If we restrict the computations to the deformation parameters associated with the ray-degree pairs 
	 \[(\rho_5,(0,-a,-1)), \quad (\rho_5,(0,-2,-1)), \quad (\rho_6,(-1,0,1))\]
	  we obtain results identical to the deformation parameters associated with the
	  ray-degree pairs 
	  \[(\rho_5,(0,-2,-1)), \quad (\rho_5,(0,-2,-1)), \quad (\rho_6,(-1,0,1))\] in Example \ref{example:1,-2,5}. Indeed, the simplicial complexes $V_{\rho,\bu}$ and the rays involved are identical in these two situations.. Thus, we must have a third order obstruction, proving claim \ref{case:obstrcuted1}.
	  
	  Consider instead the case $e\geq 2$, $a\leq -e$ and $b \geq  1+{(2-a)}/{e}$. We now  use Figure \vref{fig:type2} and the discussion following the proof of \thref{lemma:possibleobs}. If  $a\not \equiv 1 \mod e$ we have the relation
	  \[(-1,\eta,-1)=(0,\xi,-1)+(0,-1,-1)\]
	  see \eqref{eq:relation2}. Restricting the computations to the deformation parameters associated with the ray-degree pairs 
	  \[(\rho_5,(0,\xi,-1)), \quad (\rho_2,(-1,-1,0))\]
	  we obtain results identical to the deformation parameters associated with the
	  ray-degree pairs 
	  \[(\rho_5,(0,-1,-1)), \quad (\rho_2,(-k,-1,0))\]
	   in Example \ref{example:e,-e,b}. Thus, we must have a second order obstruction.
	  
	 If instead $a\equiv 1 \mod e$, it is not possible to have relations among degrees that could provide a quadratic obstruction since $\xi-\eta=2$. However, we have
	 \[(-1,\eta,-1)=(1,\xi,-1)\; +\; 2\cdot (-1,-1,0),\]
	 see \eqref{eq:relation3}. Restricting the computations to the deformation parameters associated with the ray-degree pairs 
	 \[(\rho_5,(1,\xi,-1)), \quad (\rho_2,(-1,-1,0))\]
	 we obtain results identical to the deformation parameters associated with the
	 ray-degree pairs 
	 \[(\rho_5,(1,-1,-1)), \quad (\rho_2,(-1,-1,0))\]
	 in Example \ref{example:2,-3,4}. Thus, we must have a third order obstruction. This completes the proof of claim \ref{case:obstructed2}.
\end{proof}

\appendix
\numberwithin{equation}{section}

\section{Comparison theorem for open subschemes}\label{sec:comp}
Let $X$ be a scheme over $\KK$ and $U\subseteq X$ be an open subscheme. 
There is a natural map of functors
$\Def_X \to \Def_{U}$ obtained by restriction.
We make use of the following folklore result that was first brought to our attention by A.~Petracci:
\begin{thm}\thlabel{CMiso}
	Suppose that $X$ is a Noetherian separated scheme over $\KK$ and $U\subset X$ an open subscheme with complement $Z=X\setminus U$.
Then the restriction map
\[\Def_X\to \Def_U\]
is injective if 
the depth of $\CO_X$ at every (not necessarily closed) point of $Z$ is at least two, and an isomorphism if the depth of $\CO_X$ at every (not necessarily closed) point of $Z$ is at least three. In particular, it is an isomorphism if $Z$ has codimension at least three in $X$ and $X$ is Cohen-Macaulay.
\end{thm}
The above theorem is stated and proved in the affine case (with slightly different hypotheses) by M.~Artin in \cite[Proposition 9.2]{artin}. We now show how to globalize the argument of loc.~cit. 
In the following, $X$ will be a Noetherian separated scheme over $\KK$ and $U$ an open subscheme with complement $Z$.
\begin{lemma}\label{lemma:lc}
	Let $k\in \NN$ and assume that the depth of $\CO_X$ at every point of $Z$ is at least $k$. Then
	\[H^i(X,\CO_X)=H^i(U,\CO_X)\]
	for any $i<k-1$.
\end{lemma}
\begin{proof}
	By the depth condition, one obtains that $\mathcal{H}^i_Z(\CO_X)=0$ for $i<k$, see \cite[Theorem 3.8]{localcohomology}. Here, $\mathcal{H}_Z^i$ is the sheaf of local cohomology with support in $Z$. By \cite[Proposition 1.11]{localcohomology}, it follows that $H^i(X,\CO_X)=H^i(U,\CO_U)$ for $i<k-1$.
\end{proof}
\begin{lemma}\label{lemma:extend}
Let $X$ be affine and assume that the depth of $\CO_X$ at every point of $Z$ is at least $2$.
Let $X'$ be any deformation of $X$ over some $A\in\Art$. Then
\[
	H^0(X,\CO_{X'})=H^0(U,\CO_{X'}).
\]
\end{lemma}
\begin{proof}
	This is claim (*) from the proof of \cite[Lemma 9.1]{artin}, and follows from straightforward induction on the length of $A$. The base case $A=\KK$ follows from Lemma \ref{lemma:lc}.
\end{proof}

\begin{lemma}\label{lemma:flat}
	Let $X$ be affine and assume that the depth of $\CO_X$ at every point of $Z$ is at least $3$. Let $U'$ be any deformation of $U$ over some $A\in \Art$. Then $H^0(U,\CO_{U'})$ is flat over $A$.
\end{lemma}
\begin{proof}
	This is the contents of \cite[Lemma 9.2]{artin}, see the final sentence of the proof. For the reader's convenience, we summarize the argument here.
The proof is by induction on the length of $A$; the case $A=\KK$ is trivial.
By Lemma \ref{lemma:lc} we have that $H^i(X,\CO_X)=H^i(U,\CO_U)$ for $i=0,1$, in particular this vanishes for $i=1$ since $X$ is affine. 

	Realizing $A$ as a small extension
\[
0\to \KK \to A \to A_0\to 0
\]
the flatness of $U'$ over $A$ implies the exactness of 
\[
	0\to \CO_U \to \CO_{U'} \to \CO_{U'}\otimes_A A_0\to 0.
\]
By the isomorphisms of the previous paragraph, the exactness of 
\[
	0\to H^0(X,\CO_X) \to H^0(U,\CO_{U'}) \to H^0(U,\CO_{U'}\otimes_A A_0)\to 0
\]
follows.
The ring $H^0(U,\CO_{U'}\otimes_A A_0)$ is flat over $A_0$ by the induction hypothesis and the $A$-flatness of $H^0(U,\CO_{U'})$ follows by a version of the local criterion of flatness, see \cite[Proposition 8.1]{artin} or \cite[Proposition 2.2]{deftheory}.
\end{proof}

\begin{proof}[Proof of \thref{CMiso}]
We first show the injectivity of the map $\Def_X\to \Def_U$ when the depth of $\CO_X$ along $Z$ is at least two. 
	Fix an affine open cover $\mcU=\{U_i\}_{i\in I}$ of $X$.
Consider two deformations $X'$ and $X''$ of $X$ over $A\in\Art$, and denote their restrictions to $U$ by $U'$ and $U''$. Assume that there is an isomorphism of deformations $\phi:U'\to U''$. We thus have isomorphisms 
\[
	\phi_i^\#:H^0(U_i\cap U, \CO_{U''})\to H^0(U_i\cap U, \CO_{U'})
\]
satisfying the obvious cocycle condition.
By Lemma \ref{lemma:extend} 
\[
	H^0(U_i, \CO_{X'})=H^0(U_i\cap U, \CO_{U'}),\qquad
	H^0(U_i, \CO_{X''})=H^0(U_i\cap U, \CO_{U''})
\]
so we obtain isomorphisms $\phi_i:X'_{|U_i}\to X_{|U_i}''$. By the cocycle condition, these glue to give an isomorphism $X'\to X''$. This shows that 
$\Def_X\to \Def_U$ is injective.

For the surjectivity when the depth is at least three, consider any deformation $U'$ of $U$ over $A\in \Art$.
Let $\iota:U\to X$ denote the inclusion of $U$ in $X$, and set $\CO_{X'}:=\iota_*(\CO_{U'})$.
By Lemma \ref{lemma:flat}, $\CO_{X'}$ is flat over $A$, hence defines a deformation $X'$. By construction, this restricts to the deformation $U'$, hence $\Def_X\to \Def_U$ is surjective.\end{proof}

\section{Solving the deformation equation}\label{ap:solve}
In this appendix, we will prove \thref{prop:defeqsolving} and state and prove a lemma we used in proving \thref{hull}.
We use notation as established in \S\ref{sec:defeqsetup} and \S\ref{sec:deq}.
\begin{proof}[Proof of \thref{prop:defeqsolving}]
	To solve \eqref{eq:defeq2} we consider the small extension
	\[
		0\to J_{r}/(\mfm\cdot J_r) \to S/(\mfm\cdot  J_{r}) \to S/J_{r}\to 0.
	\]
	It follows from \eqref{eq:defeq} (modulo $\mfm\cdot J_{r-1}$) that
	\[ \lambda(\mfp(s({\alpha^{(r)}}))) \equiv 
		\lambda(\mfp(s({\alpha^{(r)}})))-\sum_{\ell=1}^q g_\ell^{(r)}\cdot \omega_\ell \equiv 
	0 \quad \mod J_r.\]

	From Lemma \ref{lemma:obstruction} we obtain that the image of \[\xi=\lambda(\mfp(s({\alpha^{(r)}})))-\sum_{\ell=1}^q g_\ell^{(r)}\cdot \omega_\ell\] in $\check{C}^{1}(\sU,\mcK/\mcL)\otimes S/(\mfm\cdot J_{r})$ is a cocycle. 
Because of this, the normal form of $\xi$ with respect to $\mfm\cdot J_r$ is also a cocycle. In fact, the normal form belongs to 
\[
	\check{Z}^{1}(\sU,\mcK/\mcL)\otimes \mfm_{r+1} 
	\]
since 
\[\xi\equiv 0 \mod \mfm \cdot J_{r-1}
\]
and $\mfm \cdot J_{r-1}=\mfm J_r+\mfm^{r+1}$. This latter equality follows from the assumption that $J_{r}+ \mfm^r= J_{r-1}+\mfm^r$. Since the images of the $\omega_\ell$ span $\check{H}^1(\sU,\mcK/\mcL)$, there exists
\[\beta^{(r+1)}\in \check{C}^0(\sU,\mcK/\mcL)\otimes \mfm_{r+1}, \quad \gamma_1^{(r+1)},\ldots,\gamma_q^{(r+1)}\in  \mfm_{r+1} \] satisfying \eqref{eq:defeq2}.
This implies claim \ref{prop:defeqsolve1}.

We now prove claim \ref{prop:defeqsolve2}.
By \eqref{eqn:odiff},
\[
	\lambda(\mfp(s(\alpha^{(r+1)})))\equiv \lambda(\mfp(s(\alpha^{(r)})))-d(\beta^{(r+1)})\qquad\mod \mfm^{r+2}.
\]
Equation \eqref{eq:defeq} then follows directly from \eqref{eq:defeq2}. Likewise, $J_{r+1}+ \mfm^{r+1}= J_{r}+ \mfm^{r+1}$ follows from the fact that $\gamma_\ell^{(r+1)}$ belongs to $\mfm_{r+1}$.
\end{proof}

Finally, with notation as in \S\ref{sec:defeqversal}, the following lemma is used in proving \thref{hull}:

\begin{lemma}\thlabel{lemma:obfinj}
		Consider the map of functors  $f: \Hom(R,-)\to \widehat \rF_{\mcL }$. 
		There exists a natural injective map $ob_f$ satisfying the hypotheses of \thref{standardsmooth} defined by
	 \begin{align*}
		ob_f: (J/\mfm J)^* &\to H^1(\sU,\mcK/\mcL)\\
		\varphi &\mapsto \sum_{\ell=1}^{q} \varphi(\overline g_{\ell})\cdot \omega_{\ell},
	\end{align*}
	where $\overline g_{\ell}$ denotes the image of $g_{\ell}$ in $J/\mfm\cdot J$.
\end{lemma}
\begin{proof}
	By construction, $J\subseteq \mfm^2$.
	It is well-known that $(J/\mfm J)^*$ is an obstruction space for $\Hom(R,-)$ (see \cite[Example 3.6.9]{Lie}). An obstruction space for $\widehat{\rF}_{\mcL }$ is given by $\check{H}^1(\sU,\mcK/\mcL)$ (see Lemma \ref{lemma:obstruction}). It is immediate that $ob_f$ is injective since the $\omega_\ell$ form a basis for $H^1(\sU,\mcK/\mcL)$ and the $g_\ell$ generate $J$. We claim that $ob_f$ satisfies the hypothesis of 
\thref{standardsmooth}.
	
	To prove this, we first claim that we can choose $n\gg 0$ such that 
	\begin{equation}\label{Jisomorphism}
		(J+ \mfm^n)/(\mfm\cdot J+ \mfm^n )\cong J/\mfm\cdot J.
	\end{equation}
	We will show this using the ideas from the proof of \cite[Theorem 11.1]{deftheory}.
	According to the Artin-Rees lemma \cite[Corollary 10.10]{commut}, we have $J\cap \mfm^n \subseteq \mfm\cdot J$ for $n\gg0$. 
	This implies that
	\[ (\mfm\cdot J+\mfm^n)\cap J=\mfm\cdot J,\]
	which leads to the isomorphism
	\[(J+ \mfm^n)/(\mfm\cdot J+ \mfm^n ) \cong J/(\mfm \cdot J +\mfm^n)\cap J = J/\mfm\cdot J\]
	as desired.

	Let $\zeta \in \Hom(R,A)$ and consider a small extension as in \eqref{eqn:smallextension}.
	We will show that 
	\[ ob_f\Big(\phi(\zeta, A')\Big)= \phi(f(\zeta),A')\]
	where we use $\phi$ to denote the map taking a small extension to its obstruction class for both functors $\Hom(R,-)$ and $\widehat{\rF}_{\mcL}$.
	
	We define a local morphism of $\KK$-algebras $\eta: S\to A'$ by mapping each variable $t_{\ell}$ to any lifting of its image under the map $S\to R\to A$. Because $A,A'$ are Artinian rings and by the discussion above, there exists $r\gg0$ such that $\eta,\zeta$ factor respectively through $S/\mfm^{r+1}$ and $R_r$, and 
	\[ (J+ \mfm^{r+1})/(\mfm\cdot J+ \mfm^{r+1} ) \cong J/\mfm\cdot J.\]

	The image of  $\mfm\cdot J$ under $\eta$ is zero, because $\mfm_{A'}\cdot I=0$.	Consequently, $\eta$ factors through $S/\mfm\cdot J$.
	Furthermore, since $g_{\ell}-g_{\ell}^{(r)}\in \mfm^{r+1}$, it follows that $J+\mfm^{r+1}= J_r$ and $\mfm\cdot J +\mfm^{r+1}= \mfm\cdot J_r +\mfm^{r+1}$. 
	From the above discussion, we thus have the following commutative diagram:
	\[
	\begin{tikzcd}[column sep=small]
		0\ar[r]& J \ar[r] \ar[d]& S \ar[r] \ar[d] & R \ar[d,"\id"] \ar[r] & 0  \\
		0\ar[r] & J/\mfm J \ar[r] \ar[d, "\cong"]& S/(\mfm \cdot J)\ar[r] \ar[d]& R\ar[r] \ar[d,"\pi_r"]& 0 \\	
		0\ar[r] & J_r/(\mfm\cdot J_r+\mfm^{r+1}) \ar[r] \ar[d, "\overline\eta_r"]& S/(\mfm \cdot J_r+\mfm^{r+1})\ar[r] \ar[d,"\eta_r"]& R_r\ar[r] \ar[d,"\zeta_r"]& 0 \\
		0\ar[r]& I \ar[r] & A' \ar[r] \ & A  \ar[r] & 0  
	\end{tikzcd}.
	\]
	Since $J\subseteq \mfm^2$, it follows that  \[\overline \eta: J/\mfm\cdot J \xrightarrow{\cong} J_r/(\mfm\cdot J_r+\mfm^{r+1}) \xrightarrow{\overline \eta_r} I\] remains unaffected by the choice of $\eta$.

	Applying obstruction theory to $\zeta\in \Hom(R,A)$ (see for example, \cite[Example 3.6.9]{Lie}), we obtain
	\[ \phi(\zeta, A')= \overline\eta.\]
	By \eqref{eq:defeq}
	we have 
	\begin{align*}
		\phi\Big(\alpha^{(r)},S/(\mfm\cdot J_r+\mfm^{r+1})\Big) &= \sum_{\ell=1}^{q} \overline g_{\ell}^{(r)} \cdot \omega_{\ell}
	\end{align*}
	where $\overline g_\ell^{(r)}$ is the image of $g_{\ell}^{(r)}$ in $J_r/(\mfm \cdot J_r+\mfm^{r+1})$.
	By the functoriality of obstruction theory, we have
	\begin{align*}
		\phi\Big(f(\zeta),A'\Big)&= \sum_{\ell=1}^{q} \overline\eta_r(\overline g_{\ell}^{(r)}) \cdot \omega_{\ell}\\
		&= \sum_{\ell=1}^{q} \overline\eta(\overline g_{\ell}) \cdot \omega_{\ell}\\
		&= ob_f(\overline\eta)
	\end{align*}
as desired.
\end{proof}

\bibliographystyle{alpha}
\bibliography{v3paper} 	

\end{document}